\documentclass[11pt]{article}
\usepackage{mathtools}
\usepackage[T1]{fontenc}
\usepackage{amsfonts}
\usepackage{amsmath}
\usepackage{amssymb}
\usepackage{amsthm}
\usepackage{bbm}
\usepackage{bm}
\usepackage{xcolor}
\usepackage{mathrsfs}
\usepackage{enumitem}
\usepackage{needspace}
\usepackage[colorlinks=true, linkcolor=blue, citecolor=blue, urlcolor=blue]{hyperref}
\mathtoolsset{showonlyrefs}
\usepackage[
  backend=biber,
  style=numeric,
  sorting=nyt,
  giveninits=true,
  maxbibnames=99,
  doi=true,
  isbn=false,
  url=false,
  eprint=true
]{biblatex}

\DeclareNameAlias{author}{family-given}
\DeclareNameAlias{editor}{family-given}
\DeclareFieldFormat[article]{title}{#1}
\DeclareFieldFormat[incollection]{title}{#1}
\DeclareFieldFormat[inproceedings]{title}{#1}
\DeclareFieldFormat[book]{title}{\mkbibemph{#1}}
\renewbibmacro{in:}{}

\AtEveryBibitem{%
  \clearfield{eprintclass}%
  \clearfield{primaryclass}%
}
\providebibmacro*{eprint:arxiv}{}
\renewbibmacro*{eprint:arxiv}{%
  \printtext[eprint:arxiv]{%
    \printfield{eprinttype}%
    \setunit{\addcolon}%
    \printfield{eprint}}}

\DeclareMathOperator{\spt}{spt}
\DeclareMathOperator{\supp}{supp}
\DeclareMathOperator{\dist}{dist}
\DeclareMathOperator{\Lip}{Lip}

\DeclareMathOperator{\diam}{diam}

\newcommand{\RR}{\mathbb{R}}
\newcommand{\R}{\RR}

\newcommand{\eps}{\varepsilon}

\newcommand{\Id}{\mathbbm{I}}

\newcommand{\1}{\mathbf{1}}

\usepackage[top=29mm,bottom=29mm,left=29mm,right=29mm]{geometry}
\newcommand{\mykill}[1]{}
\numberwithin{equation}{section}
\usepackage[capitalize, noabbrev]{cleveref}
\makeatletter
\AtBeginDocument{%
  \let\label@in@display\cref@old@label@in@display
  \let\ltx@label\cref@old@label
}
\makeatother %
\crefname{equation}{}{} %

\theoremstyle{plain}
\newtheorem{theorem}{Theorem}[section]
\newtheorem{proposition}[theorem]{Proposition}
\newtheorem{lemma}[theorem]{Lemma}
\newtheorem{corollary}[theorem]{Corollary}
\theoremstyle{definition}

\newtheorem{remark}[theorem]{Remark}

\newtheorem{assumption}[theorem]{Assumption}
\crefname{assumption}{Assumption}{Assumptions}
\Crefname{assumption}{Assumption}{Assumptions}
\theoremstyle{remark}

\AddToHook{env/proposition/begin}{\crefalias{theorem}{proposition}}
\AddToHook{env/lemma/begin}{\crefalias{theorem}{lemma}}
\AddToHook{env/corollary/begin}{\crefalias{theorem}{corollary}}
\AddToHook{env/definition/begin}{\crefalias{theorem}{definition}}
\AddToHook{env/remark/begin}{\crefalias{theorem}{remark}}
\AddToHook{env/example/begin}{\crefalias{theorem}{example}}
\AddToHook{env/assumption/begin}{\crefalias{theorem}{assumption}}
\crefname{theorem}{Theorem}{Theorems}
\crefname{proposition}{Proposition}{Propositions}
\crefname{lemma}{Lemma}{Lemmas}
\crefname{corollary}{Corollary}{Corollaries}
\crefname{definition}{Definition}{Definitions}
\crefname{remark}{Remark}{Remarks}
\crefname{example}{Example}{Examples}
\crefname{assumption}{Assumption}{Assumptions}
\usepackage{tocloft}

\newcommand{\dd}{\,\mathrm d}
\newcommand{\intr}{\operatorname{int}}
\DeclareMathOperator{\id}{id}

\newcommand{\esssup}{\operatorname*{ess\,sup}}

\newcommand{\N}{\mathbb N}
\newcommand{\B}{\overline B}
\newcommand{\Hh}{\mathbb H}
\newcommand{\Sn}{\mathbb S^{d-1}}

\DeclareMathOperator{\TV}{TV}

\allowdisplaybreaks[2]
\hypersetup{
 pdftitle={Geometry and Convergence of Quadratically Regularized Optimal Transport II},
 pdfsubject={Quadratically regularized optimal transport, Brenier graph inclusion, potential expansions, and bias asymptotics},
 pdfkeywords={optimal transport, regularization, Brenier graph, boundary layers, bias asymptotics},
 pdfauthor={Alberto Gonzalez-Sanz, Marcel Nutz},
 pdflang={en},
 pdfdisplaydoctitle=true
}
\title{\vspace{-1em}Geometry and Convergence of Quadratically Regularized Optimal Transport II}
\author{Alberto Gonz{\'a}lez-Sanz\thanks{Department of Statistics, Columbia University, ag4855@columbia.edu.} \and Marcel Nutz\thanks{Departments of Mathematics and Statistics, Columbia University, mnutz@columbia.edu.}}
\date{\today}
\begin{document}
\setlength{\abovedisplayskip}{10pt plus 2.5pt minus 4.5pt}
\setlength{\belowdisplayskip}{10pt plus 2.5pt minus 4.5pt}
\setlength{\abovedisplayshortskip}{0pt plus 2.5pt}
\setlength{\belowdisplayshortskip}{6pt plus 2.5pt minus 2.5pt}
\maketitle
\vspace{-1em}
\begin{abstract} 
We study quadratically regularized optimal transport with quadratic cost in the regime of small regularization $\varepsilon$, where the support of the optimal coupling is sparse. For smooth marginal densities in $\mathbb{R}^d$, we show that the support contains the graph of the Brenier map. After centering by the Brenier image and rescaling by $\ell=\varepsilon^{1/(d+2)}$, the sections of the support have explicit ellipsoidal limits for $\varepsilon\to0$, except at boundary points, where the limits are half-space profiles. The optimal dual potentials admit expansions at order~$\ell^2$, uniformly up to the boundary. We identify the leading coefficients, which consist of a common local profile and opposite global corrections determined by a linear Neumann problem. Finally, we analyze two approximations to the Brenier map, namely the gradient of the dual potential and the conditional mean of the coupling. We obtain the sharp $L^p$ rate $\ell^{1+1/p}$ with exact leading constants and further identify the leading interior and boundary biases. Taken together, our results illustrate that quadratic regularization induces an accurate sparse approximation of classical optimal transport. 
\end{abstract}
\medskip
\setcounter{tocdepth}{2}
\tableofcontents
\medskip

\Needspace{12\baselineskip}
\section{Introduction}\label{sec:main}
Optimal transport provides a flexible framework for comparing probability measures and has become a standard tool in statistics, machine learning, image processing, and other applications. In many of these areas, regularizing transport problems is essential for computational tractability and improved statistical properties. The present work studies quadratically regularized optimal transport (QOT) with quadratic cost. Let $\mu,\nu$ be probability measures with compact supports $X,Y\subset\R^d$, and let $\Pi(\mu,\nu)$ denote their couplings. For $\eps>0$, consider
\begin{equation}\label{eq:problem}
\inf_{\pi\in\Pi(\mu,\nu)}\left\{\frac12\int_{X\times Y}|x-y|^2\dd\pi+\frac\eps2\int_{X\times Y}\left(\frac{\dd\pi}{\dd(\mu\otimes\nu)}\right)^2\dd(\mu\otimes\nu)\right\},
\end{equation}
where $|\cdot|$ is the Euclidean norm and the objective is $+\infty$ if $\pi\not\ll\mu\otimes\nu$. We denote the unique minimizer by $\pi_\eps$.  The quadratic penalty makes the problem strictly convex while allowing sparse support $\spt\pi_\eps$, in contrast to the full support of entropically regularized couplings \cite{BSR,Nutz}. This sparsity immediately raises a fidelity question: does the support still contain the pairs $(x,y)$ matched by unregularized optimal transport?

For smooth positive densities on uniformly convex supports, unregularized optimal transport is given by the Brenier map $T:X\to Y$. Let $S_{\eps,x}$ be the section of $\spt\pi_\eps$ above~$x$; intuitively, the range of images of $x$ under $\pi_\eps$. The recent localization estimates in Part I (see \cite{GNGeometry}) show that this section lies within distance $O(\ell)$ of the unregularized image $T(x)$, where $\ell=\eps^{1/(d+2)}$. However, before the present work, it was a longstanding open question whether $T(x)$ actually belongs to the section. We prove that it does, for all sufficiently small $\eps$: for any marginals $\mu,\nu$ satisfying \cref{ass:data}, there is $\eps_0>0$ such that
\[
 T(x)\in S_{\eps,x}\qquad\text{for all }x\in X\text{ and }0<\eps<\eps_0.
\]
In fact, \cref{thm:main-geometry} shows that the convex set $S_{\eps,x}$ contains a relative neighborhood of $T(x)$ of radius comparable to $\ell$, the same order as the diameter of $S_{\eps,x}$. This is the first indication that $S_{\eps,x}$ is approximately centered at $T(x)$ for small $\eps$. From an applied point of view, it suggests that quadratic regularization induces sparsity without sacrificing fidelity.

To further validate this claim, we identify the limiting shapes (\cref{thm:main-geometry}). After subtracting $T(x)$ and rescaling distances by $\ell$, sections for interior $x$ converge to explicit ellipsoids centered at the origin. Under the same centering and rescaling, the conditional target laws $\pi_\eps^x$ converge to laws whose densities are (positive parts of) explicit quadratic functions. 

The density $h_\eps$ of the coupling $\pi_\eps$ and its support are described in terms of convex dual potentials $u_\eps:X\to\R$ and $v_\eps:Y\to\R$ via~\eqref{eq:slack}. We therefore investigate the expansion of the potentials around the limiting Brenier potentials (\cref{thm:corrector}). The leading corrections for the two potentials consist of a common local profile and opposite global terms determined by a linear Neumann problem. While expansions are shown to hold uniformly (including at the boundary), it is noteworthy that local geometry alone does not determine the individual corrections.

Finally, we return to the original goal of approximating unregularized optimal transport. The regularized optimizer gives rise to two  approximations for the Brenier map: the gradient of the potential and the conditional mean of the coupling. We show in \cref{thm:main-bias} that both have an approximation error of order $\ell^{1+1/p}$ in $L^p(\mu)$. We further identify their leading interior biases at scale $\ell^2$ and their boundary profiles at scale $\ell$, while estimates connecting the interior and boundary regimes yield the exact leading constants for the $L^p(\mu)$ errors. %

Beyond these results for quadratic regularization, our approach can substantially advance the state of the art on entropic regularization, where only partial results in the direction of \cref{thm:corrector,thm:main-bias} are available. See \cref{se:previousResults} for a comparison with related literature.

\subsection{Assumptions and main results}\label{sec:main-results}

We impose the following assumptions on the marginals $\mu,\nu$ and their supports $X,Y$.

\begin{assumption}\label{ass:data}
Let $0<\alpha<1$. The sets $X,Y\subset\R^d$ are closures of bounded uniformly convex domains with $C^{3,\alpha}$ boundaries. The probability measures $\mu=\rho_0\1_X\dd x$ and $\nu=\rho_1\1_Y\dd y$ have densities $\rho_0\in C^{1,\alpha}(X)$ and $\rho_1\in C^{1,\alpha}(Y)$ that are uniformly bounded above and away from zero.
\end{assumption}

For $d=1$, the sets $X,Y$ are simply nondegenerate compact intervals. The optimizer $\pi_\eps$ of \eqref{eq:problem} admits convex potentials $u_\eps:X\to\R$ and $v_\eps:Y\to\R$. More precisely, $\dd\pi_\eps=h_\eps\dd(\mu\otimes\nu)$, where
\begin{equation}\label{eq:slack}
 h_\eps(x,y)=\frac{(q_\eps(x,y))_+}{\eps},\qquad q_\eps(x,y)=x\cdot y-u_\eps(x)-v_\eps(y).
\end{equation}
Here $a_+=\max\{a,0\}$, and we call $q_\eps$ the slack. The potentials are $C^{1,1}$ up to the boundary; see \cref{prop:inputs}.

Let $u_0,v_0$ be the convex potentials for classical (unregularized) transport, with Brenier map $T=\nabla u_0$ and inverse $S=\nabla v_0=T^{-1}$. 
The unregularized potentials are $C^{3,\beta}$ up to the boundary for some $\beta\in(0,\alpha)$. Set $H=DT=D^2u_0$, which is uniformly positive definite. 

Each pair of potentials is determined up to opposite additive constants. We fix them by
\begin{equation}\label{eq:gauge}
 \int_Xu_\eps\dd\mu=\int_Yv_\eps\dd\nu,\qquad \int_Xu_0\dd\mu=\int_Yv_0\dd\nu.
\end{equation}
Throughout, we use the length scale $\ell=\eps^{1/(d+2)}$, and limits are taken as $\eps\downarrow0$ unless stated otherwise. %

\paragraph{Graph inclusion and section profiles.}
For $x\in X$, define
\begin{equation}\label{eq:forward-sections}
 S_{\eps,x}=\{y\in Y:q_\eps(x,y)\ge0\},\qquad \pi_\eps^x(\dd y)=h_\eps(x,y)\nu(\dd y).
\end{equation}
By \cref{prop:inputs}, $\spt\pi_\eps=\{q_\eps\ge0\}$, so $S_{\eps,x}$ is indeed a section of the support, and the conditional law $\pi_\eps^x$ is a probability measure for every $x\in X$. To describe the limiting sections and conditional laws, let $\omega_d$ be the volume of the Euclidean unit ball and set
\begin{equation}\label{eq:scale}
\begin{gathered}
 h_d=\left(\frac{d+2}{2^{1+d/2}\omega_d}\right)^{2/(d+2)},\qquad
 \sigma(x)=\left(\frac{\sqrt{\det H(x)}}{\rho_0(x)}\right)^{1/(d+2)},\\[.2em]
 h(x)=h_d\sigma(x)^2=h_d[\rho_0(x)\rho_1(T(x))]^{-1/(d+2)}.
\end{gathered}
\end{equation}
The second expression for $h$ follows from the Monge--Amp\`ere equation
\begin{equation}\label{eq:transport-jacobian}
 \rho_1(T(x))\det H(x)=\rho_0(x)\qquad(x\in X).
\end{equation}
Define
\begin{equation}\label{eq:main-ellipsoid-density}
 \mathscr E_x=\{w\in\R^d:w^TH(x)^{-1}w\le2h(x)\},\qquad k_x(w)=\rho_1(T(x))\left(h(x)-\tfrac12w^TH(x)^{-1}w\right)_+.
\end{equation}
Then $\mathscr E_x$ is an ellipsoid centered at the origin, and by \eqref{eq:scale}, $k_x$ is a probability density supported on $\mathscr E_x$. Let $\lambda_{\eps,x}$ be the law of $(y-T(x))/\ell$ under $\pi_\eps^x$, describing how the regularized coupling distributes mass around the unregularized image $T(x)$. We write $d_H$ for Hausdorff distance between nonempty compact sets and $\|\lambda-\lambda'\|_{\TV}=\sup_A|\lambda(A)-\lambda'(A)|$ for total variation distance.

\begin{theorem}[Brenier graph inclusion and section profiles]\label{thm:main-geometry}
Under \cref{ass:data}, there are $c,r,R,\eps_0>0$ such that, for $0<\eps<\eps_0$ and every $x\in X$,
\begin{equation}\label{eq:main-graph-inclusion}
 q_\eps(x,T(x))\ge c\ell^2,\qquad Y\cap B(T(x),r\ell)\subset S_{\eps,x}\subset Y\cap\B(T(x),R\ell).
\end{equation}
In particular, $\operatorname{graph}T\subset\{q_\eps>0\}\subset\spt\pi_\eps$.

For every nonempty compact $K\subset\intr X$,
\begin{equation}\label{eq:main-interior-profiles}
 \sup_{x\in K}\left[d_H\bigl(\ell^{-1}(S_{\eps,x}-T(x)),\mathscr E_x\bigr)+\|\lambda_{\eps,x}-k_x(w)\dd w\|_{\TV}\right]\longrightarrow0.
\end{equation}
\end{theorem}

The restriction to sufficiently small $\eps$ is essential: for every $d\ge2$, the inclusion $\operatorname{graph}T\subset\spt\pi_\eps$ can fail at a fixed positive parameter; see \cref{prop:graph-exclusion}.

The inclusions \eqref{eq:main-graph-inclusion} hold up to the boundary, and this clearly necessitates the intersection of the (inner) ball with~$Y$. On the other hand, \eqref{eq:main-interior-profiles} does not hold up to the boundary. The ellipsoidal approximation remains valid along $x=x_\eps$ whenever $\dist(x_\eps,\partial X)/\ell\to\infty$, but within the $O(\ell)$ boundary layer, the sections and conditional laws are instead approximated by half-space profiles which are detailed in \cref{thm:section-laws}. 

The section profiles also determine the limiting curvature of the potentials. Using the formulas that express the Hessians as integrals over the section boundaries, \cref{prop:interior-hessian} gives $D^2u_\eps\to D^2u_0$ in $L^\infty_{\mathrm{loc}}(\intr X)$ and $D^2v_\eps\to D^2v_0$ in $L^\infty_{\mathrm{loc}}(\intr Y)$. Together with the uniform Hessian bounds, this yields, for every $1\le p<\infty$,
\begin{equation}\label{eq:main-Sobolev}
 u_\eps\longrightarrow u_0\quad\text{in }W^{2,p}(X),\qquad v_\eps\longrightarrow v_0\quad\text{in }W^{2,p}(Y).
\end{equation}
Here and below, $L^p$ norms are taken with respect to Lebesgue measure when no measure is specified, and Sobolev spaces on closed domains are understood on their interiors.

\paragraph{Expansion of the potentials.} We show that the errors $u_\eps-u_0$ and $v_\eps-v_0$ are of order $\ell^2$ and identify the corrections at this scale, uniformly up to the boundary. To that end, we introduce a half-space profile that will describe the boundary contribution. Let $e_d$ be the last coordinate vector and $\Hh=\{\zeta\in\R^d:\zeta_d\ge0\}$. There is a unique bounded continuous function $\varphi_d:[0,\infty)\to\R$ satisfying
\begin{equation}\label{eq:scalar-halfspace-profile}
 \int_\Hh\left(\varphi_d(t)+\varphi_d(\zeta_d)-\tfrac12|te_d-\zeta|^2\right)_+\dd\zeta=1\qquad(t\ge0).
\end{equation}
The profile $\varphi_d$ is positive and continuously differentiable, with $\varphi_d(t)\to h_d/2$ as $t\to\infty$; see \cref{lem:profile-regularity,prop:profile-exponential}. We set $\varphi_d(\infty)=h_d/2$.

Let $\delta(x)=\dist(x,\partial X)$. Fix $\delta_*>0$ so that every point in $\mathcal U=\{x\in X:\delta(x)<\delta_*\}$ has a unique nearest point on $\partial X$, and let $n(x)$ be the inward unit normal at that point. Define
\begin{equation}\label{eq:boundary-coordinates}
 \gamma(x)=\sqrt{n(x)^TH(x)^{-1}n(x)}\quad(x\in\mathcal U),\qquad
 t_\eps(x)=\begin{cases}\delta(x)/(\ell\sigma(x)\gamma(x)),&x\in\mathcal U,\\ \infty,&x\notin\mathcal U,\end{cases}
\end{equation}
so that $t_\eps(x)$ is the rescaled inward distance from the boundary. For $x\notin\mathcal U$, we extend by setting $n(x)=e_d$ and $\gamma(x)=(e_d^TH(x)^{-1}e_d)^{1/2}$. Write
\begin{equation}\label{eq:potential-errors}
 r_\eps=u_\eps-u_0,\qquad s_\eps=v_\eps-v_0,\qquad \mathfrak p_\eps(x)=\sigma(x)^2\varphi_d(t_\eps(x)).
\end{equation}
Outside $\mathcal U$, we have $\mathfrak p_\eps=h/2$ by definition. On every compact subset of $\intr X$, the rescaled depth $t_\eps$ tends uniformly to infinity, so \eqref{eq:profile-deep-limit} gives $\mathfrak p_\eps\to h/2$ locally uniformly there.

The following statement compares $r_\eps$ and $s_\eps\circ T$, so that both quantities are defined on~$X$. It identifies the boundary contribution $\mathfrak p_\eps$ as a common correction to both errors, while the opposite global corrections $\pm g_0$ are determined by the mean-zero solution of a linear Neumann problem. We call $g_0$ the Neumann corrector.

\begin{theorem}[Potential expansions]\label{thm:corrector}
Under \cref{ass:data}, there is a unique $g_0\in H^1(X)$ with $\int_Xg_0\dd\mu=0$ satisfying
\begin{equation}\label{eq:corrector-Neumann}
 \int_X\rho_0 H^{-1}\left(\nabla g_0-\frac{h}{2(d+4)}\nabla\log\det H\right)\cdot\nabla\psi\dd x=0\qquad(\psi\in H^1(X)).
\end{equation}
This solution belongs to $C^{1,\beta}(X)$ for some $\beta\in(0,\alpha)$. Under the normalization \eqref{eq:gauge}, the potential errors~\eqref{eq:potential-errors} satisfy
\begin{equation}\label{eq:corrector-uniform}
\begin{aligned}
 r_\eps(x)&=\ell^2[g_0(x)-\mathfrak p_\eps(x)]+o(\ell^2),\\
 s_\eps(T(x))&=\ell^2[-g_0(x)-\mathfrak p_\eps(x)]+o(\ell^2)
\end{aligned}
\end{equation}
for all $x\in X$, where the remainders are uniform in $x$. Consequently, 
\begin{equation}\label{eq:global-bulk-bias}
 \ell^{-2}r_\eps\longrightarrow g_0-h/2,\qquad \ell^{-2}(s_\eps\circ T)\longrightarrow-g_0-h/2
\end{equation}
locally uniformly on $\intr X$ and also in $L^p(\mu)$ for every $1\le p<\infty$.
\end{theorem}

Equation~\eqref{eq:corrector-Neumann} is a weak Neumann problem: the flux $\rho_0 H^{-1}\bigl(\nabla g_0-\tfrac{h}{2(d+4)}\nabla\log\det H\bigr)$ has zero divergence in $\intr X$ and zero normal component on $\partial X$. The expansions \eqref{eq:corrector-uniform} separate the common local term $-\mathfrak p_\eps$ from the opposite global corrections $\pm g_0$. It follows that $r_\eps$ and $s_\eps\circ T$ agree up to $o(\ell^2)$ if and only if $g_0=0$, or equivalently if $\det DT$ is constant; see \cref{prop:equal}. %

\paragraph{Bias of the transport map approximations.}
We consider two approximations to the Brenier map: the potential gradient $z_\eps(x)$ and the conditional mean $m_\eps(x)$,
\begin{equation}\label{eq:forward-approximations}
 z_\eps(x)=\nabla u_\eps(x),\qquad m_\eps(x)=\int_Y y\,\pi_\eps^x(\dd y).
\end{equation}
While $m_\eps(x)$ is the average of $y$ under $\pi_\eps^x$, one can express $z_\eps(x)$ as the average of $y$ under the normalized restriction $\nu(\cdot\cap S_{\eps,x})/\nu(S_{\eps,x})$ of the marginal; cf.\ \eqref{eq:barycenters}. We show that both approximate $T(x)$ at the same rate, and we identify the leading terms (which differ).

Some notation is necessary to state those results. The fields
\begin{equation}\label{eq:biases}
 b_z=\nabla g_0-\tfrac12\nabla h,\qquad b_m=\nabla g_0-\tfrac{d}{2(d+4)}\nabla h
\end{equation}
describe the leading errors in the interior, as stated in \eqref{eq:main-local-bias} below. In the boundary layer, the leading terms occur at scale $\ell$ and are parallel to $n(x)$. After removing the common factor $\ell\sigma(x)/\gamma(x)$, their signed amplitudes are $-\varphi_d'(t_\eps(x))$ and $\beta_d(t_\eps(x))$, respectively, where $\beta_d$ is the mean normal displacement in the half-space model:
\begin{equation}\label{eq:halfspace-mean-profile}
 \beta_d(t)=\int_{\{w_d\ge-t\}}w_d\left(\varphi_d(t)+\varphi_d(t+w_d)-\tfrac12|w|^2\right)_+\dd w\qquad(t\ge0).
\end{equation}
See also \cref{cor:boundary-gradients}.
For $1\le p<\infty$, define
\begin{equation}\label{eq:main-bias-constants}
 A_{d,p}=\int_0^\infty|\varphi_d'(t)|^p\dd t,\qquad B_{d,p}=\int_0^\infty|\beta_d(t)|^p\dd t.
\end{equation}
Both constants are finite and strictly positive by \cref{lem:profile-regularity,prop:profile-exponential}. We write $\mathcal H^{d-1}$ for $(d-1)$-dimensional Hausdorff measure (counting measure when $d=1$).

\begin{theorem}[Local and global bias asymptotics]\label{thm:main-bias}
Under \cref{ass:data}, the potential gradient $z_\eps(x)$ and the conditional mean $m_\eps(x)$ satisfy, for every compact $K\subset\intr X$,
\begin{equation}\label{eq:main-local-bias}
 \ell^{-2}(z_\eps-T)\longrightarrow b_z,\qquad \ell^{-2}(m_\eps-T)\longrightarrow b_m\qquad\text{strongly in }L^2(K).
\end{equation}
For every $1<p<\infty$,
\begin{equation}\label{eq:main-Lp}
\begin{aligned}
 \ell^{-(p+1)}\|z_\eps-T\|_{L^p(\mu)}^p
 &\longrightarrow A_{d,p}\int_{\partial X}\rho_0\sigma^{p+1}\gamma^{1-p}\dd\mathcal H^{d-1},\\
 \ell^{-(p+1)}\|m_\eps-T\|_{L^p(\mu)}^p
 &\longrightarrow B_{d,p}\int_{\partial X}\rho_0\sigma^{p+1}\gamma^{1-p}\dd\mathcal H^{d-1}.
\end{aligned}
\end{equation}
At $p=1$,
\begin{equation}\label{eq:main-L1}
\begin{aligned}
 \ell^{-2}\|z_\eps-T\|_{L^1(\mu)}
 &\longrightarrow \int_X|b_z|\dd\mu+A_{d,1}\int_{\partial X}\rho_0\sigma^2\dd\mathcal H^{d-1},\\
 \ell^{-2}\|m_\eps-T\|_{L^1(\mu)}
 &\longrightarrow \int_X|b_m|\dd\mu+B_{d,1}\int_{\partial X}\rho_0\sigma^2\dd\mathcal H^{d-1}.
\end{aligned}
\end{equation}
In particular, for each $1\le p\le\infty$, %
\begin{equation}\label{eq:main-bias-rates}
 \|z_\eps-T\|_{L^p(\mu)}\asymp\|m_\eps-T\|_{L^p(\mu)}\asymp\ell^{1+1/p}=\eps^{(1+1/p)/(d+2)}.
\end{equation}
\end{theorem}

We observe that for $p>1$, the leading error comes entirely from the boundary layer. At $p=1$, the interior contributes at the same order through $\int_X|b_z|\dd\mu$ or $\int_X|b_m|\dd\mu$, respectively. \Cref{thm:Lp} strengthens the norm asymptotics to weak convergence of the normalized error measures. The limiting measures are supported on the boundary for $p>1$, whereas at $p=1$ they also have the interior parts $|b_z|\mu$ and $|b_m|\mu$, respectively.

Interchanging $\mu$ and $\nu$ gives the corresponding approximations to the inverse Brenier map $S=T^{-1}:Y\to X$. The forward and reverse local bias expansions also imply that the limits in \eqref{eq:global-bulk-bias} hold strongly in $H^1_{\mathrm{loc}}(\intr X)$; see \cref{cor:H1}.

\subsection{Related literature}\label{se:previousResults}

\emph{Quadratically regularized OT.} Starting with \cite{Muzellec.2017.AAAI,BSR,EssidSolomon.18,LMM}, quadratic regularization was introduced to allow for sparse solutions. 
The scale $\ell^2=\eps^{2/(d+2)}$ first arises in the asymptotics of the optimal QOT value. Eckstein and Nutz \cite{EN} obtain an upper bound at this order by quantization. Garriz-Molina, Gonz\'alez-Sanz and Mordant \cite{GGM} identify the leading coefficient under smoothness assumptions, while Cardoso-Perell\'o, Gonz\'alez-Sanz and Nutz \cite{CGN} weaken those assumptions and generalize the regularization. Profiles with a quadratic cap appear in the manifold-learning construction of Zhang, Mordant, Matsumoto and Schiebinger \cite{ZMMSManifoldQOT}, who study QOT in the case of self-transport (i.e., $\mu=\nu$). 
Regarding the size and location of the support, Gonz\'alez-Sanz and Nutz \cite{GSN} prove in the one-dimensional case that section lengths are of order $\eps^{1/3}$. Wiesel and Xu \cite{WX} obtain concentration and location bounds in higher dimensions. The sharp scale $\eps^{1/(d+2)}$ is obtained for self-transport, a result generalized to marginals with affine Brenier map~$T$ by Nguyen-Chi, Nguyen and Nguyen \cite{NCN}. Regarding the dual potentials, Gvalani and Koch \cite{GK} prove interior Hessian bounds and local $C^{1,\gamma}$ convergence, while Gonz\'alez-Sanz, Gvalani and Koch \cite{GGK} establish the sharp scale for interior sections. Global versions of those results are obtained in our Part I \cite{GNGeometry}, whose bounds are the starting point of the present paper. We are not aware of previous works in QOT obtaining leading coefficients or asymptotic profiles for the potentials, sections, conditional laws, or approximations of the Brenier map (see below for entropic OT).

The aforementioned works do not address the inclusion of the Brenier graph. For discrete QOT, the solution for small $\eps$ coincides with an unregularized solution; in that sense, the inclusion is clear. Gonz\'alez-Sanz and Nutz \cite{GNExactRegularization} give a smallness threshold for such $\eps$. Gonz\'alez-Sanz, Nutz and Riveros Valdevenito \cite{GNRMonotonicity} show that supports need not be nested w.r.t.\ $\eps$ and give a discrete example with 5-point marginals where, for some $\eps$, the regularized support does not include the unregularized support. Our \cref{prop:graph-exclusion} gives smooth examples for $d\geq2$.

\emph{Entropically regularized OT.} For entropic regularization, the coupling has full support, hence this aspect is not directly comparable. Concentration around the Brenier map can instead be described for the optimal density, as for instance by the large deviation principle of Bernton, Ghosal and Nutz \cite{BGNLargeDeviations}. Malamut and Sylvestre \cite{MSCouplingConvergence} identify the sharp $\sqrt{\eps}$ scale for convergence of the coupling in $W_2$, with a matching upper bound when the Brenier map is Lipschitz. Qualitative convergence of the entropic potentials for $\eps\to0$ was studied by \cite{GigliTamanini.21,NWPotentials,CCGTGradients}, among others. L\'opez-Rivera \cite{LRUniform} obtains locally uniform bounds of order $\eps^{1/(d+4)}$ for a normalized potential and its gradient, under curvature bounds for the marginal log-densities on $\R^d$. Gonz\'alez-Sanz, Mordant and Sheng \cite{GMS} prove the sharp $\eps\log(1/\eps)$ rate in $L^1(\mu\otimes\nu)$ for the sum of the two potential errors in a compact setting.

Closer to the present work, more precise expansions identify the leading corrections. For marginal densities with a common compact convex support, Mordant \cite{MordantExpansions} states the coefficients at order $\eps$ and $\eps\log\eps$ for expansions in $L^2$, while remarking that uniform expansions would require additional boundary corrections. On the flat torus, Yao, Kim and Schiebinger \cite[Theorem~1.6 and Corollary~1.9]{YKSPotentials} obtain $L^2$ expansions to arbitrary order for nearby marginals along smooth density curves, thus including higher-order expansions for self-transport. Those results can be compared to our \cref{thm:corrector}, which gives expansions that are uniform on the domains and identifies the boundary contribution. We conjecture that an entropic analogue of \cref{thm:corrector} holds, with potential corrections of order $\eps$ after subtracting the common term $\frac{d\eps}{4}\log(2\pi\eps)$ and a boundary layer of width $\sqrt{\eps}$. While in QOT the corrector is nonlocal and determined by the Neumann problem \eqref{eq:corrector-Neumann}, the entropic corrector appears to be simpler, with the explicit expression $\frac14(\log\det DT-\int_X\log\det DT\,\dd\mu)$.

In entropic OT, one commonly uses the conditional mean of the coupling as an approximation of the Brenier map. (The conditional mean equals the gradient of the entropic convex potential, whereas in QOT these are two different approximations~\eqref{eq:forward-approximations}.) Pooladian and Niles-Weed \cite{PNMaps} study this approximation and obtain an upper bound for the bias in $L^2$. Later works also study the leading bias, and are therefore most comparable to ours. Agarwal, Harchaoui, Mulcahy and Pal \cite{AHMPSelf} treat self-transport on $\R^d$ using Langevin diffusion. Extending to distinct marginals, Mulcahy and Pal \cite{MPEntropicDiffusion} prove that the entropic conditional mean admits the expansion $m_\eps=T+\frac{\eps}{2}\nabla\log\rho_0+o(\eps)$ in $L^2(\mu)$ in their setting. The assumptions include uniform log-concavity for the  marginals, so that there is no boundary. Note that the leading coefficient depends only on the source density. The interior QOT biases in our \cref{thm:main-bias}, by contrast, also involve the global Neumann corrector $g_0$. On the other hand, our results show the influence of a boundary layer. We conjecture that similar phenomena arise in the entropic case when the marginal supports have a boundary, and more generally that an entropic analogue of \cref{thm:main-bias} can be shown. The expansion would then carry an additional boundary term of order $\sqrt{\eps}$, and the approximation would slow to $\|m_\eps-T\|_{L^p(\mu)}\asymp \eps^{(1+1/p)/2}$. 

\emph{Semidiscrete OT.} The $L^p$ rate in \cref{thm:main-bias} depends on $p$. Such a dependence has previously been observed in semidiscrete transport. In the unregularized case, Sadhu, Goldfeld and Kato \cite{SGKSemidiscreteStability} identify the first-order variation of the $p$th power of the $L^p$ distance between transport maps under perturbations of the dual weights; their statistical limits have the scale $n^{-1/(2p)}$. The dependence on $p$ reflects jumps of the transport map across Laguerre interfaces. Entropic regularization produces continuous transitions around the same interfaces: Pooladian, Divol and Niles-Weed \cite{PDNSemidiscreteMaps} prove that the squared $L^2$ population bias is $O(\eps)$ and show that this order is sharp. The semidiscrete geometry also determines finer entropic asymptotics. Altschuler, Niles-Weed and Stromme \cite{ANSSemidiscreteAsymptotics} obtain an exact expansion at order $\eps^2$ for the (unpenalized) transport cost of the entropic optimizer, with a coefficient given by certain integrals over Laguerre interfaces. Sadhu, Goldfeld and Kato \cite{SGKEntropicApproximation} show that cancellation across these interfaces yields an $O(\eps^2\log^3(1/\eps))$ bound for the error $m_\eps-T$ when tested against bounded Lipschitz vector fields. In contrast to these semidiscrete phenomena arising from the Laguerre interfaces, the present work considers a smooth limit. The $L^p$ asymptotics in \cref{thm:main-bias} are dominated by a boundary error of amplitude $\ell$ in a layer of width $\ell$, except for $p=1$, where the interior error of order $\ell^2$ also contributes.

\emph{Unregularized OT.} The Neumann corrector in \cref{thm:corrector} connects our potential expansions to linearization of unregularized transport. For transport between Lebesgue measure and measures that are close to it, Goldman, Huesmann and Otto \cite{GHOLinearization} approximate transport displacements $T-\id$ by gradients of functions solving linear Poisson equations. Gonz\'alez-Sanz and Sheng \cite{GS} study differentiability of transport potentials under smooth perturbations of the data, with derivatives characterized by linear elliptic equations and oblique boundary conditions. 
In our Neumann problem \eqref{eq:corrector-Neumann}, the forcing arises from regularization rather than perturbation of the marginals as in the aforementioned works. Finally, the Sobolev convergence \eqref{eq:main-Sobolev} can be compared to the work of De Philippis and Figalli \cite{DFStrongSobolev}, who establish local $W^{2,1}$ stability for the Monge--Amp\`ere equation under $L^1$ convergence of the right-hand sides and derive local Sobolev convergence of the optimal transport maps.

\subsection{Proof strategy}\label{sec:proof-strategy}
We build on the upper bounds in our Part I \cite{GNGeometry}, recalled in \cref{prop:inputs,lem:prelocalization} below: the potential errors $r_\eps=u_\eps-u_0$ and $s_\eps=v_\eps-v_0$ are $O(\ell^2)$, their gradients are $O(\ell)$, and the Hessians are uniformly bounded. After rescaling distances by $\ell$ and potential errors by $\ell^2$, these bounds give local $C^1$ compactness. We identify the limits as unique solutions of certain equations and use their properties to prove the graph inclusion. After that, we determine the potential corrections and the remainders needed for the bias asymptotics.

\paragraph{Brenier graph inclusion.}
Our aim is to bound $\ell^{-2}q_\eps(x,T(x))$ below by a positive constant, uniformly in $x\in X$. Writing $y=T(z)$ puts both variables on $X$ and turns the Brenier graph into the diagonal $x=z$. In these coordinates,
\[
\begin{aligned}
 q_\eps(x,T(z))&=-r_\eps(x)-s_\eps(T(z))-D(x,z), \qquad \mbox{where}\\
 D(x,z)&=u_0(x)-u_0(z)-T(z)\cdot(x-z).
\end{aligned}
\]
The gap $D$ vanishes on the diagonal and is quadratic to leading order near it.

Near a base point $x_0$, which may vary with $\eps$, we introduce coordinates
\[
 x=x_0+\ell\sigma(x_0)H(x_0)^{-1/2}\xi,\qquad z=x_0+\ell\sigma(x_0)H(x_0)^{-1/2}\zeta.
\]
After division by $\ell^2\sigma(x_0)^2$, Taylor expansion gives the limiting gap $|\xi-\zeta|^2/2$. Along a subsequence, the negative errors $-r_\eps(x)$ and $-s_\eps(T(z))$, divided by the same factor, converge to functions $F$ and $G$. The limiting slack is therefore
\[
 Q(\xi,\zeta)=F(\xi)+G(\zeta)-\tfrac12|\xi-\zeta|^2.
\]
The limiting domain $\Omega$ is the whole space if $\dist(x_0,\partial X)/\ell\to\infty$, and a half-space if this ratio has a finite limit. Since $T_\#\mu=\nu$, both marginal constraints involve the same measure in the source coordinates. The choice of $\sigma$ makes the limiting density coefficient equal to one, so these constraints become
\[
 \int_\Omega Q(\xi,\zeta)_+\dd\zeta=1\quad(\xi\in\Omega),\qquad \int_\Omega Q(\xi,\zeta)_+\dd\xi=1\quad(\zeta\in\Omega).
\]
The global $O(\ell^2)$ bounds for $r_\eps$ and $s_\eps$ ensure that $F$ and $G$ are bounded on all of $\Omega$.

The key step is to show that these equations determine $Q$ uniquely among bounded pairs $(F,G)$. We first construct a symmetric solution $F=G=\phi$ with $\inf_\Omega\phi>0$. On the whole space, $\phi=h_d/2$. For the half-space, we take an auxiliary problem with identical uniform marginals on a ball and rescale near a boundary point. Symmetry makes the two regularized potentials equal. At a minimum of their common negative error, convexity and the barycenter formula for the potential gradient bound the slack above by twice that minimum. The marginal constraint and the $O(\ell^d)$ section-volume bound then force the minimum to be at least $c\ell^2$. This positive lower bound survives rescaling; see \cref{lem:halfspace-exists}. The resulting half-space solution is $\phi(\xi)=\varphi_d(\dist(\xi,\partial\Omega))$, corresponding to the boundary profile introduced in \eqref{eq:scalar-halfspace-profile}.

To prove uniqueness, we subtract the marginal equations for $(\phi,\phi)$ from those for $(F,G)$. Factoring each difference of positive parts as $a_+-b_+=\theta(a-b)$, with $\theta\in[0,1]$, gives a coupled linear system for $A=F-\phi$ and $B=G-\phi$. Eliminating $B$ shows that each value of $A$ is a weighted average of its values at nearby points. The kernel in this averaging equation is symmetric and bounded, vanishes beyond a fixed distance, and is uniformly positive for sufficiently close pairs. The last property follows from the positive lower bound for $\phi$. The Liouville theorem in \cref{lem:liouville} implies that $A$ is constant, and the other equation then gives $B=-A$. Thus, as stated in \cref{prop:classification},
\[
 F=\phi+a,\qquad G=\phi-a,\qquad Q(\xi,\xi)=2\phi(\xi)\ge c>0.
\]
The opposite constants $\pm a$ cancel from the slack. Since the base points may vary with $\eps$, compactness turns this bound for rescaled limits into $q_\eps(x,T(x))\ge c\ell^2$ uniformly on $X$, as desired.

The limiting slack also yields the section and conditional-law profiles. In the rescaled source coordinates, the limiting section above $\xi$ is $\{\zeta\in\Omega:Q(\xi,\zeta)\ge0\}$, and the limiting conditional law has density $Q(\xi,\zeta)_+$ with respect to $\dd\zeta$. In the interior, $Q(\xi,\zeta)=h_d-\tfrac12|\xi-\zeta|^2$, so the limiting sections are balls centered at $\xi$, and converting to target coordinates gives the ellipsoids and densities in \eqref{eq:main-ellipsoid-density}. Near the boundary, the analogous description holds with $Q$ determined by the half-space solution $\phi$.

\paragraph{The Neumann corrector.} 
The above classification determines the common part $-\tau_\eps/\ell^2=(r_\eps+s_\eps\circ T)/(2\ell^2)$ of the rescaled potential errors, but leaves opposite constants whose values may depend on the base point. To determine how the corrections are split between the two potentials, we cancel the common profile by considering
\[
g_\eps=\frac{r_\eps-s_\eps\circ T}{2\ell^2}.
\]
Our goal is to show that $g_\eps$ converges uniformly to the mean-zero solution $g_0$ of \eqref{eq:corrector-Neumann}.

Write $Q_\eps(x,z)=q_\eps(x,T(z))$. Since $T_\#\mu=\nu$, the integrals of $(Q_\eps(x,z))_+$ and $(Q_\eps(z,x))_+$ against $\mu(\dd z)$ both equal $\eps$. To subtract these marginal identities, we use the elementary identity
$a_+-b_+=(a-b)\int_0^1\1_{\{(1-t)b+ta>0\}}\dd t$. Defining
\[
\begin{aligned}
J_\eps(x,z)&=\ell^{-d}\int_0^1\1_{\{(1-t)Q_\eps(z,x)+tQ_\eps(x,z)>0\}}\dd t,\\
\mathfrak a(x,z)&=\tfrac12[D(x,z)-D(z,x)],
\end{aligned}
\]
and substituting the expressions for the slacks, we obtain
\[
\int_XJ_\eps(x,z)\left\{\ell^2[g_\eps(z)-g_\eps(x)]-\mathfrak a(x,z)\right\}\dd\mu(z)=0.
\]
The kernel $J_\eps$ is symmetric. Positivity of both slacks near the diagonal gives $J_\eps(x,z)=\ell^{-d}$ when $|x-z|<a\ell$, while support localization gives $J_\eps(x,z)=0$ when $|x-z|>R\ell$, for fixed $0<a<R$; see \eqref{eq:omega-ellipticity}. Moreover, the quadratic terms of $D(x,z)$ and $D(z,x)$ cancel. Their difference is therefore of cubic order: for bounded $w$ with $x+\ell w\in X$, 
\begin{equation}\label{eq:corrector-cubic-intro}
\ell^{-3}\mathfrak a(x,x+\ell w)\longrightarrow\tfrac1{12}D^3u_0(x)[w,w,w].
\end{equation}

We derive a weak formulation by multiplying the equation for $g_\eps$ by a test function $\psi(x)$, integrating against $\mu(\dd x)$, and averaging with the identity obtained by interchanging~$x$ and~$z$. Symmetry of $J_\eps$ and antisymmetry of $\mathfrak a$ then give
\begin{equation}\label{eq:corrector-weak-intro}
\iint_{X\times X}J_\eps(x,z)\left\{\ell^2[g_\eps(z)-g_\eps(x)]-\mathfrak a(x,z)\right\}[\psi(z)-\psi(x)]\dd\mu(x)\dd\mu(z)=0.
\end{equation}
Taking $\psi=g_\eps$ in \eqref{eq:corrector-weak-intro}, the cubic bound and Cauchy--Schwarz give
\[
\iint_{X\times X}J_\eps(x,z)\left|\frac{g_\eps(z)-g_\eps(x)}{\ell}\right|^2\dd\mu(x)\dd\mu(z)\le C.
\]
This energy estimate controls the average squared difference quotients over nearby pairs. Together with the uniform bound on $g_\eps$, it allows us to approximate $g_\eps$ by functions bounded in $H^1(X)$, with an $L^2$ approximation error tending to zero. Rellich compactness therefore gives a subsequential limit $g_\eps\to g$ in $L^2(\mu)$, with $g\in H^1(X)$.

To identify $g$, we set $z=x+\ell w$ in \eqref{eq:corrector-weak-intro}. For interior $x$, both slacks divided by $\ell^2$ tend to $h(x)-w^TH(x)w/2$. Consequently, $\ell^dJ_\eps(x,x+\ell w)$ tends a.e.\ to the indicator of the ellipsoid
\[
E_x=\{w\in\R^d:w^TH(x)w/2<h(x)\}.
\]
The compactness argument also shows that the difference quotients $(g_\eps(x+\ell w)-g_\eps(x))/\ell$ converge weakly in $L^2$ to $\nabla g(x)\cdot w$. The difference quotients of a smooth test function converge strongly to $\nabla\psi(x)\cdot w$. Combining these limits with the cubic expansion \eqref{eq:corrector-cubic-intro} gives
\[
\int_X\rho_0(x)^2\int_{E_x}\left[\nabla g(x)\cdot w-\tfrac1{12}D^3u_0(x)[w,w,w]\right]\nabla\psi(x)\cdot w\dd w\dd x=0.
\]
The term containing $\nabla g$ involves second moments of $w$, while the cubic term paired with $\nabla\psi(x)\cdot w$ involves fourth moments. Evaluating them yields the coefficient $\rho_0H^{-1}$ and the vector field $\tfrac{h}{2(d+4)}\nabla\log\det H$ in \eqref{eq:corrector-Neumann}. The test functions need not vanish on $\partial X$, so this limit gives the zero normal flux condition as well as the equation in the interior. The normalization \eqref{eq:gauge} gives $\int_Xg_\eps\dd\mu=0$, hence $\int_Xg\dd\mu=0$. Uniqueness of the mean-zero Neumann solution identifies every subsequential limit with $g_0$, proving $g_\eps\to g_0$ in $L^2(\mu)$. To obtain the potential expansions \eqref{eq:corrector-uniform}, the proof then strengthens the $L^2(\mu)$ convergence to uniform convergence on $X$. For the sake of brevity, we do not detail this step here.

Separately, we also use the weak formulation \eqref{eq:corrector-weak-intro} 
to derive 
the stronger control of increments stated in \eqref{eq:short-energy}, which is needed for the second-order bias expansions below.

\paragraph{Second-order bias and boundary decay.}
The potential expansions do not by themselves control the gradients. Instead, we obtain the biases of $z_\eps$ and $m_\eps$ via their barycenter formulas:
\begin{equation}\label{eq:barycenter-formulas-proofstrat}
 \begin{aligned}
 z_\eps(x)-T(x)&=\frac{\int_{S_{\eps,x}}(y-T(x))\dd\nu(y)}{\nu(S_{\eps,x})},\qquad
 m_\eps(x)-T(x)&=\int_Y(y-T(x))\frac{(q_\eps(x,y))_+}{\eps}\dd\nu(y).
 \end{aligned}
\end{equation}

We first sharpen the approximation on the Brenier graph. Set
\[
 \tau_\eps=-\frac{r_\eps+s_\eps\circ T}{2}=\frac{q_\eps(\cdot,T(\cdot))}{2},
\]
whose leading interior term is $\ell^2h/2$. With $Q_\eps(x,z)=q_\eps(x,T(z))$, the average of the two slacks is
\[
 U_\eps(x,z)=\frac{Q_\eps(x,z)+Q_\eps(z,x)}2
 =\tau_\eps(x)+\tau_\eps(z)-\frac{D(x,z)+D(z,x)}2.
\]
We wish to control $(U_\eps(x,z))_+$. Writing
\[
 C_\eps(x,z)=\frac{(Q_\eps(x,z))_++(Q_\eps(z,x))_+}{2}-(U_\eps(x,z))_+,
\]
averaging the two marginal equations yields the identity
\begin{equation}\label{eq:averaged-marginal-intro}
 \int_X(U_\eps(x,z))_+\dd\mu(z)=\eps-\int_XC_\eps(x,z)\dd\mu(z).
\end{equation}
We prove in \eqref{eq:switching-small} that the correction $C_\eps$ is small enough not to affect the expansion we seek.

We next substitute $\ell^2h/2$ for $\tau_\eps$ in $U_\eps$ and estimate the error in \eqref{eq:averaged-marginal-intro}. In the coordinates $z=x+\ell w$, the leading term satisfies the marginal constraint by the definition of $h$. The first-order terms from the variations of $h$, the gap, and $\rho_0$ are odd in $w$, so their integrals vanish whenever the integration ball $B(x,R\ell)$ lies in $X$. Near $\partial X$, the restriction $x+\ell w\in X$ destroys this symmetry. Using the notation $\delta(x)=\dist(x,\partial X)$, we therefore separate the error supported in the strip $\{\delta\le R\ell\}$ from the smaller error in its complement.

The difference between $U_\eps$ and this approximation is the sum of $\tau_\eps-\ell^2h/2$ at $x$ and $z$. Factoring the difference of their positive parts gives a symmetric nonnegative kernel $W_\ell$. Dividing by its integral in $z$ yields
\begin{equation}\label{eq:height-linear-intro}
 \begin{aligned}
 (I+P_\ell)(\tau_\eps-\ell^2h/2)&=\mathcal F_\ell,\qquad
 P_\ell f(x)&=\frac{\int_XW_\ell(x,z)f(z)\dd\mu(z)}{\int_XW_\ell(x,z)\dd\mu(z)}.
 \end{aligned}
\end{equation}
Here $I$ is the identity, and $\mathcal F_\ell$ combines the correction involving $C_\eps$ with the error made by substituting $\ell^2h/2$ in the marginal equation. The kernel $W_\ell$ is uniformly positive for sufficiently close pairs and vanishes beyond distance $O(\ell)$; its integral is comparable to $\ell^d$. Thus $P_\ell$ is an average over nearby points. The preceding estimates split $\mathcal F_\ell$ into an $o(\ell^3)$ term in $L^2(X)$ and an $O(\ell^2)$ term supported in a boundary strip of width $O(\ell)$.

Lemma~\ref{lem:plus}, applied with $l=\ell$, gives uniform bounds for $(I+P_\ell)^{-1}$ on $L^2$ and $L^\infty$. It also represents the inverse by
\[
 (I+P_\ell)^{-1}=\frac12\sum_{j=0}^\infty\left(\frac{I-P_\ell}{2}\right)^j,
\]
with geometrically decreasing operator norms. The spatial decay follows from this representation: each application of $P_\ell$ averages only over distances $O(\ell)$, so the $j$th term can carry a function at most distance $O(j\ell)$ beyond its support. A contribution supported near the boundary can therefore reach a distant point only through the small, high-index terms of the series. Applying these estimates to \eqref{eq:height-linear-intro} gives
\[
 \tau_\eps=\ell^2h/2+e_\eps+b_\eps,\qquad
 \|e_\eps\|_{L^2(X)}=o(\ell^3),\qquad
 |b_\eps(x)|\le C\ell^2e^{-c\delta(x)/\ell};
\]
see \eqref{eq:height-decomp}. 
This completes our discussion of the slack $q_\eps(\cdot,T(\cdot))$ on the graph.

To expand the slack away from the graph, we use target coordinates
\[
 y=T(x)+\ell B_xw,\qquad B_x=\sigma(x)H(x)^{1/2}.
\]
Putting $z=S(y)$ gives
\[
 q_\eps(x,y)=\tau_\eps(x)+\tau_\eps(z)+\ell^2[g_\eps(z)-g_\eps(x)]-D(x,z).
\]
Since $z-x=O(\ell)$, the increment of $g_\eps$ contributes at order $\ell^3$. We must therefore approximate it with an $o(\ell)$ error in $L^2$. The estimate \eqref{eq:short-energy} averages over nearby pairs and cannot directly be evaluated at the prescribed point $z=S(y)$. To obtain this control, we smooth $g_\eps-g_0$ at scale $\ell$. The smoothing error is $o(\ell)$ in $L^2$, and the gradient of the smoothed function tends to zero in $L^2$. For bounded $w$, the maps $x\mapsto S(T(x)+\ell B_xw)$ are $O(\ell)$ perturbations of the identity with derivatives $\Id+O(\ell)$. On convex regions sufficiently far from the boundary, change of variables controls the smoothing error after composition with these maps. Integrating the gradient of the smooth part along the segments from $x$ to $S(y)$ then gives
\[
 g_\eps(S(y))-g_\eps(x)=\nabla g_0(x)\cdot(S(y)-x)+o(\ell)
\]
in $L^2$ (w.r.t.\ $x$) on these interior regions, uniformly for $w$ in bounded sets; see \cref{lem:sampling}. Combining this estimate with the expansion of $\tau_\eps$ and a cubic Taylor expansion of the gap determines the order-$\ell^3$ term of the slack. The remainder again separates into a small $L^2$ term and an exponentially decaying boundary term; see \cref{thm:slack}.

The slack expansion also determines the radial displacement of the section boundary $\partial S_{\eps,x}$ from its limiting ellipsoid. In the target coordinates above, the normalized slack
\[
 \frac{q_\eps(x,T(x)+\ell B_xw)}{\ell^2\sigma(x)^2}
\]
has leading term $h_d-|w|^2/2$, whose zero set is the sphere of radius $a_d=\sqrt{2h_d}$. For $|\theta|=1$, let $R_\eps(x,\theta)$ be the radius at which the normalized slack vanishes along $w=r\theta$. Since the radial derivative of the leading term at $r=a_d$ equals $-a_d$, the $C^1$ profile estimates keep the radial derivative of the normalized slack bounded away from zero when $\delta(x)\gg\ell$. The mean-value theorem then yields an expansion of $R_\eps(x,\theta)-a_d$ with the same remainder estimates as above; see \cref{thm:radial}.

Finally, we insert these expansions into the barycenter formulas~\eqref{eq:barycenter-formulas-proofstrat}. For $z_\eps$, both the displacement of the section boundary and the variation of the target measure $\nu$ contribute to the integral, and the result must then be divided by the section mass. For $m_\eps$, we directly expand the integrand. The leading interior profiles are centered, so their first moments vanish. The corrections of order $\ell$ in the rescaled coordinates become biases of order $\ell^2$ after multiplication by the spatial scale $\ell B_x$. Evaluating these moments gives $b_z$ and $b_m$ in \eqref{eq:biases}, while retaining the separate interior and boundary remainders. The latter separation is essential for the global asymptotics discussed next.

\paragraph{Global bias asymptotics.}
We obtain the global limits in \cref{thm:main-bias} by integrating the errors separately near the boundary and away from it. Fix $1\le p<\infty$, let $V_\eps$ denote either $z_\eps-T$ or $m_\eps-T$, and let $b_V$ denote $b_z$ or $b_m$, respectively. For fixed $L>0$, the boundary expansions \eqref{eq:boundary-input} approximate $V_\eps$ on the strip $\{\delta\le L\ell\}$ by
$-\ell\sigma\varphi_d'(t_\eps)n/\gamma$
and
$\ell\sigma\beta_d(t_\eps)n/\gamma$,
respectively, with a uniform $o(\ell)$ remainder. This suffices to compute the leading $p$th-power error on the strip. Over all of $X$, however, the resulting integral error need not be negligible after division by $\ell^{p+1}$. 

The preceding second-order analysis gives stronger estimates on the complement of the strip. \Cref{thm:centers} shows $V_\eps=\ell^2b_V+\mathcal R_\eps+\mathcal B_\eps$ on $X_\ell=\{x\in X:\delta(x)>L_0\ell\}$, for fixed $L_0>0$, where
\[
 \begin{gathered}
 \|\mathcal R_\eps\|_{L^2(X_\ell)}=o(\ell^2),\qquad
 \|\mathcal R_\eps\|_{L^\infty(X_\ell)}\le C\ell,\qquad
 |\mathcal B_\eps(x)|\le C\ell e^{-c\delta(x)/\ell}.
 \end{gathered}
\]
The two remainders are controlled differently. On the one hand, $\int_{X_\ell}|\mathcal R_\eps|^p\dd\mu=o(\ell^{p+1})$. For~$\mathcal B_\eps$, integrating its exponential bound w.r.t.\ the distance $\delta$ from the boundary gives
$
 \ell^{-(p+1)}\int_{\{\delta>L\ell\}}|\mathcal B_\eps|^p\dd\mu
 \le C_pe^{-c_pL}
$
for sufficiently small $\eps$, with $L>L_0$ fixed. Thus its contribution outside the strip becomes arbitrarily small as $L$ increases.

We first let $\eps\downarrow0$ with $L>L_0$ fixed, using the boundary expansion on $\{\delta\le L\ell\}$ and the second-order decomposition on its complement. We then let $L\to\infty$, recovering the full boundary profile integrals and removing the remaining exponential error. The $p$th power of the interior term $\ell^2b_V$ has  integral $O(\ell^{2p})$, which vanishes after division by $\ell^{p+1}$ when $p>1$. At $p=1$, the integral over the complement instead contributes $\int_X|b_V|\dd\mu$. This explains the two different forms of the limits in \eqref{eq:main-Lp} and \eqref{eq:main-L1}. The $L^\infty$ rates follow from the uniform $O(\ell)$ bounds and the nonzero boundary profiles.

Including continuous scalar test functions gives the weak-* limits of $\ell^{-(p+1)}|V_\eps|^p\mu$. Integrating $V_\eps$ itself against vector-valued test functions gives the vector-measure limit, retaining the direction of the error; see \cref{thm:Lp}.

\subsection{Organization}\label{sec:organization}
Section~\ref{sec:preliminaries} collects background results and formulas for various  quantities related to the rescaled potential errors. Section~\ref{sec:models} constructs and classifies bounded solutions of the whole-space and half-space models and proves exponential convergence of the half-space profile. Section~\ref{sec:profiles} proves graph inclusion and uniform limits for the rescaled potentials, sections, and conditional laws, including at the boundary. It also derives the boundary bias formulas and proves Hessian and Sobolev convergence. Section~\ref{sec:macroscopic} identifies the Neumann corrector and proves the uniform potential expansions and nonlocal energy convergence. Section~\ref{sec:centering} gives the second-order bias expansions, with separate estimates for the interior and boundary remainders. Section~\ref{sec:norms} combines these estimates with the boundary profiles to obtain the absolute and signed measure limits and the global $L^p$ asymptotics. Appendix~\ref{app:graph-exclusion} constructs smooth positive densities on a common ball for which graph inclusion fails at a fixed positive regularization parameter.

\section{Preliminaries}\label{sec:setting}\label{sec:preliminaries}

We write $B(x,r)$ and $\B(x,r)$ for the open and closed balls, respectively, $B_r=B(0,r)$, $\Sn=\{\theta\in\R^d:|\theta|=1\}$, and $|E|$ for the Lebesgue measure of a measurable set $E$. A convex body is a compact convex set with nonempty interior. 
When $d=1$, integrals over $\R^{d-1}=\R^0=\{0\}$ mean evaluation at zero.

The symbols $\id$ and $\Id$ denote the identity map and matrix, respectively. For symmetric matrices, $A\preceq B$ means that $B-A$ is positive semidefinite. Matrix norms are operator norms. We write $\Lip(f)$ for the Lipschitz constant of $f$ and $C_b(D)$ for the bounded continuous real-valued functions on $D$.

Regularity on a closed domain means regularity up to the boundary. In particular, a function in $C^{1,1}$ then has a Lipschitz gradient on the closed set. Boundary derivatives are understood as continuous extensions of interior derivatives whenever these extensions exist. On intervals, derivatives at the endpoints are one-sided. Hessians are weak derivatives in the interior, and their bounds hold almost everywhere with respect to Lebesgue measure unless stated otherwise.

For continuous functions, $\|\cdot\|_\infty$ denotes the supremum on the indicated domain. For measurable functions, it denotes the essential supremum. Densities are extended by zero outside their domains. Unless stated otherwise, constants $c,C>0$ may depend on the domains and densities in \cref{ass:data} and on any fixed exponent or compact set under consideration. They are independent of sufficiently small $\eps$ and of the free spatial variables, and may change from line to line. The symbol $A\asymp B$ means that $cB\le A\le CB$ for sufficiently small $\eps$.

\subsection{Regularity and localization}\label{sec:inputs}
For $(x,y)\in X\times Y$, the unregularized potentials satisfy
\begin{equation}\label{eq:unregularized-duality}
 u_0(x)+v_0(y)\ge x\cdot y,\qquad u_0(x)+v_0(T(x))=x\cdot T(x),\qquad T_\#\mu=\nu.
\end{equation}
Recall the section and conditional law \eqref{eq:forward-sections}, and the approximations \eqref{eq:forward-approximations} of $T$. Fixing the target point $y\in Y$ instead of the source point $x$ gives the reverse section and conditional law, together with two approximations to the inverse map $S=T^{-1}$:
\begin{equation}\label{eq:sections}
\begin{gathered}
 K_{\eps,y}=\{x\in X:q_\eps(x,y)\ge0\},\qquad \widehat\pi_\eps^y(\dd x)=h_\eps(x,y)\mu(\dd x),\\
 \widehat z_\eps(y)=\nabla v_\eps(y),\qquad \widehat m_\eps(y)=\int_Xx\,\widehat\pi_\eps^y(\dd x).
\end{gathered}
\end{equation}
We retain the normalization \eqref{eq:gauge}.

\begin{proposition}[Potential estimates and marginal identities]\label{prop:inputs}
Under \cref{ass:data}, the maps $T:X\to Y$ and $S:Y\to X$ are mutually inverse bi-Lipschitz maps with $T(\partial X)=\partial Y$. For some $\beta\in(0,\alpha)$, $u_0\in C^{3,\beta}(X)$ and $v_0\in C^{3,\beta}(Y)$, and the eigenvalues of both Hessians are bounded above and away from zero. For each $\eps>0$, the regularized potentials satisfy $u_\eps\in C^{1,1}(X)$ and $v_\eps\in C^{1,1}(Y)$. For all sufficiently small $\eps$,
\begin{align}
 \|u_\eps-u_0\|_\infty+\|v_\eps-v_0\|_\infty&\le C\ell^2,\label{eq:G-inputs}\\
 c\Id\preceq D^2u_\eps\preceq C\Id\quad\text{a.e. on }X,&\qquad c\Id\preceq D^2v_\eps\preceq C\Id\quad\text{a.e. on }Y.\label{eq:curvature-input}
\end{align}
For every $\eps>0$ and every $x\in X$, $y\in Y$, the sections $S_{\eps,x}$ and $K_{\eps,y}$ are convex bodies. We have the marginal equations 
\begin{align}
\int_Y(q_\eps(x,y))_+\dd\nu(y)&=\eps,&\int_X(q_\eps(x,y))_+\dd\mu(x)&=\eps,\label{eq:marginal}
\end{align}
and the barycenter representations
\begin{align}
 \nabla u_\eps(x)&=\frac{\int_{S_{\eps,x}}y\dd\nu(y)}{\nu(S_{\eps,x})},&\nabla v_\eps(y)&=\frac{\int_{K_{\eps,y}}x\dd\mu(x)}{\mu(K_{\eps,y})}.\label{eq:barycenters}
\end{align}
The zero set of $q_\eps(x,\cdot)$ in $Y$ is contained in $\partial S_{\eps,x}$, and the zero set of $q_\eps(\cdot,y)$ in $X$ is contained in $\partial K_{\eps,y}$. Both zero sets have Lebesgue measure zero. The support of the optimizer $\pi_\eps$ is
\begin{equation}\label{eq:preliminary-support}
 \spt\pi_\eps=\{(x,y)\in X\times Y:q_\eps(x,y)\ge0\}.
\end{equation}
\end{proposition}
\begin{proof}
The properties of $T,S$ and $u_0,v_0$ are classical; see \cite[Lemma~2.1]{GNGeometry}. Under the normalization \eqref{eq:gauge}, the estimates \eqref{eq:G-inputs} and \eqref{eq:curvature-input} follow from \cite[Theorems~1.3 and~1.2]{GNGeometry}, respectively. The $C^{1,1}$ regularity of the regularized potentials, the marginal and barycenter identities \eqref{eq:marginal}--\eqref{eq:barycenters}, and the assertions about the sections and support are well known; see \cite[Section~2.2]{GNGeometry} for details.
\end{proof}

The formulas in \eqref{eq:forward-sections} and \eqref{eq:sections} specify the conditional laws at every point. By \eqref{eq:marginal}, each has total mass one. Equation~\eqref{eq:barycenters} expresses each potential gradient as the mean of the corresponding marginal restricted to its section and normalized to have mass one.

\begin{lemma}[Support localization]\label{lem:prelocalization}
The unregularized duality gap $G_0(x,y)=u_0(x)+v_0(y)-x\cdot y$ satisfies
\begin{equation}\label{eq:unregularized-gap}
 c|y-T(x)|^2\le G_0(x,y)\le C|y-T(x)|^2\qquad((x,y)\in X\times Y).
\end{equation}
For all sufficiently small $\eps$ and all $(x,y)\in X\times Y$,
\begin{gather}
 q_\eps(x,y)\ge0\ \Longrightarrow\ |y-T(x)|+|x-S(y)|\le R\ell,\label{eq:prelocalization}\\
 \|\nabla u_\eps-T\|_\infty+\|\nabla v_\eps-S\|_\infty\le C\ell,\qquad |S_{\eps,x}|+|K_{\eps,y}|\le C\ell^d,\label{eq:pregrad}\\
 q_\eps(x,y)\le C\ell^2.\label{eq:height-upper}
\end{gather}
\end{lemma}
\begin{proof}
By \eqref{eq:unregularized-duality} and $\nabla v_0(T(x))=x$, the gap $G_0(x,y)$ is the remainder after linearizing $v_0$ at $T(x)$. Convexity keeps the segment $[T(x),y]$ inside $Y$, so the Hessian bounds for $v_0$ in \cref{prop:inputs} give \eqref{eq:unregularized-gap}. The potential errors defined in \eqref{eq:potential-errors} satisfy
\begin{equation}\label{eq:gap-slack-identity}
 q_\eps(x,y)=-G_0(x,y)-r_\eps(x)-s_\eps(y).
\end{equation}
Since $G_0\ge0$, the potential bound \eqref{eq:G-inputs} gives \eqref{eq:height-upper}. On $\{q_\eps\ge0\}$, the identity \eqref{eq:gap-slack-identity} also gives $G_0\le C\ell^2$, and hence $|y-T(x)|\le C\ell$ by \eqref{eq:unregularized-gap}. The Lipschitz bound $|x-S(y)|=|S(T(x))-S(y)|\le\Lip(S)|y-T(x)|$ then proves \eqref{eq:prelocalization}. Thus $S_{\eps,x}$ lies in a ball of radius $C\ell$ about $T(x)$, and $K_{\eps,y}$ lies in such a ball about $S(y)$. These inclusions give the volume bounds in \eqref{eq:pregrad}. The normalized restricted marginals in \eqref{eq:barycenters} are supported in the same balls, so their means satisfy the gradient bounds.
\end{proof}

\subsection{Potential errors and scaling}
Composing $s_\eps$ with $T$ puts both potential errors from \eqref{eq:potential-errors} on $X$. We separate their sum and difference by setting
\begin{equation}\label{eq:errors}
 \tau_\eps=-\tfrac12(r_\eps+s_\eps\circ T),\qquad \eta_\eps= \tfrac12(r_\eps-s_\eps\circ T) = \ell^2g_\eps,\qquad g_\eps=\frac{r_\eps-s_\eps\circ T}{2\ell^2}.
\end{equation}
The errors are recovered as $r_\eps=-\tau_\eps+\eta_\eps$ and $s_\eps\circ T=-\tau_\eps-\eta_\eps$. Thus $-\tau_\eps$ is the common contribution, while $\eta_\eps=\ell^2g_\eps$ enters with opposite signs. The normalization \eqref{eq:gauge} and $T_\#\mu=\nu$ give $\int_Xg_\eps\dd\mu=0$. The potential estimate \eqref{eq:G-inputs} gives
\begin{equation}\label{eq:preliminary-error-bounds}
 \|g_\eps\|_\infty+\ell^{-2}\|\tau_\eps\|_\infty\le C.
\end{equation}

Writing $y=T(z)$ also puts the duality gap on $X\times X$. We define its symmetric and antisymmetric parts with respect to exchanging $x$ and $z$ by
\begin{equation}\label{eq:D}
\begin{gathered}
 D(x,z)=u_0(x)-u_0(z)-T(z)\cdot(x-z)=G_0(x,T(z)),\\
 D^{\rm s}(x,z)=\tfrac12[D(x,z)+D(z,x)],\qquad \mathfrak a(x,z)=\tfrac12[D(x,z)-D(z,x)].
\end{gathered}
\end{equation}
The regularized slack in these source coordinates is obtained by substituting \eqref{eq:errors} and \eqref{eq:D} into \eqref{eq:gap-slack-identity}:
\begin{equation}\label{eq:Q}
 Q_\eps(x,z):=q_\eps(x,T(z))=\tau_\eps(x)+\tau_\eps(z)+\ell^2[g_\eps(z)-g_\eps(x)]-D(x,z).
\end{equation}
Under the change of variables $y=T(z)$, the marginal identities \eqref{eq:marginal} both involve $\mu$, since $T_\#\mu=\nu$:
\begin{equation}\label{eq:Q-marginals}
 \int_X(Q_\eps(x,z))_+\dd\mu(z)=\eps\quad(x\in X),\qquad \int_X(Q_\eps(x,z))_+\dd\mu(x)=\eps\quad(z\in X).
\end{equation}
Moreover, for all $x,z\in X$,
\begin{equation}\label{eq:common-gap-bounds}
 Q_\eps(x,x)=2\tau_\eps(x),\qquad D(x,z)\asymp|x-z|^2,\qquad |\mathfrak a(x,z)|\le C|x-z|^3.
\end{equation}
The first identity follows from \eqref{eq:Q} and identifies $\tau_\eps(x)$ as half the slack at the Brenier pair $(x,T(x))$. The two-sided bound for $D$ follows from Taylor's formula and the Hessian bounds for $u_0$ in \cref{prop:inputs}. Since $H$ is Lipschitz, both $D(x,z)$ and $D(z,x)$ equal $\tfrac12(x-z)^TH(z)(x-z)+O(|x-z|^3)$, uniformly on $X\times X$. Their common quadratic term cancels in the difference, giving the cubic bound for $\mathfrak a$.

For the rescaled profiles, we use a radial quadratic density and linear changes of source and target coordinates. With $h_d,\sigma,h$ from \eqref{eq:scale}, set
\begin{equation}\label{eq:rescaling-matrices}
\begin{gathered}
 a_d=\sqrt{2h_d},\qquad \kappa_d(w)=(h_d-|w|^2/2)_+,\\
 \Lambda_x=\sigma(x)H(x)^{-1/2},\qquad B_x=\sigma(x)H(x)^{1/2}.
\end{gathered}
\end{equation}
The matrices $\Lambda_x$ and $B_x$ normalize source and target displacements, respectively. In the coordinates $z=x+\ell\Lambda_x\xi$ and $y=T(x)+\ell B_xw$, the quadratic parts of $D(x,z)$ and $G_0(x,y)$ are $\ell^2\sigma(x)^2|\xi|^2/2$ and $\ell^2\sigma(x)^2|w|^2/2$. Since $H(x)\Lambda_x=B_x$, the linearization of $T$ takes the source displacement $\ell\Lambda_x\xi$ to the target displacement $\ell B_x\xi$.

The function $\kappa_d$ is positive on $B_{a_d}$ and zero outside that ball. Radial integration gives
\begin{equation}\label{eq:cap-normalization}
 \int_{\R^d}\kappa_d(w)\dd w=\frac{\omega_da_d^{d+2}}{d+2}=1.
\end{equation}
Thus $\kappa_d$ is a probability density. Combining the Monge--Amp\`ere equation \eqref{eq:transport-jacobian} with \eqref{eq:scale} gives
\begin{equation}\label{eq:cap-change-of-variables}
 \frac{\kappa_d(B_x^{-1}w)}{\det B_x}=\rho_1(T(x))\left(h(x)-\tfrac12w^TH(x)^{-1}w\right)_+=k_x(w).
\end{equation}
The linear map $w\mapsto B_xw$ therefore sends the ball $\B(0,a_d)$ to the ellipsoid $\mathscr E_x$ and the probability measure $\kappa_d(w)\dd w$ to $k_x(w)\dd w$. This expresses the interior conditional profiles from \eqref{eq:main-ellipsoid-density} as linear images of a single radial law.

By \cref{ass:data,prop:inputs}, the coefficient fields $H,\rho_0,\sigma,h,\Lambda,B$ belong to $C^{1,\beta}(X)$. The scalar functions $\rho_0,\sigma,h$ and the eigenvalues of $H,\Lambda,B$ are bounded above and away from zero.

\section{Whole-space and half-space models}\label{sec:models}

We study the model equations that arise by rescaling the marginal identities \eqref{eq:Q-marginals} near the Brenier graph. Their domains are the whole space or a half-space, depending on the distance to the boundary after rescaling. We classify the bounded solutions, construct the scalar profile $\varphi_d$ from \eqref{eq:scalar-halfspace-profile}, and prove its exponential convergence to $h_d/2$. This section analyzes the model itself; the connection with the rescaled limits will be established in \cref{sec:profiles}.

\subsection{Classification of bounded solutions}
In our model, the slack is $F(\xi)+G(\zeta)-|\xi-\zeta|^2/2$, and its positive part must integrate to $\tau$ in either variable. A symmetric solution has $F=G$. We allow arbitrary $\tau>0$, although the rescaled limits in \cref{sec:profiles} have $\tau=1$.

\begin{proposition}[Classification of bounded solutions]\label{prop:classification}
Let $\Omega=\R^d$ or a closed half-space, and let $\tau>0$. There is a unique $\phi_{\Omega,\tau}\in C_b(\Omega)$ satisfying
\begin{equation}\label{eq:symmetric-general-model}
 \int_\Omega\left(\phi_{\Omega,\tau}(\xi)+\phi_{\Omega,\tau}(\zeta)-\tfrac12|\xi-\zeta|^2\right)_+\dd\zeta=\tau\qquad(\xi\in\Omega).
\end{equation}
Every pair $F,G\in C_b(\Omega)$ satisfying
\begin{equation}\label{eq:model-pair}
\begin{aligned}
 \int_\Omega(F(\xi)+G(\zeta)-\tfrac12|\xi-\zeta|^2)_+\dd\zeta&=\tau,\\
 \int_\Omega(F(\zeta)+G(\xi)-\tfrac12|\xi-\zeta|^2)_+\dd\zeta&=\tau
\end{aligned}
\qquad(\xi\in\Omega)
\end{equation}
has the form
\begin{equation}\label{eq:classification}
 F=\phi_{\Omega,\tau}+a,\qquad G=\phi_{\Omega,\tau}-a
\end{equation}
for a constant $a\in\R$. On the whole space,
\begin{equation}\label{eq:whole-constant}
 \phi_{\R^d,\tau}=\frac{h_d}{2}\tau^{2/(d+2)}.
\end{equation}
For a half-space, let $L:\Omega\to\Hh$ be an affine isometry onto $\Hh$. Then
\begin{equation}\label{eq:halfspace-scaling}
 \phi_{\Omega,\tau}(\xi)=\tau^{2/(d+2)}\phi_{\Hh,1}\bigl(\tau^{-1/(d+2)}L(\xi)\bigr).
\end{equation}
In both cases,
$c_d\tau^{2/(d+2)}\le\phi_{\Omega,\tau}\le C_d\tau^{2/(d+2)}$ for positive constants depending only on $d$.
\end{proposition}

The function $\phi_{\Hh,1}$ is described in detail in \cref{sec:profile-properties}. The proof compares an arbitrary solution with a positive symmetric one. Leveraging results and arguments of \cite{MSC2015}, we first prove a Liouville theorem for the averaging equation that arises from this comparison.

\begin{lemma}[Liouville theorem for finite-range kernels]\label{lem:liouville}
Let $\Omega=\R^d$ or a closed half-space. Suppose that $J:\Omega\times\Omega\to[0,\infty)$ is measurable and symmetric, and that
\begin{equation}\label{eq:liouville-kernel}
 c\1_{\{|x-y|<a\}}\le J(x,y)\le C\1_{\{|x-y|\le b\}}\qquad(x,y\in\Omega)
\end{equation}
for constants $c,C>0$ and $0<a\le b<\infty$. Every $u\in L^\infty(\Omega)$ satisfying
\begin{equation}\label{eq:liouville-equation}
 \int_\Omega J(x,y)[u(y)-u(x)]\dd y=0\qquad\text{for a.e. }x\in\Omega
\end{equation}
is almost everywhere constant. If $u$ is continuous on $\Omega$, it is constant everywhere.
\end{lemma}
\begin{proof} 
 Set $m(x)=\int_\Omega J(x,y)\dd y$, $\vartheta(\dd x)=m(x)\dd x$, $p(x,y)=J(x,y)/(m(x)m(y))$, and $Pf(x)=\int_\Omega p(x,y)f(y)\dd\vartheta(y)$.
Since $|\Omega\cap B(x,r)|\asymp r^d$, uniformly in $x\in\Omega$, the bounds
\eqref{eq:liouville-kernel} imply $m\asymp1$. Hence $p$ is symmetric,
$p(x,y)\vartheta(\dd y)$ is a probability measure for every $x$, and
$V(x,r):=\vartheta(\Omega\cap B(x,r))\asymp r^d$. In particular,
$(\Omega,|\cdot|,\vartheta)$ satisfies both local and large-scale volume doubling in
the sense of \cite[Definition~2.5]{MSC2015}. Equation
\eqref{eq:liouville-equation} is equivalent to $Pu=u$ almost everywhere. Since the
transition measures are absolutely continuous with respect to $\vartheta$, this implies
$P(Pu)=Pu$ everywhere. Thus $\widetilde u:=Pu$ is a bounded pointwise
$P$-harmonic representative of $u$.

For $\Omega=\mathbb R^d$, the classical Euclidean Poincar\'e inequality is
$(P)_{\mathrm{Rm}}$ in the sense of \cite[Definition~3.1]{MSC2015}; hence,
by \cite[Proposition~3.20]{MSC2015}, $(P)_{h_0}$ holds for $h_0=a/2$. In the half-space case, an isometry reduces to $\Omega=\{x_d\ge0\}$. For $x_0\in\Omega$, write $D_R=\Omega\cap B(x_0,R)$. Let $\mathfrak r(x',x_d)=(x',|x_d|)$ and extend $f\in L^\infty_{\rm loc}(\Omega)$ to $\R^d$ by $F=f\circ\mathfrak r$. With the notation $f_{E,\lambda}:=\frac{1}{\lambda(E)}\int_E f\,\dd\lambda$, we then have
\[
 \int_{D_R}|f-f_{D_R,\dd x}|^2\dd x
 \le \int_{B(x_0,R)}|F-F_{B(x_0,R),\dd x}|^2\dd x.
\]
Apply the nonlocal Poincar\'e inequality for $\R^d$ at scale $h_0$ to the right-hand side. The map $\mathfrak r$ is $1$-Lipschitz, fixes $x_0$, and has at most two preimages almost everywhere. Applying the reflection map $\mathfrak r$ to both integration variables and using $|\Omega\cap B(x,h_0)|\le|B(0,h_0)|$ therefore gives
\[
\begin{aligned}
 &\int_{B(x_0,\kappa R)}\frac{1}{|B(0,h_0)|}
 \int_{B(x,h_0)}|F(y)-F(x)|^2\dd y\,\dd x\\
 &\qquad\le4\int_{D_{\kappa R}}\frac{1}{|\Omega\cap B(x,h_0)|}
 \int_{\Omega\cap B(x,h_0)}|f(y)-f(x)|^2\dd y\,\dd x.
\end{aligned}
\]
Thus $(P)_{h_0}$ holds on the half-space as well.
Consequently, for some $C,\kappa,R_0>0$,
\[
 \int_{D_R}|f-f_{D_R,\dd x}|^2\dd x
 \le CR^2
 \int_{D_{\kappa R}}
 \frac{1}{|\Omega\cap B(x,h_0)|}
 \int_{\Omega\cap B(x,h_0)}
 |f(y)-f(x)|^2\dd y\,\dd x,
 \qquad R\ge R_0 .
\]
This inequality is unchanged, up to its constant, when Lebesgue measure is replaced by
$\vartheta$. Indeed, write $0<\lambda\le m\le\Lambda$. Since
$f_{D_R,\vartheta}$ minimizes the $L^2(\vartheta)$ error,
$\int_{D_R}|f-f_{D_R,\vartheta}|^2\dd\vartheta
\le\Lambda\int_{D_R}|f-f_{D_R,\dd x}|^2\dd x$.
Moreover
$\lambda|\Omega\cap B(x,h_0)|\le V(x,h_0)
\le\Lambda|\Omega\cap B(x,h_0)|$ and
$\dd\vartheta(x)\dd\vartheta(y)\ge\lambda^2\dd x\,\dd y$.
Consequently,
\begin{equation}\label{eq:liouville-poincare}
 \int_{D_R}|f-f_{D_R,\vartheta}|^2\dd\vartheta
 \le C R^2
 \int_{D_{\kappa R}}
 \frac{1}{V(x,h_0)}
 \int_{\Omega\cap B(x,h_0)}
 |f(y)-f(x)|^2
 \dd\vartheta(y)\dd\vartheta(x),
\end{equation}
after changing $C$. This is precisely the Poincar\'e inequality $(P)_{h_0}$ of
\cite[Definition~3.4]{MSC2015}.

It remains to check the hypotheses on the Markov kernel. Since $\Omega$ is convex,
straight segments show that it is $\beta$-geodesic, in the sense of
\cite[Definition~2.3]{MSC2015}, for every $\beta>0$. In particular, we may take
$\beta=h_0/2<h_0$. Furthermore, using $m\asymp1$, $V(x,r)\asymp r^d$, and
\eqref{eq:liouville-kernel}, there are constants $c_1,C_1>0$ such that
\[
 \frac{c_1}{V(x,h_0)}\1_{B(x,h_0)}(y)
 \le p(x,y)
 \le \frac{C_1}{V(x,b)}\1_{B(x,b)}(y).
\]
Thus $P$ is weakly $(h_0,b)$-compatible in the sense of
\cite[Definition~4.7]{MSC2015}. Since $h_0>\beta$, 
\cite[Lemma~4.9]{MSC2015} yields $k\ge1$ such that $P^k$ has an
$(h_0,kb)$-compatible kernel. As $P\widetilde u=\widetilde u$, we also have
$P^k\widetilde u=\widetilde u$. We have therefore verified all the hypotheses of
\cite[Proposition~6.19]{MSC2015} for $P^k$ on
$(\Omega,|\cdot|,\vartheta)$. That proposition implies that the bounded
$P^k$-harmonic function $\widetilde u$ is constant. Since
$\widetilde u=u$ almost everywhere, $u$ is almost everywhere constant. If $u$ is
continuous, it is constant everywhere.
\end{proof}

Next, we construct the positive half-space solution.

\begin{lemma}[Existence of a half-space profile]\label{lem:halfspace-exists}
Let $\Hh=\{\xi_d\ge0\}$. There is a function $\phi_{\Hh,1}\in C^{1,1}(\Hh)$ satisfying
\begin{equation}\label{eq:symmetric-model}
 \int_\Hh\left(\phi_{\Hh,1}(\xi)+\phi_{\Hh,1}(\zeta)-\tfrac12|\xi-\zeta|^2\right)_+\dd\zeta=1\qquad(\xi\in\Hh).
\end{equation}
For constants $c,C>0$ depending only on $d$,
\begin{equation}\label{eq:halfspace-profile-bounds}
 c\le\phi_{\Hh,1}\le C,\qquad \|\nabla\phi_{\Hh,1}\|_\infty+\Lip(\nabla\phi_{\Hh,1})\le C,\qquad c\Id\preceq D^2U\preceq C\Id\quad\text{a.e. on }\Hh,
\end{equation}
where $U(\xi)=|\xi|^2/2-\phi_{\Hh,1}(\xi)$.
\end{lemma}
\begin{proof}
We construct the solution by rescaling near a boundary point in a problem with identical marginals. Put $R_*=\omega_d^{-1/d}$ and $D_*=\B(R_*e_d,R_*)$. This unit volume ball lies in $\Hh$ and is tangent to its boundary at zero. Take both marginals in \eqref{eq:problem} to be Lebesgue measure on $D_*$. The unregularized map is the identity, and \cref{prop:inputs,lem:prelocalization} apply with constants depending only on $d$, for example with $\alpha=1/2$. Symmetry and uniqueness under \eqref{eq:gauge} give $u_\eps=v_\eps$, while $u_0=v_0=|x|^2/2$. With $f_\eps(x)=|x|^2/2-u_\eps(x)$,
\begin{equation}\label{eq:self-slack}
 q_\eps(x,y)=f_\eps(x)+f_\eps(y)-\tfrac12|x-y|^2.
\end{equation}
Equations~\eqref{eq:G-inputs} and \eqref{eq:pregrad} give $\|f_\eps\|_\infty\le C\ell^2$ and $\|\nabla f_\eps\|_\infty\le C\ell$.

For the lower bound, let $x_*$ minimize $f_\eps$ on $D_*$ and set $z=\nabla u_\eps(x_*)$. The barycenter identity \eqref{eq:barycenters} gives $z\in D_*$. Evaluating the supporting-plane inequality for $u_\eps$ at $x_*$ and $z$ gives
\[
 f_\eps(z)\le f_\eps(x_*)-\tfrac12|z-x_*|^2.
\]
Minimality forces $z=x_*$. Evaluating the same supporting plane at an arbitrary $y\in D_*$ then gives $q_\eps(x_*,y)\le2f_\eps(x_*)$. Since the marginal integral \eqref{eq:marginal} is positive, $f_\eps(x_*)>0$. The marginal density is one, so the bound \eqref{eq:pregrad} for the section volume yields
\begin{equation}\label{eq:self-positive-height}
 \ell^{d+2}=\eps\le2f_\eps(x_*)|S_{\eps,x_*}|\le C\ell^df_\eps(x_*).
\end{equation}
Thus $c\ell^2\le f_\eps\le C\ell^2$ on $D_*$.

Rescale distances from the tangency point by $\ell$ and potential errors by $\ell^2$. On $D_\eps=\ell^{-1}D_*$, set $F_\eps(\xi)=\ell^{-2}f_\eps(\ell\xi)$ and $U_\eps(\xi)=|\xi|^2/2-F_\eps(\xi)=\ell^{-2}u_\eps(\ell\xi)$. The bounds just proved and \eqref{eq:curvature-input} give
\begin{equation}\label{eq:halfspace-approximation-bounds}
 c\le F_\eps\le C,\qquad \|\nabla F_\eps\|_\infty+\Lip(\nabla F_\eps)\le C,\qquad c\Id\preceq D^2U_\eps\preceq C\Id\quad\text{a.e. on }D_\eps.
\end{equation}
Since $D_\eps=\{\xi:\ell|\xi|^2\le2R_*\xi_d\}$, every compact subset of $\intr\Hh$ eventually lies in $D_\eps$, and $\1_{D_\eps}\to\1_\Hh$ almost everywhere. By \eqref{eq:halfspace-approximation-bounds} and Arzel\`a--Ascoli, along a sequence $\eps\downarrow0$ the functions $F_\eps$ converge in $C^1$ on compact subsets of $\intr\Hh$ to a function $\phi_{\Hh,1}$. The uniform value and gradient bounds pass to the limit. The Lipschitz bounds extend $\phi_{\Hh,1}$ and $\nabla\phi_{\Hh,1}$ continuously to $\Hh$, giving $\phi_{\Hh,1}\in C^{1,1}(\Hh)$. Passing the two-sided Taylor inequalities for $U_\eps$ to the limit at interior points gives the Hessian bounds for $U$. Thus \eqref{eq:halfspace-profile-bounds} holds.

For fixed $\xi\in\intr\Hh$, rescaling \eqref{eq:marginal} gives
\[
 \int_{D_\eps}\left(F_\eps(\xi)+F_\eps(\zeta)-\tfrac12|\xi-\zeta|^2\right)_+\dd\zeta=1
\]
for all sufficiently small $\eps$ along this sequence. The uniform bound for $F_\eps$ confines the integrand to a fixed ball about $\xi$. Extend the integrand by zero outside $D_\eps$. Interior convergence and convergence of the domain indicators allow dominated convergence, proving \eqref{eq:symmetric-model} for $\xi\in\intr\Hh$. Its left-hand side is continuous in $\xi$ on $\Hh$, by boundedness and continuity of $\phi_{\Hh,1}$ and the same support bound, so the identity extends to $\partial\Hh$.
\end{proof}

\begin{proof}[Proof of \cref{prop:classification}]
On $\R^d$, the constant in \eqref{eq:whole-constant} solves \eqref{eq:symmetric-general-model} by the normalization \eqref{eq:cap-normalization} and the dilation $w\mapsto\tau^{1/(d+2)}w$. On a half-space, let $\phi_{\Hh,1}$ be the solution from \cref{lem:halfspace-exists}. Changing variables shows that $\xi\mapsto\tau^{2/(d+2)}\phi_{\Hh,1}(\tau^{-1/(d+2)}L(\xi))$ solves \eqref{eq:symmetric-general-model}. In either case, denote this positive bounded reference solution by $\phi_0$.

Let $F,G$ be any bounded solution pair of \eqref{eq:model-pair}. Its slack and the reference slack are
\[
 Q(\xi,\zeta)=F(\xi)+G(\zeta)-\tfrac12|\xi-\zeta|^2,\qquad Q_0(\xi,\zeta)=\phi_0(\xi)+\phi_0(\zeta)-\tfrac12|\xi-\zeta|^2.
\]
The fundamental theorem of calculus yields
\begin{equation}\label{eq:W-secant}
 (Q)_+-(Q_{0})_+=W(Q-Q_0),\qquad W(\xi,\zeta)=\int_0^1\1_{\{(1-t)Q_0(\xi,\zeta)+tQ(\xi,\zeta)>0\}}\dd t.
\end{equation}
The weight $W$ lies in $[0,1]$. Boundedness of $F,G,\phi_0$ and $\inf\phi_0>0$ give, for some $c>0$ and $0<r<R<\infty$,
\begin{equation}\label{eq:W-bounds}
 c\1_{\{|\xi-\zeta|<r\}}\le W(\xi,\zeta)\le\1_{\{|\xi-\zeta|\le R\}}.
\end{equation}

Write $A=F-\phi_0$ and $B=G-\phi_0$. Let $m(\xi)=\int_\Omega W(\xi,\zeta)\dd\zeta$ and $n(\zeta)=\int_\Omega W(\xi,\zeta)\dd\xi$. Integrating in \eqref{eq:W-secant} and using that both pairs solve \eqref{eq:model-pair} gives 
\begin{equation}\label{eq:classification-linear-system}
\begin{aligned}
 m(\xi)A(\xi)+\int_\Omega W(\xi,\zeta)B(\zeta)\dd\zeta&=0,\\
 n(\zeta)B(\zeta)+\int_\Omega W(\xi,\zeta)A(\xi)\dd\xi&=0.
\end{aligned}
\end{equation}
The bounds \eqref{eq:W-bounds} and $|\Omega\cap B(\xi,r)|\asymp r^d$ give $m,n\asymp1$. Solve the second equation of \eqref{eq:classification-linear-system} for $B$ and substitute it into the first. This gives
\begin{equation}\label{eq:eliminated}
 \int_\Omega J(\xi,\xi')[A(\xi')-A(\xi)]\dd\xi'=0,\qquad J(\xi,\xi')=\int_\Omega\frac{W(\xi,\zeta)W(\xi',\zeta)}{n(\zeta)}\dd\zeta.
\end{equation}
The kernel $J$ is symmetric,  uniformly bounded and vanishes for $|\xi-\xi'|>2R$. If $|\xi-\xi'|<r/2$, both factors in its defining integral are bounded below on $\Omega\cap B(\xi,r/2)$, giving a uniform lower bound for $J$. Thus \cref{lem:liouville} makes $A$ constant, and the second equation in \eqref{eq:classification-linear-system} gives $B=-A$. This proves \eqref{eq:classification}. For a symmetric solution, $A=B$, so both constants vanish.  Uniqueness follows, and the formulas and bounds follow from those on $\phi_0$.
\end{proof}

\begin{remark}[Necessity of boundedness]\label{rem:boundedness}
The boundedness assumption in \cref{prop:classification} is essential for uniqueness. More specifically, it rules out translated profiles as follows. On $\R^d$, put $h_\tau=h_d\tau^{2/(d+2)}$. For $b\in\R^d$, the affine functions $F(\xi)=h_\tau/2-|b|^2/4+b\cdot\xi$ and $G(\zeta)=h_\tau/2-|b|^2/4-b\cdot\zeta$ give the slack $h_\tau-|\xi-\zeta-b|^2/2$. Translation and \eqref{eq:cap-normalization} show that both marginal integrals of its positive part equal $\tau$. For fixed $\xi$, the section is centered at $\zeta=\xi-b$ rather than $\zeta=\xi$. When $b\ne0$, the potentials are unbounded, and the diagonal slack is negative if $|b|^2>2h_\tau$. %
\end{remark}

\subsection{Properties of the half-space profile}\label{sec:profile-properties}

By \cref{prop:classification}, $\phi_{\Hh,1}$ is the unique bounded continuous solution of \eqref{eq:symmetric-model}. Write $\xi=(\xi_{-d},\xi_d)$. For every $\xi'\in\R^{d-1}$, the function $\xi\mapsto\phi_{\Hh,1}(\xi_{-d}+\xi',\xi_d)$ solves the same equation. Uniqueness therefore gives 
\[
\phi_{\Hh,1}(\xi)=\phi_{\Hh,1}(0,\xi_d)=:\varphi_d(\xi_d),
\]
where $\varphi_d$ satisfies \eqref{eq:scalar-halfspace-profile}. Conversely, any bounded continuous solution $\widetilde\varphi_d$ of \eqref{eq:scalar-halfspace-profile} yields a bounded continuous solution $\xi\mapsto\widetilde\varphi_d(\xi_d)$ of \eqref{eq:symmetric-model}. Uniqueness in \cref{prop:classification} then gives $\widetilde\varphi_d=\varphi_d$.

For $0\le t<\infty$ and a unit vector $N$, consider
\eqref{eq:symmetric-general-model} with $\tau=1$ on the half-space
$  \Hh_{t,N}:=\{w\in\R^d:N\cdot w\ge -t\}.$
By \cref{prop:classification}, its unique bounded solution is
$\phi_{t,N}(w)=\phi_{\Hh_{t,N},1}(w)=\varphi_d(t+N\cdot w)$, since
$\Hh_{t,N}$ is isometric to $\Hh$ and
$\phi_{\Hh,1}(\xi)=\varphi_d(\xi_d)$. Conditioning at $0\in\Hh_{t,N}$, the
corresponding slack, section, and conditional density are therefore
\begin{equation}\label{eq:full-boundary-model}
\begin{gathered}
 q_{t,N}(w)=\varphi_d(t)+\varphi_d(t+N\cdot w)-\frac{|w|^2}{2},\\
 \mathcal C_{t,N}=\{w\in\Hh_{t,N}:q_{t,N}(w)\ge0\},\qquad
 \kappa_{t,N}(w)=(q_{t,N}(w))_+\1_{\Hh_{t,N}}(w).
\end{gathered}
\end{equation}
By \eqref{eq:halfspace-mean-profile} and rotational invariance,
$\beta_d(t)=\int_{\R^d}(N\cdot w)\kappa_{t,N}(w)\dd w$. At $t=\infty$, we use the whole-space model: $\Hh_{\infty,N}=\R^d$, $q_{\infty,N}(w)=h_d-|w|^2/2$, $\mathcal C_{\infty,N}=\B(0,a_d)$ and $\kappa_{\infty,N}=\kappa_d$. We retain $\varphi_d(\infty)=h_d/2$ and set $\varphi_d'(\infty)=\beta_d(\infty)=0$. Proposition~\ref{prop:profile-exponential} will show that these definitions agree with the limits as $t\to\infty$. The next lemma distinguishes the centroid of the section, computed with uniform weight, from the mean of its conditional density.

\begin{lemma}[Regularity of the half-space profile]\label{lem:profile-regularity}
The function $\varphi_d$ belongs to $C^{1,1}([0,\infty))$ and is bounded above and away from zero. The function $U(\xi)=|\xi|^2/2-\varphi_d(\xi_d)$ satisfies the Hessian bounds in \eqref{eq:halfspace-profile-bounds}. For every $t\in[0,\infty]$ and every unit vector $N$, the set $\mathcal C_{t,N}$ is a convex body. There are constants $r_0,R_0,c_0>0$, depending only on $d$, such that
\begin{equation}\label{eq:model-section-bounds}
 B(r_0N,r_0/2)\subset\mathcal C_{t,N}\subset\B(0,R_0),\qquad q_{t,N}\ge c_0\quad\text{on }B(r_0N,r_0/2).
\end{equation}
The densities $\kappa_{t,N}$ have unit mass, $\beta_d$ is bounded and continuous on $[0,\infty)$, and
\begin{equation}\label{eq:model-centroid}
 \frac1{|\mathcal C_{t,N}|}\int_{\mathcal C_{t,N}}w\dd w=-\varphi_d'(t)N,\qquad \int_{\R^d}w\kappa_{t,N}(w)\dd w=\beta_d(t)N.
\end{equation}
In particular, $\varphi_d'(0)<0$ and $\beta_d(0)>0$.
\end{lemma}
\begin{proof}
The regularity, positivity and Hessian bounds follow from \cref{lem:halfspace-exists}.

For finite $t$, the Hessian bounds for $U$ make $q_{t,N}$ uniformly concave. At $t=\infty$,  this follows from \eqref{eq:whole-constant}. Set $m_*=\min\{\inf\varphi_d,h_d/2\}>0$ and $M_*=\max\{\sup\varphi_d,h_d/2\}$. For $w\in\Hh_{t,N}$ and $t\in [0,\infty]$, \eqref{eq:full-boundary-model} and \eqref{eq:whole-constant} give
\begin{equation}\label{eq:model-slack-bounds}
 2m_*-|w|^2/2\le q_{t,N}(w)\le2M_*-|w|^2/2.
\end{equation}
Choose $r_0>0$ so that
$9r_0^2/8\le m_*$. If $w\in B(r_0N,r_0/2)$, then
$N\cdot w\ge r_0-|w-r_0N|>r_0/2\ge -t$ and
$|w|\le |w-r_0N|+r_0<3r_0/2$. Hence
$B(r_0N,r_0/2)\subset\operatorname{int}\Hh_{t,N}$ and, by
\eqref{eq:model-slack-bounds},
$q_{t,N}(w)\ge 2m_*-9r_0^2/8\ge m_*$. Thus the first inclusion in
\eqref{eq:model-section-bounds} holds, with $c_0=m_*$. Conversely, if
$w\in\mathcal C_{t,N}$, then $q_{t,N}(w)\ge0$, and the upper bound in
\eqref{eq:model-slack-bounds} yields $|w|^2\le4M_*$. Hence, with
$R_0=2\sqrt{M_*}$, one has
$\mathcal C_{t,N}\subset\B(0,R_0)$, proving
\eqref{eq:model-section-bounds}. For finite $t$, the function $q_{t,N}$ is continuous and concave on the convex
set $\Hh_{t,N}$, hence its superlevel set $\mathcal C_{t,N}$ is closed and
convex. By \eqref{eq:model-section-bounds} it is bounded and contains a ball,
so it is a convex body. The same conclusion holds for $t=\infty$ from the
explicit formula $q_{\infty,N}(w)=h_d-|w|^2/2$.

Finally, let $t<\infty$ and choose an orthogonal map $R$ such that $Re_d=N$.
In \eqref{eq:scalar-halfspace-profile}, make the translation
$\zeta=te_d+z$. Then $\zeta\in\Hh$ is equivalent to $z_d\ge-t$, and the
integrand becomes
\[
 \left(\varphi_d(t)+\varphi_d(t+z_d)-\frac{|z|^2}{2}\right)_+
 =(q_{t,e_d}(z))_+.
\]
Hence \eqref{eq:scalar-halfspace-profile} gives
$\int_{\Hh_{t,e_d}}(q_{t,e_d}(z))_+\dd z=1$. Now set $w=Rz$. Since $R$
preserves Lebesgue measure and Euclidean norm, while
$N\cdot w=(Re_d)\cdot(Rz)=z_d$, one has
$R\Hh_{t,e_d}=\Hh_{t,N}$ and $q_{t,N}(Rz)=q_{t,e_d}(z)$. Therefore, by
\eqref{eq:full-boundary-model},
\[
 \int_{\R^d}\kappa_{t,N}(w)\dd w
 =\int_{\Hh_{t,N}}(q_{t,N}(w))_+\dd w
 =\int_{\Hh_{t,e_d}}(q_{t,e_d}(z))_+\dd z
 =1.
\]
For $t=\infty$, one has $\kappa_{\infty,N}=\kappa_d$, so the unit-mass
property is \eqref{eq:cap-normalization}.

We next prove the first identity in \eqref{eq:model-centroid}. For
$\xi\in\Hh$, set
\[
 Q_\xi(\zeta)
 :=\phi_{\Hh,1}(\xi)+\phi_{\Hh,1}(\zeta)-\frac12|\xi-\zeta|^2,
 \qquad \zeta\in\Hh.
\]
By \eqref{eq:halfspace-profile-bounds}, the map
$\zeta\mapsto Q_\xi(\zeta)$ is uniformly concave. Moreover,
\eqref{eq:symmetric-model} gives
$\int_{\Hh}(Q_\xi)_+\dd\zeta=1$, so $Q_\xi$ is positive somewhere. Hence
the zero set $\{Q_\xi=0\}$ is contained in the boundary of the convex
section $\{Q_\xi\ge0\}$ and is therefore Lebesgue null.

We may differentiate \eqref{eq:symmetric-model} with respect to $\xi$.
Indeed, if $\xi'$ ranges in a fixed neighborhood of $\xi$, boundedness of
$\phi_{\Hh,1}$ implies that $Q_{\xi'}(\zeta)>0$ only when
$|\zeta-\xi'|$ is bounded by a constant independent of $\xi'$. Thus all
positive parts are supported in a fixed bounded subset of $\Hh$. Moreover,
\[
 \nabla_\xi Q_\xi(\zeta)=\nabla\phi_{\Hh,1}(\xi)+\zeta-\xi
\]
is uniformly bounded on this set by the gradient bound in
\eqref{eq:halfspace-profile-bounds}. Since $\{Q_\xi=0\}$ is null, the
positive-part function is differentiable in $\xi$ for almost every
$\zeta$, with
$\nabla_\xi(Q_\xi(\zeta))_+
 =\1_{\{Q_\xi>0\}}\nabla_\xi Q_\xi(\zeta)$.
Dominated convergence applied to \eqref{eq:symmetric-model} therefore
gives
\[
 0=\int_{\{Q_\xi>0\}}
 \bigl(\nabla\phi_{\Hh,1}(\xi)+\zeta-\xi\bigr)\dd\zeta .
\]
Taking $\xi=te_d$, using $\phi_{\Hh,1}(te_d)=\varphi_d(t)$ and
$\nabla\phi_{\Hh,1}(te_d)=\varphi_d'(t)e_d$, and then setting
$w=\zeta-te_d$, yields
\begin{equation}\label{eq:profile-marginal-derivative}
 0=\int_{\mathcal C_{t,e_d}}
 \bigl(\varphi_d'(t)e_d+w\bigr)\dd w .
\end{equation}
For $t=0$, the same argument is applied to the right difference quotient
$\xi=se_d$, $s\downarrow0$; the support and derivative bounds above are
uniform for $0\le s\le s_0$, and $\varphi_d'(0)$ is the right derivative
provided by \cref{lem:halfspace-exists}. Since the boundary of
$\mathcal C_{t,e_d}$ is null, \eqref{eq:profile-marginal-derivative} is
unchanged if $\{q_{t,e_d}>0\}$ is replaced by $\mathcal C_{t,e_d}$. Equation \eqref{eq:profile-marginal-derivative} gives
\[
 \frac1{|\mathcal C_{t,e_d}|}
 \int_{\mathcal C_{t,e_d}}w\dd w
 =-\varphi_d'(t)e_d.
\]
If $R$ is orthogonal and $Re_d=N$, then
$R\mathcal C_{t,e_d}=\mathcal C_{t,N}$ by
\eqref{eq:full-boundary-model}. Changing variables $w=Rz$ therefore gives
the first identity in \eqref{eq:model-centroid}. For the second identity, set
$m_{t,N}:=\int_{\R^d}w\kappa_{t,N}(w)\dd w$. If $Q$ is orthogonal and
$QN=N$, then \eqref{eq:full-boundary-model} gives
$\kappa_{t,N}(Qw)=\kappa_{t,N}(w)$. Consequently
$Qm_{t,N}=m_{t,N}$ for every such $Q$, and hence
$m_{t,N}$ is parallel to $N$. Writing $m_{t,N}=aN$ and taking the scalar
product with $N$, \eqref{eq:halfspace-mean-profile} gives
$a=\beta_d(t)$. Thus
$\int_{\R^d}w\kappa_{t,N}(w)\dd w=\beta_d(t)N$, which is the second
identity in \eqref{eq:model-centroid}. At $t=\infty$, one has $\mathcal C_{\infty,N}=\B(0,a_d)$ and
$\kappa_{\infty,N}=\kappa_d$. Since both the ball and the density $\kappa_d$
are invariant under $w\mapsto-w$, 
\[
 \int_{\mathcal C_{\infty,N}} w\,\dd w=0,
 \qquad
 \int_{\R^d} w\,\kappa_{\infty,N}(w)\dd w=0.
\]
Thus the two identities in \eqref{eq:model-centroid} also hold at
$t=\infty$, with $\varphi_d'(\infty)=\beta_d(\infty)=0$.

It remains to verify the stated properties of $\beta_d$. Let
$t_j\to t<\infty$. If $N\cdot w>-t$, then
$t_j+N\cdot w>0$ for all sufficiently large $j$, and continuity of
$\varphi_d$ gives
$q_{t_j,N}(w)\to q_{t,N}(w)$. If $N\cdot w<-t$, then
$w\notin\Hh_{t_j,N}$ for all sufficiently large $j$, so
$\kappa_{t_j,N}(w)=0=\kappa_{t,N}(w)$. Hence $\kappa_{t_j,N}(w)\to\kappa_{t,N}(w)$ for every
$w\notin\partial\Hh_{t,N}=\{N\cdot w=-t\}$ and 
$|\partial\Hh_{t,N}|=0$. By \eqref{eq:model-section-bounds} and
\eqref{eq:model-slack-bounds},
\[
 |w|\,\kappa_{t_j,N}(w)
 \le 2M_*R_0\,\1_{\B(0,R_0)}(w),
\]
uniformly in $j$. Dominated convergence therefore yields convergence of
the first moments, and hence
$\beta_d(t_j)\to\beta_d(t)$ by \eqref{eq:model-centroid}. Thus
$\beta_d$ is continuous on $[0,\infty)$. Moreover, since $\kappa_{t,N}$ has unit mass and is supported on
$\mathcal C_{t,N}\subset\B(0,R_0)$ by
\eqref{eq:model-section-bounds},
$ |\beta_d(t)|
 \le \int_{\R^d}|w|\kappa_{t,N}(w)\dd w
 \le R_0\int_{\R^d}\kappa_{t,N}(w)\dd w
 =R_0.$
Thus $\beta_d$ is bounded.

Finally, for $t=0$ one has
$\mathcal C_{0,N}\subset\{w:N\cdot w\ge0\}$. By
\eqref{eq:model-section-bounds}, on $B(r_0N,r_0/2)$ one has
$N\cdot w\ge r_0/2$ and $q_{0,N}(w)\ge c_0$. Hence
\[
 \int_{\mathcal C_{0,N}}N\cdot w\,\dd w>0,
 \qquad
 \int_{\R^d}(N\cdot w)\kappa_{0,N}(w)\dd w>0.
\]
Using the two identities in \eqref{eq:model-centroid}, these inequalities
give $-\varphi_d'(0)>0$ and $\beta_d(0)>0$.
\end{proof}

In the proof of \cref{prop:profile-exponential} below, we compare the half-space
profile $\varphi_d$ with the whole-space profile $\varphi_d(\infty)=h_d/2$.
The latter has marginal mass one on $\R^d$ by
\eqref{eq:cap-normalization}. Restricting its marginal integral to the
half-space leaves a missing-mass term supported near the boundary. Subtracting
this restricted whole-space equation from
\eqref{eq:scalar-halfspace-profile} yields an equation of the form
$(I+P_1)e=F$, with $e(t)=\varphi_d(t)-h_d/2$. The following lemma gives
uniform bounds for $(I+P_l)^{-1}$ on $L^2$ and $L^\infty$, and an
exponential bound for $(I+P_l)^{-1}F$ in terms of the distance to
$\operatorname{spt}F$. We state it for bounded convex domains, $\R^d$, and
closed half-spaces, since the bounded case is used in \cref{thm:height}. The free parameter $l>0$ will be set to $1$ in this half-space comparison and to $\ell$ in the QOT application.
\begin{lemma}[Uniform invertibility of $I+P_l$]\label{lem:plus}
Let $\Omega\subset\R^d$ be a convex body, the whole space, or a closed half-space. Let $\lambda(\dd x)=\varrho(x)\dd x$ on $\Omega$, where $\varrho$ is measurable and bounded above and away from zero. Let $W_l:\Omega\times\Omega\to[0,\infty)$ be measurable, symmetric a.e.\ and such that
\begin{equation}\label{eq:plus-kernel}
 c\1_{\{|x-z|<al\}}\le W_l(x,z)\le C\1_{\{|x-z|<Rl\}}\qquad\text{for a.e. }(x,z)\in\Omega^2
\end{equation}
for fixed $c,C>0$ and $0<a\le R<\infty$. Define, for almost every $x\in\Omega$,
\begin{equation}\label{eq:plus-operator}
 m_l(x)=\int_\Omega W_l(x,z)\dd\lambda(z),\qquad P_l f(x)=\frac1{m_l(x)}\int_\Omega W_l(x,z)f(z)\dd\lambda(z).
\end{equation}
Then $m_l\asymp l^d$ almost everywhere, and $I+P_l$ has a uniformly bounded inverse on both $L^2(\lambda)$ and $L^\infty(\lambda)$, where $I$ is the identity operator. These assertions hold for all sufficiently small $l$ on a fixed convex body, and for every $l>0$ on the whole space or a half-space. The constants depend only on $d$, the density bounds, $a,R,c,C$, and, in the bounded case, on $\Omega$.

If $F\in L^\infty(\lambda)$ vanishes almost everywhere outside a measurable set $E\subset\Omega$, then
\begin{equation}\label{eq:inverse-tail}
 |(I+P_l)^{-1}F(x)|\le C e^{-c\dist(x,E)/l}\|F\|_{L^\infty(\lambda)}\qquad\text{for }\lambda\text{-a.e. }x\in\Omega.
\end{equation}
For $E=\varnothing$, the right-hand side is interpreted as zero.
\end{lemma}
\begin{proof}
All kernel identities below hold almost everywhere. The volume of $\Omega\cap B(x,r)$ is comparable to $r^d$, uniformly in $x$, for small $r$ on a fixed convex body and for every $r>0$ on the whole space or a half-space. In the bounded case, choose $B(z_*,r_*)\subset\Omega$ and set $D=\diam\Omega$. For $0<r<D$, the map $z\mapsto x+r(z-x)/(2D)$ sends this interior ball into $\Omega\cap B(x,r)$, giving the lower volume bound. Integrating \eqref{eq:plus-kernel} and using the density bounds for $\lambda$ therefore gives $m_l\asymp l^d$.

Write $p_l(x,z)=W_l(x,z)/m_l(x)$ for the transition density of
$P_l$ relative to $\lambda$. By the definition of $m_l$,
$\int_\Omega p_l(x,z)\dd\lambda(z)=1$. Moreover,
\eqref{eq:plus-kernel} and $m_l\asymp l^d$ give
$p_l(x,z)\ge cl^{-d}$ whenever $|x-z|<al$. The transition density of $P_l^2$ is
$p_l^{(2)}(x,z)=\int_\Omega p_l(x,y)p_l(y,z)\dd\lambda(y)$, and
it also satisfies $\int_\Omega p_l^{(2)}(x,z)\dd\lambda(z)=1$. If
$|x-z|<al/4$ and $y\in\Omega\cap B(x,al/4)$, then
$|x-y|<al$ and
$|y-z|\le |y-x|+|x-z|<al/2<al$. Hence both factors in the
definition of $p_l^{(2)}(x,z)$ are bounded below by $cl^{-d}$.
The volume estimate above and the lower density bound for $\lambda$ then
give $p_l^{(2)}(x,z)\ge cl^{-d}$ whenever
$|x-z|<al/4$. Thus, on $\Omega\cap B(x,al/4)$, both
$2p_l(x,z)$ and $p_l^{(2)}(x,z)$ are bounded below by
$cl^{-d}$. Integrating over this set gives a constant
$\eta\in(0,1)$, independent of $l$, such that
\begin{equation}\label{eq:plus-overlap}
 \int_\Omega\min\{2p_l(x,z),p_l^{(2)}(x,z)\}\dd\lambda(z)
 \ge\eta
 \qquad\text{for a.e. }x\in\Omega.
\end{equation}
Set $\mathcal R_l=(I-P_l)/2$. Then
$I+P_l=2(I-\mathcal R_l)$ and
$\mathcal R_l^2=(I-2P_l+P_l^2)/4$. For
$f\in L^\infty(\lambda)$,
\[
 |\mathcal R_l^2f(x)|
 \le\frac14\left(
 |f(x)|+
 \int_\Omega|p_l^{(2)}(x,z)-2p_l(x,z)|
 |f(z)|\dd\lambda(z)\right).
\]
Since $p_l(x,\cdot)$ and $p_l^{(2)}(x,\cdot)$ both have mass one,
the identity $|s-t|=s+t-2\min\{s,t\}$ gives
\[
 \int_\Omega|p_l^{(2)}(x,z)-2p_l(x,z)|\dd\lambda(z)
 =3-2\int_\Omega
 \min\{2p_l(x,z),p_l^{(2)}(x,z)\}\dd\lambda(z).
\]
Using \eqref{eq:plus-overlap}, we obtain
\begin{equation}\label{eq:plus-contraction}
 \|\mathcal R_l^2\|_{L^\infty\to L^\infty}
 \le\frac14(1+3-2\eta)
 =1-\eta/2
 =:\theta<1.
\end{equation}

Now let $\lambda_l(\dd x)=l^{-d}m_l(x)\lambda(\dd x)$.
Since $m_l\asymp l^d$, the measures $\lambda_l$ and $\lambda$
have uniformly comparable densities. For nonnegative measurable $f$, symmetry of $W_l$ and Tonelli's theorem give
\[
 \int_\Omega P_l f(x)\dd\lambda_l(x)
 =l^{-d}\iint_{\Omega^2}
 W_l(x,z)f(z)\dd\lambda(z)\dd\lambda(x)
 =\int_\Omega f(z)\dd\lambda_l(z),
\]
so $\lambda_l$ is invariant for $P_l$. The identity extends to every $f\in L^1(\lambda_l)$. Jensen's inequality and invariance of $\lambda_l$ give $\|P_l f\|_{L^2(\lambda_l)}\le\|f\|_{L^2(\lambda_l)}$. The same symmetry gives
$\langle g,P_l f\rangle_{L^2(\lambda_l)}
=\langle P_l g,f\rangle_{L^2(\lambda_l)}$, so $P_l$ is
self-adjoint on $L^2(\lambda_l)$. Because $\lambda$ and $\lambda_l$ have the same null sets,
\eqref{eq:plus-contraction} is also the corresponding
$L^\infty(\lambda_l)$ bound. Because $P_l$ is self-adjoint on $L^2(\lambda_l)$, so are
$\mathcal R_l=(I-P_l)/2$ and $\mathcal R_l^2$. Hence, for
$f\in L^1(\lambda_l)\cap L^2(\lambda_l)$ and
$g\in L^\infty(\lambda_l)\cap L^2(\lambda_l)$,
\[
 \int_\Omega (\mathcal R_l^2 f)\,g\,\dd\lambda_l
 =
 \int_\Omega f\,(\mathcal R_l^2 g)\,\dd\lambda_l.
\]
Using \eqref{eq:plus-contraction} on the right-hand side gives
\[
 \left|
 \int_\Omega (\mathcal R_l^2 f)\,g\,\dd\lambda_l
 \right|
 \le
 \|f\|_{L^1(\lambda_l)}
 \|\mathcal R_l^2 g\|_{L^\infty(\lambda_l)}
 \le
 \theta\|f\|_{L^1(\lambda_l)}
 \|g\|_{L^\infty(\lambda_l)}.
\]
Taking the supremum over $g$ with
$\|g\|_{L^\infty(\lambda_l)}\le1$, and then extending by density,
yields
$ \|\mathcal R_l^2\|_{L^1(\lambda_l)\to L^1(\lambda_l)}
 \le\theta.$
Together with \eqref{eq:plus-contraction}, the Riesz--Thorin
interpolation theorem \cite[Theorem~1.3.4]{GrafakosClassicalFourierAnalysis} between $L^1(\lambda_l)$ and
$L^\infty(\lambda_l)$ therefore gives
\[
 \|\mathcal R_l^2\|_{L^2(\lambda_l)\to L^2(\lambda_l)}
 \le
 \theta^{1/2}\theta^{1/2}
 =\theta.
\]

The Markov property gives
$\|P_l\|_{L^\infty(\lambda_l)\to L^\infty(\lambda_l)}\le1$.
Therefore $\|\mathcal R_l\|\le1$ on both
$L^\infty(\lambda_l)$ and $L^2(\lambda_l)$. Combining this with
\eqref{eq:plus-contraction} gives, for every $j\ge0$,
$\|\mathcal R_l^j\|\le\theta^{\lfloor j/2\rfloor}$ on both spaces.
Hence the series
\begin{equation}\label{eq:plus-inverse-series}
 (I+P_l)^{-1}
 =\frac12\sum_{j=0}^\infty\mathcal R_l^j
\end{equation}
converges in both operator norms, with a bound independent of $l$.
Since $\lambda_l$ and $\lambda$ are uniformly comparable, this gives
the asserted $L^2(\lambda)$ and $L^\infty(\lambda)$ bounds.

It remains to prove \eqref{eq:inverse-tail}. If $G$ vanishes almost
everywhere outside a measurable set $A\subset\Omega$, then the upper
bound in \eqref{eq:plus-kernel} gives
$P_l G(x)=0$ for almost every $x$ with
$\dist(x,A)>Rl$. Since $G$ also vanishes there,
$\mathcal R_l G(x)=0$ on the same set. Iterating this statement with
$A=\{x:\dist(x,E)\le jRl\}$ gives
\[
 \mathcal R_l^jF(x)=0
 \qquad\text{for a.e. }x\text{ with }\dist(x,E)>jRl.
\]
After removing the countable union of the corresponding null sets, this
holds simultaneously for every $j\ge0$. Fix such an $x$. In \eqref{eq:plus-inverse-series}, only indices
$j\ge \dist(x,E)/(Rl)$ can contribute. Hence, for $\lambda$-almost every $x\in\Omega$, 
\[
 |(I+P_l)^{-1}F(x)|
 \le\frac12\|F\|_{L^\infty(\lambda)}
 \sum_{j\ge \dist(x,E)/(Rl)}
 \theta^{\lfloor j/2\rfloor} \leq   C e^{-c\,\dist(x,E)/l}\|F\|_{L^\infty(\lambda)},
\]
where we used $\theta\in(0,1)$. This proves \eqref{eq:inverse-tail}.
\end{proof}
\Needspace{8\baselineskip}
\begin{proposition}[Exponential decay of the boundary correction]\label{prop:profile-exponential}
There are constants $c,C>0$, depending only on $d$, such that
\begin{equation}\label{eq:profile-exponential}
 |\varphi_d(t)-h_d/2|+|\varphi_d'(t)|+|\beta_d(t)|\le Ce^{-ct}\qquad(t\ge0).
\end{equation}
In particular,
\begin{equation}\label{eq:profile-deep-limit}
 \lim_{t\to\infty}\varphi_d(t)=h_d/2,\qquad \lim_{t\to\infty}\varphi_d'(t)=0.
\end{equation}
Moreover,
\begin{equation}\label{eq:beta-sum-rule}
 \int_0^\infty\beta_d(t)\dd t=\frac{h_d}{d+4}.
\end{equation}
\end{proposition}
\begin{proof}
Set $e(t)=\varphi_d(t)-h_d/2$. We first compare
\eqref{eq:scalar-halfspace-profile} with the marginal equation for the
whole-space profile $\varphi_d(\infty)=h_d/2$. Fix $t\ge0$. In \eqref{eq:scalar-halfspace-profile}, the first point is
$\xi=te_d\in\Hh$, while the integration variable is
$\zeta=(z',s)\in\R^{d-1}\times[0,\infty)$. Thus
$\zeta_d=s$ and
$|\xi-\zeta|^2=|z'|^2+(s-t)^2$. Since
$\varphi_d(t)=h_d/2+e(t)$ and
$\varphi_d(s)=h_d/2+e(s)$,
\eqref{eq:scalar-halfspace-profile} becomes
\[
 1=\int_0^\infty\int_{\R^{d-1}}
 \left(
 h_d+e(t)+e(s)-\frac{(t-s)^2+|z'|^2}{2}
 \right)_+
 \dd z'\dd s.
\]
For real numbers $A$ and $B$, the identity
$(A+B)_+-A_+=B\int_0^1\1_{\{A+\theta B>0\}}\dd\theta$ gives, with
$A=h_d-[(t-s)^2+|z'|^2]/2$ and $B=e(t)+e(s)$,
\begin{equation}\label{eq:elementary-ident}
\begin{aligned}
 &\left(h_d+e(t)+e(s)-\frac{(t-s)^2+|z'|^2}{2}\right)_+
 -\left(h_d-\frac{(t-s)^2+|z'|^2}{2}\right)_+\\
 &\qquad=[e(t)+e(s)]\int_0^1
 \1_{\{h_d+\theta[e(t)+e(s)]-[(t-s)^2+|z'|^2]/2>0\}}\dd\theta.
\end{aligned}
\end{equation}
Integrating the factor multiplying $e(t)+e(s)$ in $z'$ defines
\begin{equation}\label{eq:profile-comparison-kernel}
 W(t,s)=\int_{\R^{d-1}}\int_0^1
 \1_{\{h_d+\theta[e(t)+e(s)]-[(t-s)^2+|z'|^2]/2>0\}}
 \dd\theta\dd z'.
\end{equation}

The kernel $W$ in \eqref{eq:profile-comparison-kernel} is symmetric. Since $e(r)=\varphi_d(r)-h_d/2$,
\[
 h_d+\theta[e(t)+e(s)]=(1-\theta)h_d+\theta[\varphi_d(t)+\varphi_d(s)].
\]
By \cref{lem:profile-regularity}, this quantity is bounded above and away from zero uniformly in $t,s\ge0$ and $\theta\in[0,1]$. We may therefore choose $0<a<b<\infty$, depending only on $d$, so that the indicator in \eqref{eq:profile-comparison-kernel} equals one when $|t-s|<a$ and $|z'|<a$, and vanishes when $|(t-s,z')|\ge b$. Integrating in $z'$ and $\theta$ gives, for constants $c,C>0$ depending only on $d$,
\begin{equation}\label{eq:profile-kernel-bounds}
 c\1_{\{|t-s|<a\}}\le W(t,s)\le C\1_{\{|t-s|<b\}},
 \qquad t,s\ge0.
\end{equation}
Define $m(t)=\int_0^\infty W(t,s)\dd s$. Integrating \eqref{eq:profile-kernel-bounds} in $s$ gives
\[
 ca\le m(t)\le2Cb\qquad(t\ge0),
\]
because $[0,\infty)\cap(t-a,t+a)$ has length at least $a$, while $[0,\infty)\cap(t-b,t+b)$ has length at most $2b$. Thus $m\asymp1$ uniformly on $[0,\infty)$.

We next identify the equation satisfied by $e$. By
\eqref{eq:cap-normalization}, the whole-space density is
$\kappa_d(z)=(h_d-|z|^2/2)_+$ and satisfies
$\int_{\R^d}\kappa_d(z)\dd z=1$. Fix $t\ge0$ and write a point of
$\R^d$ as $\zeta=(z',s)$. Translating by the fixed point $te_d$ gives
\[
 1=\int_{-\infty}^{\infty}\int_{\R^{d-1}}
 \left(h_d-\frac{|z'|^2+(s-t)^2}{2}\right)_+\dd z'\dd s.
\]
Splitting this integral into $s\ge0$ and $s<0$, and setting
$z_d=s-t$ in the second part, gives
\begin{equation}
    \label{eq:Rd-1-toF} \int_0^\infty\int_{\R^{d-1}}
 \left(h_d-\frac{|z'|^2+(s-t)^2}{2}\right)_+\dd z'\dd s
 =1-F(t),
\end{equation}
where $ F(t):=\int_{\{z_d<-t\}}\kappa_d(z)\dd z.$ 
On the other hand, \eqref{eq:scalar-halfspace-profile} gives
\[
 \int_0^\infty\int_{\R^{d-1}}
 \left(h_d+e(t)+e(s)
 -\frac{|z'|^2+(s-t)^2}{2}\right)_+\dd z'\dd s=1.
\]
Subtracting \eqref{eq:Rd-1-toF} 
from this identity therefore gives
\[
 F(t)
 =
 \int_0^\infty\int_{\R^{d-1}}
 \left[
 \left(h_d+e(t)+e(s)
 -\frac{|z'|^2+(s-t)^2}{2}\right)_+
 -
 \left(h_d-\frac{|z'|^2+(s-t)^2}{2}\right)_+
 \right]\dd z'\dd s.
\]
Integrating in $z'$, using the definition \eqref{eq:profile-comparison-kernel} of $W(t,s)$ and \eqref{eq:elementary-ident}, gives 
\begin{equation}\label{eq:profile-linearization}
 m(t)e(t)+\int_0^\infty W(t,s)e(s)\dd s=\int_0^\infty W(t,s)[e(t)+e(s)]\dd s=F(t).
\end{equation}
By \eqref{eq:cap-normalization}, $\kappa_d\ge0$ has total mass one, so
$0\le F(t)\le1$. Moreover,
$\supp\kappa_d=\overline{B(0,a_d)}$. If $t\ge a_d$, then
$\{z_d<-t\}\cap B(0,a_d)=\varnothing$, and hence $F(t)=0$.

Dividing \eqref{eq:profile-linearization} by $m(t)$ gives
$(I+P)e=F/m$, where
$Pe(t)=m(t)^{-1}\int_0^\infty W(t,s)e(s)\dd s$. The bounds
\eqref{eq:profile-kernel-bounds} verify the hypotheses of
\cref{lem:plus} with $l=1$. Since $m\asymp1$, the function $F/m$ is
bounded and is supported in $[0,a_d]$. Hence
\eqref{eq:inverse-tail} gives
\begin{equation}\label{eq:profile-height-decay}
 |e(t)|\le Ce^{-ct}\qquad(t\ge0),
\end{equation}
first for almost every $t$. Since $e$ is continuous by
\cref{lem:profile-regularity}, the estimate holds for every $t\ge0$.

We next estimate $e'=\varphi_d'$. By
\cref{lem:profile-regularity}, $e'$ is Lipschitz. For $t\ge1$ and
$0<r\le1$, Taylor's formula  gives
$|e(t+r)-e(t)-re'(t)|\le \frac12\Lip(e')r^2$, and therefore
\[
 |e'(t)|
 \le\frac{|e(t+r)|+|e(t)|}{r}+\frac r2\Lip(e')
 \le\frac{2}{r}\|e\|_{L^\infty([t-1,t+1])}
      +\frac r2\Lip(e').
\]
By \eqref{eq:profile-height-decay},
$\|e\|_{L^\infty([t-1,t+1])}\le Ce^{-ct}$ after changing $C$.
Taking $r=e^{-ct/2}$ yields
$|\varphi_d'(t)|=|e'(t)|\le Ce^{-ct/2}$ for $t\ge1$. Boundedness of
$\varphi_d'$ on $[0,1]$ follows from
\cref{lem:profile-regularity}. After decreasing $c$, this proves the
 exponential bound \eqref{eq:profile-exponential} for $\varphi_d'$ on $[0,\infty)$.

We now estimate $\beta_d$. Choose
$R_1>\max\{R_0,a_d\}$, where $R_0$ is given by
\eqref{eq:model-section-bounds}. If $t>R_1$ and $z\in B(0,R_1)$, then
$z_d>-R_1>-t$, so $B(0,R_1)\subset\Hh_{t,e_d}$. Moreover,
\eqref{eq:model-section-bounds} gives
$\supp\kappa_{t,e_d}\subset\B(0,R_0)\subset B(0,R_1)$, while
\eqref{eq:cap-normalization} gives
$\supp\kappa_d=\B(0,a_d)\subset B(0,R_1)$. By
\eqref{eq:halfspace-mean-profile},
$\beta_d(t)=\int z_d\kappa_{t,e_d}(z)\dd z$, whereas radial symmetry of
$\kappa_d$ gives $\int z_d\kappa_d(z)\dd z=0$. Hence
\[
 |\beta_d(t)|
 \le\int_{B(0,R_1)}
 |z_d|\,|\kappa_{t,e_d}(z)-\kappa_d(z)|\dd z.
\]
On $B(0,R_1)$ one has
$\kappa_{t,e_d}(z)=
(h_d+e(t)+e(t+z_d)-|z|^2/2)_+$ and
$\kappa_d(z)=(h_d-|z|^2/2)_+$. Since $r\mapsto r_+$ is
$1$-Lipschitz,
\[
 |\beta_d(t)|
 \le\int_{B(0,R_1)}
 |z_d|\bigl(|e(t)|+|e(t+z_d)|\bigr)\dd z
 \le C\sup_{|s-t|\le R_1}|e(s)|.
\]
For $t>R_1$, all $s$ in the last supremum are nonnegative, so
\eqref{eq:profile-height-decay} gives
$|\beta_d(t)|\le Ce^{-ct}$ after changing $c$ and $C$.
Boundedness of $\beta_d$ on $[0,R_1]$ follows from
\cref{lem:profile-regularity}. This proves
\eqref{eq:profile-exponential}, after decreasing $c$ once more, and
\eqref{eq:profile-deep-limit} follows immediately.

It remains to compute $\int_0^\infty\beta_d(t)\dd t$. Fix $t\ge0$ and
write the integration variable in \eqref{eq:scalar-halfspace-profile} as
$\zeta=(z',s)\in\R^{d-1}\times[0,\infty)$. For this fixed $t$, the
function
\[
 (z',s)\longmapsto
 \left(\varphi_d(t)+\varphi_d(s)
 -\tfrac12[(t-s)^2+|z'|^2]\right)_+
\]
has integral one over $\R^{d-1}\times[0,\infty)$ by
\eqref{eq:scalar-halfspace-profile}. For each fixed $t\ge0$,  its marginal density $k_d(t,\cdot)$ is
\begin{equation}\label{eq:normal-profile-kernel}
 k_d(t,s)=\int_{\R^{d-1}}
 \left(\varphi_d(t)+\varphi_d(s)
 -\tfrac12[(t-s)^2+|z'|^2]\right)_+\dd z',
 \qquad s\ge0.
\end{equation}
For fixed $t$, the identity
\eqref{eq:scalar-halfspace-profile} gives
$\int_0^\infty k_d(t,s)\dd s=1$. The formula
\eqref{eq:normal-profile-kernel} is symmetric in $t$ and $s$, so
$k_d(t,s)=k_d(s,t)$. The bounds from
\cref{lem:profile-regularity} also imply that $k_d$ is uniformly bounded
and that $k_d(t,s)=0$ whenever $|t-s|>b$, for some fixed
$b<\infty$. Finally, since
$ (z',s)-te_d=(z',s-t),$
\eqref{eq:halfspace-mean-profile} gives
$ \beta_d(t)=\int_0^\infty (s-t)k_d(t,s)\dd s.$ Integrating this identity over $t\in[0,L]$ gives an integral over
$[0,L]\times[0,\infty)$. On $[0,L]^2$, the function
$(s-t)k_d(t,s)$ is antisymmetric under interchange of $t$ and $s$, since
$k_d$ is symmetric. Its integral over this square is therefore zero, yielding
\begin{equation}\label{eq:normal-flux-identity}
 \int_0^L\beta_d(t)\dd t
 =\int_0^L\int_L^\infty
 (s-t)k_d(t,s)\dd s\dd t.
\end{equation}

Assume $L>b$ and set $t=L-u$ and $s=L+v$. Then
$u\in[0,L]$, $v\ge0$, and $s-t=u+v$. Since
$k_d(t,s)=0$ for $|t-s|>b$, only
$u,v\ge0$ with $u+v\le b$ contribute. Because $L>b$, this condition
already implies $u<L$. Hence the right-hand side of
\eqref{eq:normal-flux-identity} is
\[
 \int_{\substack{u,v\ge0\\u+v\le b}}
 (u+v)k_d(L-u,L+v)\dd v\dd u.
\]
For fixed $u,v$ in this triangle,
\eqref{eq:profile-deep-limit} gives
\[
 k_d(L-u,L+v)
 \longrightarrow
 \int_{\R^{d-1}}
 \left(h_d-\tfrac12[(u+v)^2+|z'|^2]\right)_+\dd z'
 =
 \int_{\R^{d-1}}\kappa_d(z',u+v)\dd z'.
\]
The triangle is bounded and $k_d$ is uniformly bounded, so dominated
convergence gives
\[
 \lim_{L\to\infty}\int_0^L\beta_d(t)\dd t
 =
 \int_{\substack{u,v\ge0\\u+v\le b}}
 (u+v)
 \int_{\R^{d-1}}\kappa_d(z',u+v)\dd z'\dd v\dd u.
\]
Set $r=u+v$. For fixed $r$, the variable $u$ ranges over $[0,r]$, so
integration in $u$ contributes a factor $r$. Since the factor
$s-t=u+v$ contributes another factor $r$, the last expression equals
\[
 \int_0^\infty r^2
 \int_{\R^{d-1}}\kappa_d(z',r)\dd z'\dd r.
\]
The upper limit may be written as $\infty$ because
$\kappa_d(z',r)=0$ for $r>a_d$. Since $\kappa_d$ is even in its last
coordinate,
\[
 \int_0^\infty r^2
 \int_{\R^{d-1}}\kappa_d(z',r)\dd z'\dd r
 =
 \frac12\int_{\R^d}z_d^2\kappa_d(z)\dd z.
\]

By radial symmetry,
$\int_{\R^d}z_d^2\kappa_d(z)\dd z
=d^{-1}\int_{\R^d}|z|^2\kappa_d(z)\dd z$. Using
$\kappa_d(z)=(h_d-|z|^2/2)_+$,
$a_d^2=2h_d$, and polar coordinates gives
\begin{equation}\label{eq:model-cap-second-moment}
 \int_{\R^d}z_d^2\kappa_d(z)\dd z
 =\frac{\omega_da_d^{d+4}}{(d+2)(d+4)}
 =\frac{2h_d}{d+4}.
\end{equation}
By \eqref{eq:profile-exponential},
$\int_0^\infty|\beta_d(t)|\dd t<\infty$. Hence the left-hand side of
\eqref{eq:normal-flux-identity} converges to
$\int_0^\infty\beta_d(t)\dd t$ as $L\to\infty$. Combining the preceding
limit with \eqref{eq:model-cap-second-moment} gives \eqref{eq:beta-sum-rule}.
\end{proof}

\section{Graph inclusion and rescaled profiles}\label{sec:profiles}
In this section, we prove the inclusion of the Brenier graph in the support of the regularized coupling $\pi_\eps$. To that end, we apply \cref{prop:classification} to the rescaled potentials analyzed below. Recall $\delta,\mathcal U,n,\gamma,t_\eps$ from \eqref{eq:boundary-coordinates}. For $x\in X$,  define the unit vector
$ N(x)=\frac{H(x)^{-1/2}n(x)}{\gamma(x)}$, recall $\Lambda_x=\sigma(x)H(x)^{-1/2}$ from \eqref{eq:rescaling-matrices} and define $\Omega_{\eps,x}=\ell^{-1}\Lambda_x^{-1}(X-x)$ and $x_{\eps,\xi}=x+\ell\Lambda_x\xi$ for $\xi\in\Omega_{\eps,x}$. We will show that the behavior of the rescaled potentials is governed by
\begin{equation}\label{eq:microscopic-comparison}
 \Phi_{\eps,x}(\xi)=\begin{cases}\varphi_d(t_\eps(x)+N(x)\cdot\xi),&x\in\mathcal U,\\ h_d/2,&x\notin\mathcal U.\end{cases}
\end{equation}
To see that \eqref{eq:microscopic-comparison} is well-defined, fix $x\in\mathcal U$ and let $p=x-\delta(x)n(x)$ be its nearest point on $\partial X$. Convexity gives $n(x)\cdot(z-p)\ge0$ for every $z\in X$. Substituting $z=x+\ell\Lambda_x\xi$ and using $n(x)\cdot\Lambda_x\xi=\sigma(x)\gamma(x)N(x)\cdot\xi$ then yields $\Omega_{\eps,x}\subset\Hh_{t_\eps(x),N(x)}$.

\subsection{Potential profiles and graph inclusion}
Throughout this section, $o_{C^1}(1)$ denotes a remainder whose value and first derivatives in the rescaled variables $\xi$ and/or $\zeta$ tend to zero uniformly on the indicated sets, uniformly in  $x\in X$. (We hold $x$ fixed when differentiating.) The notation $O_{C^1}(\ell)$ means that value and derivatives are bounded by $C\ell$. 

\begin{theorem}[Uniform blow-up limits]\label{thm:microscopic}
For every $M>0$, uniformly for $x\in X$ and $\xi,\zeta\in\Omega_{\eps,x}\cap\B(0,M)$,
\begin{equation}\label{eq:two-point-profile}
 \frac{q_\eps(x_{\eps,\xi},T(x_{\eps,\zeta}))}{\ell^2\sigma(x)^2}
 =\Phi_{\eps,x}(\xi)+\Phi_{\eps,x}(\zeta)-\frac12|\xi-\zeta|^2+o_{C^1}(1).
\end{equation}
Furthermore, the centered potentials satisfy
\begin{align}\label{eq:centered-potentials}
 \frac{u_\eps(x_{\eps,\xi})-u_\eps(x)-T(x)\cdot(x_{\eps,\xi}-x)}{\ell^2\sigma(x)^2}
 &=\tfrac12|\xi|^2-\Phi_{\eps,x}(\xi)+\Phi_{\eps,x}(0)+o_{C^1}(1),\\
 \frac{v_\eps(T(x_{\eps,\xi}))-v_\eps(T(x))-x\cdot[T(x_{\eps,\xi})-T(x)]}{\ell^2\sigma(x)^2}
 &=\tfrac12|\xi|^2-\Phi_{\eps,x}(\xi)+\Phi_{\eps,x}(0)+o_{C^1}(1).\nonumber
\end{align}
In particular,
\begin{equation}\label{eq:uniform-graph-profile}
 \sup_{x\in X}\left|\frac{q_\eps(x,T(x))}{\ell^2\sigma(x)^2}-2\varphi_d(t_\eps(x))\right|\longrightarrow0.
\end{equation}
\end{theorem}
\begin{proof}
To prove uniformity in the base point, fix arbitrary sequences $\eps_j\downarrow0$ and $x_j\in X$. Write $\ell_j=\eps_j^{1/(d+2)}$, $\sigma_j=\sigma(x_j)$, $\Lambda_j=\Lambda_{x_j}$, $\Omega_j=\Omega_{\eps_j,x_j}$ and $x_{j,\xi}=x_j+\ell_j\Lambda_j\xi$. On $\Omega_j$, define
\begin{equation}\label{eq:rescaled-errors}
 A_j(\xi)=-\frac{r_{\eps_j}(x_{j,\xi})}{\ell_j^2\sigma_j^2},\qquad
 B_j(\xi)=-\frac{s_{\eps_j}(T(x_{j,\xi}))}{\ell_j^2\sigma_j^2}.
\end{equation}
The bounds \eqref{eq:G-inputs}, \eqref{eq:curvature-input} and \eqref{eq:pregrad} imply
\begin{equation}\label{eq:rescaled-error-bounds}
 \|A_j\|_\infty+\|B_j\|_\infty+\|\nabla A_j\|_\infty+\|\nabla B_j\|_\infty+\Lip(\nabla A_j)+\Lip(\nabla B_j)\le C,
\end{equation}
where all norms are on $\Omega_j$. For $B_j$, use $T\in C^2(X)$ and the chain rule
\[
 D^2(s_{\eps_j}\circ T)=(DT)^T(D^2s_{\eps_j}\circ T)DT+\sum_{k=1}^d(\partial_k s_{\eps_j}\circ T)D^2T_k
\]
almost everywhere on $X$. The terms on the right are uniformly bounded. The bi-Lipschitz properties of $T,S$ preserve null sets under the change of variables and justify the chain rule.

\emph{Domain convergence and compactness.} Pass to a subsequence along which $x_j$ converges, $t_j:=t_{\eps_j}(x_j)\to t\in[0,\infty]$, and $N(x_j)\to N$. If $t=\infty$, then $\delta(x_j)/\ell_j\to\infty$, so every fixed ball eventually lies in $\Omega_j$. For $t<\infty$, we have $\delta(x_j)=O(\ell_j)$. To identify the limiting boundary, let $\widetilde\delta$ be the signed distance to $\partial X$, positive inside $X$. It is $C^2$ in a fixed neighborhood of $\partial X$, so Taylor's formula gives
\begin{equation}\label{eq:rescaled-boundary-chart}
 \frac{\widetilde\delta(x_j+\ell_j\Lambda_j\xi)}{\ell_j\sigma_j\gamma(x_j)}
 =t_j+N(x_j)\cdot\xi+O(\ell_j|\xi|^2)
\end{equation}
uniformly on each fixed ball. Write $\Omega=\Hh_{t,N}$, with $\Hh_{\infty,N}=\R^d$. In both cases, every compact subset of $\intr\Omega$ eventually lies in $\Omega_j$, and $\1_{\Omega_j}\to\1_\Omega$ almost everywhere. For finite $t$, these facts follow from \eqref{eq:rescaled-boundary-chart}. Every limit of points in $\Omega_j$ belongs to $\Omega$, and every point of $\Omega$ is approximated by such points.
By \eqref{eq:rescaled-error-bounds} and Arzel\`a--Ascoli, a subsequence of $(A_j,B_j)$ converges in $C^1$ on compact subsets of $\intr\Omega$ to bounded functions $A,B$ with bounded $C^{1,1}$-norm. Hence, the limiting functions and gradients extend continuously to $\Omega$.

Along this subsequence, convergence extends to varying points: if $\xi_j\in\Omega_j$ and $\xi_j\to\xi\in\Omega$,
\begin{equation}\label{eq:conv-varying-points}
 A_j(\xi_j)\to A(\xi),\qquad
 \nabla A_j(\xi_j)\to\nabla A(\xi),\qquad
 B_j(\xi_j)\to B(\xi),\qquad
 \nabla B_j(\xi_j)\to\nabla B(\xi).
\end{equation}
Indeed, if $\xi\in\partial\Omega$, fix $s>0$. Then $\xi+sN\in\intr\Omega$ and hence $\xi+sN\in\Omega_j$ for all sufficiently large $j$. The bounds \eqref{eq:rescaled-error-bounds} give
$
 |A_j(\xi_j)-A_j(\xi+sN)|
 +|\nabla A_j(\xi_j)-\nabla A_j(\xi+sN)|
 \le C|\xi_j-\xi-sN|,
$
and analogously for $B_j$. Letting $j\to\infty$ and using the interior $C^1$ convergence, then letting $s\downarrow0$, proves the assertion. %

\emph{Limit marginal equations.} Normalize the unregularized gap and source density by setting $ \mathcal D_j(\xi,\zeta)=\frac{D(x_{j,\xi},x_{j,\zeta})}{\ell_j^2\sigma_j^2},$ and 
$a_j(\xi)=\frac{\rho_0(x_{j,\xi})}{\rho_0(x_j)}.$ 
Taylor's formula and $\Lambda_j^TH(x_j)\Lambda_j=\sigma_j^2\Id$ give
\begin{equation}\label{eq:rescaled-cost-limit}
 \mathcal D_j(\xi,\zeta)=\tfrac12|\xi-\zeta|^2+o_{C^1}(1),\qquad a_j(\xi)=1+o(1)
\end{equation}
uniformly for bounded $\xi,\zeta\in\Omega_j$. For the first derivatives, use $\nabla_1D(x,z)=T(x)-T(z)$ and $\nabla_2D(x,z)=H(z)(z-x)$, where the subscripts indicate the differentiated variable. These identities follow from \eqref{eq:D}. Since $\sigma_j^2\det\Lambda_j=1/\rho_0(x_j)$ and $\eps_j=\ell_j^{d+2}$, changing variables in \eqref{eq:Q-marginals} gives
\begin{equation}\label{eq:rescaled-marginals}
\begin{aligned}
 \int_{\Omega_j}\bigl(A_j(\xi)+B_j(\zeta)-\mathcal D_j(\xi,\zeta)\bigr)_+a_j(\zeta)\dd\zeta&=1,\\
 \int_{\Omega_j}\bigl(A_j(\zeta)+B_j(\xi)-\mathcal D_j(\zeta,\xi)\bigr)_+a_j(\zeta)\dd\zeta&=1
\end{aligned}
\qquad(\xi\in\Omega_j).
\end{equation}
Fix $\xi\in\intr\Omega$, so that $\xi\in\Omega_j$ for all sufficiently large $j$. Extend the integrands in \eqref{eq:rescaled-marginals} by zero outside $\Omega_j$. The uniform bounds in \eqref{eq:rescaled-error-bounds}, the density bounds, and the quadratic lower bound from \eqref{eq:common-gap-bounds} give a common majorant $C\1_{B(\xi,R)}$ on $\R^d$, with $C,R$ independent of $j$. Interior convergence and $\1_{\Omega_j}\to\1_\Omega$ almost everywhere therefore allow dominated convergence in both marginal equations. Hence $(A,B)$ satisfies \eqref{eq:model-pair} with $\tau=1$ on $\intr\Omega$. Boundedness and continuity of $A,B$, together with a common bounded integration region when $\xi$ ranges over a bounded set, make both limiting marginal integrals continuous in $\xi$ on $\Omega$, so the identities extend to $\partial\Omega$. \Cref{prop:classification} now gives $A=\phi_{\Omega,1}+a$ and $B=\phi_{\Omega,1}-a$ for a constant $a\in\R$. Here $\phi_{\Omega,1}(\xi)=\varphi_d(t+N\cdot\xi)$ for $t<\infty$, and $\phi_{\Omega,1}=h_d/2$ for $t=\infty$. Continuity of $\varphi_d,\varphi_d'$ and \eqref{eq:profile-deep-limit} give convergence of $\Phi_{\eps_j,x_j}$ and $\nabla\Phi_{\eps_j,x_j}$ to $\phi_{\Omega,1}$ and $\nabla\phi_{\Omega,1}$, respectively, along every sequence $\xi_j\in\Omega_j$ converging to a point of $\Omega$. Together with convergence~\eqref{eq:conv-varying-points} at varying points, a compactness contradiction yields 
\begin{equation}\label{eq:classified-errors}
 A_j-\Phi_{\eps_j,x_j}=a+o_{C^1}(1),\qquad B_j-\Phi_{\eps_j,x_j}=-a+o_{C^1}(1)
\end{equation}
uniformly on $\Omega_j\cap\B(0,M)$ for every $M>0$. 
The identity $q_\eps(x,T(z))=-r_\eps(x)-s_\eps(T(z))-D(x,z)$ follows from \eqref{eq:gap-slack-identity} and \eqref{eq:D}. The opposite constants in \eqref{eq:classified-errors} cancel in this sum, so \eqref{eq:rescaled-cost-limit} proves \eqref{eq:two-point-profile} along the selected subsequence. Every pair of initial sequences admits a subsequence with this limit, proving the uniformity in \eqref{eq:two-point-profile}. By the definition of $q_\eps$, the centered formulas \eqref{eq:centered-potentials} follow by subtracting \eqref{eq:two-point-profile} at $(\xi,0)$ and $(0,\xi)$, respectively, from its value at $(0,0)$. Evaluating at $(0,0)$ also gives \eqref{eq:uniform-graph-profile}.
\end{proof}

\begin{corollary}
    There are positive constants $c_0,C_0,c_1,r,R,\eps_0$ such that, for $0<\eps<\eps_0$,
\begin{equation}\label{eq:graph-positive}
 c_0\ell^2\le q_\eps(x,T(x))\le C_0\ell^2\qquad(x\in X),
\end{equation}
and
\begin{align}\label{eq:tubes}
 Y\cap B(T(x),r\ell)&\subset\{y\in Y:q_\eps(x,y)\ge c_1\ell^2\}\subset S_{\eps,x}\subset Y\cap\B(T(x),R\ell)\quad (x\in X),\\
 X\cap B(S(y),r\ell)&\subset\{x\in X:q_\eps(x,y)\ge c_1\ell^2\}\subset K_{\eps,y}\subset X\cap\B(S(y),R\ell) \quad (y\in Y).\nonumber
\end{align}
In particular, $\operatorname{graph}T\subset\{q_\eps>0\}\subset\spt\pi_\eps$. Moreover, for positive constants $c,C$,
\begin{equation}\label{eq:density-tube}
 c\ell^{-d}\1_{\{|y-T(x)|<r\ell\}}\le h_\eps(x,y)\le C\ell^{-d}\1_{\{|y-T(x)|\le R\ell\}}\qquad((x,y)\in X\times Y).
\end{equation}
\end{corollary}

\begin{proof}
The positive lower and upper bounds for $\varphi_d$ in \cref{lem:profile-regularity}, the bounds for $\sigma$, and \eqref{eq:uniform-graph-profile} give \eqref{eq:graph-positive}. Fix $x\in X$ and put $y_0=T(x)$. Integrate $\nabla_yq_\eps(x,\cdot)=x-\nabla v_\eps$ along $[y_0,y]\subset Y$. Along this segment, write $x-\nabla v_\eps=[S(y_0)-S]+[S-\nabla v_\eps]$. The Lipschitz bound for $S$ and the gradient estimate \eqref{eq:pregrad} yield
\begin{equation}\label{eq:slack-variation}
 |q_\eps(x,y)-q_\eps(x,y_0)|\le C\ell|y-y_0|+\tfrac12\Lip(S)|y-y_0|^2.
\end{equation}
Using \eqref{eq:graph-positive}, and choosing $r>0$ in  \eqref{eq:slack-variation} so that $Cr+\Lip(S)r^2/2\le c_0/2$, proves the first inner inclusion in \eqref{eq:tubes} with $c_1=c_0/2$. The outer inclusions follow from \eqref{eq:prelocalization}. Finally,  \eqref{eq:height-upper} proves the density bounds \eqref{eq:density-tube} and \eqref{eq:preliminary-support} shows $\operatorname{graph}T\subset\spt\pi_\eps$. 
\end{proof}

\subsection{Sections and conditional laws}
We center the sections at $T(x)$ and $x$, respectively, using $B_x$ and $\Lambda_x$ to rescale:
\begin{equation}\label{eq:normalized-sections}
 \mathcal S_{\eps,x}=\ell^{-1}B_x^{-1}(S_{\eps,x}-T(x)),\qquad
 \mathcal K_{\eps,x}=\ell^{-1}\Lambda_x^{-1}(K_{\eps,T(x)}-x).
\end{equation}
In what follows, the superscripts $\rm r,\rm c$ indicate row conditioning (fixing the source $x$) and column conditioning (fixing the target $T(x)$). Let $f^{\rm r}_{\eps,x}$ and $f^{\rm c}_{\eps,x}$ be the Lebesgue densities of $(y\mapsto\ell^{-1}B_x^{-1}(y-T(x)))_\#\pi_\eps^x$ and $(z\mapsto\ell^{-1}\Lambda_x^{-1}(z-x))_\#\widehat\pi_\eps^{T(x)}$. The densities are defined on $\Omega^{\rm r}_{\eps,x}=\ell^{-1}B_x^{-1}(Y-T(x))$ and $\Omega_{\eps,x}$, respectively, and are extended by zero outside these domains. Their supports are $\mathcal S_{\eps,x}$ and $\mathcal K_{\eps,x}$.  Define
\begin{equation}\label{eq:rescaled-slacks}
 Q^{\rm r}_{\eps,x}(w)=\frac{q_\eps(x,T(x)+\ell B_xw)}{\ell^2\sigma(x)^2},\qquad
 Q^{\rm c}_{\eps,x}(w)=\frac{q_\eps(x+\ell\Lambda_xw,T(x))}{\ell^2\sigma(x)^2}
\end{equation}
on $\Omega^{\rm r}_{\eps,x}$ and $\Omega_{\eps,x}$, respectively. A family $E_\eps\subset X$ is called a \emph{bulk family} if $\inf_{x\in E_\eps}\delta(x)/\ell\to\infty$, with $\inf\varnothing=\infty$. Thus the points may approach the boundary, but their distances remain large compared with $\ell$. In particular, a fixed interior compact set is a bulk family. 

\begin{theorem}[Profiles of sections and conditional laws]\label{thm:section-laws}
Uniformly for $x\in X$,
\begin{equation}\label{eq:all-boundary-sections}
 d_H(\mathcal S_{\eps,x},\mathcal C_{t_\eps(x),N(x)})+d_H(\mathcal K_{\eps,x},\mathcal C_{t_\eps(x),N(x)})\longrightarrow0,
\end{equation}
and, for every $1\le p<\infty$,
\begin{equation}\label{eq:all-boundary-densities}
 \|f^{\rm r}_{\eps,x}-\kappa_{t_\eps(x),N(x)}\|_{L^p(\R^d)}+\|f^{\rm c}_{\eps,x}-\kappa_{t_\eps(x),N(x)}\|_{L^p(\R^d)}\longrightarrow0.
\end{equation}
The probability measures $(y\mapsto\ell^{-1}B_x^{-1}(y-T(x)))_\#\pi_\eps^x$ and $(z\mapsto\ell^{-1}\Lambda_x^{-1}(z-x))_\#\widehat\pi_\eps^{T(x)}$ converge to $\kappa_{t_\eps(x),N(x)}\dd w$ uniformly in total variation and in $p$-Wasserstein distance $W_p$ for every $1\le p<\infty$. 
On every bulk family $E_\eps$, the functions $ Q^{\rm r}_{\eps,x}$ and $ Q^{\rm c}_{\eps,x}$ are  eventually defined and satisfy
\begin{equation}\label{eq:bulk-slack-profiles}
 Q^{\rm r}_{\eps,x}(w)=h_d-|w|^2/2+o_{C^1}(1),\qquad Q^{\rm c}_{\eps,x}(w)=h_d-|w|^2/2+o_{C^1}(1),
\end{equation}
uniformly for $x\in E_\eps$ and bounded $w$. On such a family, \eqref{eq:all-boundary-densities} also holds for $p=\infty$. In physical coordinates,
\begin{equation}\label{eq:physical-bulk}
 \ell^{-2}q_\eps(x,T(x)+\ell w)=h(x)-\tfrac12w^TH(x)^{-1}w+o(1)
\end{equation}
uniformly for $x\in E_\eps$ and bounded $w$, and
\begin{equation}\label{eq:bulk-ellipsoid}
 \sup_{x\in E_\eps}d_H\left(\ell^{-1}(S_{\eps,x}-T(x)),\{w:w^TH(x)^{-1}w\le2h(x)\}\right)\longrightarrow0.
\end{equation}
\end{theorem}
\begin{proof}
Fix sequences $\eps_j\downarrow0$ and $x_j\in X$. After passing to a
subsequence, assume $t_j:=t_{\eps_j}(x_j)\to t\in[0,\infty]$ and
$N_j:=N(x_j)\to N$. As in the proof of \cref{thm:microscopic}, the expansion \eqref{eq:rescaled-boundary-chart} for $t<\infty$ and the relation $\delta(x_j)/\ell_j\to\infty$ for $t=\infty$ show that the sets $\Omega_{\eps_j,x_j}$ converge to $\Omega:=\Hh_{t,N}$. For $x\in X$, define
\begin{equation}\label{eq:row-coordinate-change}
 \Psi_{\eps,x}(\zeta)
 =
 \frac{B_x^{-1}[T(x+\ell\Lambda_x\zeta)-T(x)]}{\ell},
 \qquad
 \zeta\in\Omega_{\eps,x}.
\end{equation}
By definition of $\Omega^{\rm r}_{\eps,x}$,
$\Psi_{\eps,x}$ is a bijection from $\Omega_{\eps,x}$ onto
$\Omega^{\rm r}_{\eps,x}$. Since
$B_x^{-1}H(x)\Lambda_x=\Id$ and $T\in C^2(X)$, Taylor's formula gives
$\Psi_{\eps,x}(\zeta)=\zeta+O_{C^1}(\ell)$ on
$\Omega_{\eps,x}\cap\B(0,R)$ for every $R<\infty$, with the constant
independent of $x$. The bi-Lipschitz bounds for $T$ and $S$, together
with the bounds for $B_x$ and $\Lambda_x$ in
\eqref{eq:rescaling-matrices}, give Lipschitz constants for
$\Psi_{\eps,x}$ and $\Psi_{\eps,x}^{-1}$ independent of $\eps$ and
$x$. Since $\Psi_{\eps,x}(0)=0$, for every $R<\infty$ there is
$R'<\infty$, independent of $\eps$ and $x$, such that
$w\in\Omega^{\rm r}_{\eps,x}\cap\B(0,R)$ implies
$\Psi_{\eps,x}^{-1}(w)\in\B(0,R')$.

Substituting $w=\Psi_{\eps,x}(\zeta)$ into the definition of
$Q^{\rm r}_{\eps,x}$ and using
$x_{\eps,\zeta}=x+\ell\Lambda_x\zeta$ gives
\begin{equation}\label{eq:row-slack-pullback}
 Q^{\rm r}_{\eps,x}(\Psi_{\eps,x}(\zeta))
 =
 \frac{q_\eps(x,T(x_{\eps,\zeta}))}{\ell^2\sigma(x)^2},
 \qquad \zeta\in\Omega_{\eps,x}.
\end{equation}
We next prove convergence of $\Omega^{\rm r}_{\eps_j,x_j}$ to
$\Omega$. Let $w_j\in\Omega^{\rm r}_{\eps_j,x_j}$ with $w_j\to w$, and
set $\zeta_j=\Psi_{\eps_j,x_j}^{-1}(w_j)$. The uniform Lipschitz bound
for $\Psi_{\eps_j,x_j}^{-1}$ shows that $(\zeta_j)$ is bounded, while
\eqref{eq:row-coordinate-change} gives
$|\zeta_j-w_j|\to0$. Hence $\zeta_j\to w$, and convergence of
$\Omega_{\eps_j,x_j}$ implies $w\in\Omega$. Conversely, if
$w\in\Omega$, choose $\zeta_j\in\Omega_{\eps_j,x_j}$ with
$\zeta_j\to w$. Then \eqref{eq:row-coordinate-change} gives
$\Psi_{\eps_j,x_j}(\zeta_j)\to w$, with
$\Psi_{\eps_j,x_j}(\zeta_j)\in\Omega^{\rm r}_{\eps_j,x_j}$.

The two preceding approximation properties imply that
$\Omega^{\rm r}_{\eps_j,x_j}\to\Omega$ locally in Hausdorff distance.
Since the sets $\Omega^{\rm r}_{\eps_j,x_j}$ and $\Omega$ are convex,
every compact set $K\subset\intr\Omega$ is contained in
$\Omega^{\rm r}_{\eps_j,x_j}$ for all sufficiently large $j$.
Consequently,
$\1_{\Omega^{\rm r}_{\eps_j,x_j}}\to\1_\Omega$ almost everywhere in
$\R^d$, since $\partial\Omega$ has Lebesgue measure zero.

In \eqref{eq:two-point-profile}, take $\zeta=0$ and $\xi=w$. Since
$x_{\eps,0}=x$, the left-hand side is
$Q^{\rm c}_{\eps,x}(w)$. Hence
$Q^{\rm c}_{\eps_j,x_j}(w_j)\to q_{t,N}(w)$ whenever
$w_j\in\Omega_{\eps_j,x_j}$ and $w_j\to w\in\Omega$. In \eqref{eq:two-point-profile}, take $\xi=0$. Then, using
\eqref{eq:row-slack-pullback}, one obtains
$Q^{\rm r}_{\eps_j,x_j}(w_j)\to q_{t,N}(w)$ whenever
$w_j\in\Omega^{\rm r}_{\eps_j,x_j}$ and $w_j\to w\in\Omega$.  On every compact subset of
$\intr\Omega$, \eqref{eq:row-coordinate-change} and the chain rule give
$C^1$ convergence of both functions to $q_{t,N}$.

Set
$\mathcal S_j^{\rm r}:=\{w\in\Omega^{\rm r}_{\eps_j,x_j}:
Q^{\rm r}_{\eps_j,x_j}(w)\ge0\}$ and
$\mathcal S_j^{\rm c}:=\{w\in\Omega_{\eps_j,x_j}:
Q^{\rm c}_{\eps_j,x_j}(w)\ge0\}$. By
\eqref{eq:prelocalization} and \eqref{eq:model-section-bounds},
$\mathcal S_j^{\rm r}$, $\mathcal S_j^{\rm c}$, and
$\mathcal C_{t_j,N_j}$ are contained in a fixed ball. Moreover, $q_{t_j,N_j}(w_j)\to q_{t,N}(w)$ whenever $w_j\in\Hh_{t_j,N_j}$ and $w_j\to w$. 
For $t<\infty$, this follows from $t_j\to t$, $N_j\to N$, and
continuity of $\varphi_d$. For $t=\infty$, boundedness of $(w_j)$ gives
$t_j+N_j\cdot w_j\to\infty$, and \eqref{eq:profile-deep-limit} gives
$q_{t_j,N_j}(w_j)\to h_d-|w|^2/2=q_{\infty,N}(w)$.

Let $w_j\to w$. If $w_j\in\mathcal S_j^{\rm r}$, then
$w_j\in\Omega^{\rm r}_{\eps_j,x_j}$, so that 
$w\in\Hh_{t,N}$. Since
$Q^{\rm r}_{\eps_j,x_j}(w_j)\to q_{t,N}(w)$ and
$Q^{\rm r}_{\eps_j,x_j}(w_j)\ge0$, one has $q_{t,N}(w)\ge0$.
Hence $w\in\mathcal C_{t,N}$. Similarly, for 
$w_j\in\mathcal S_j^{\rm c}$,  we obtain the same
conclusion. Finally, if $w_j\in\mathcal C_{t_j,N_j}$, then
$t_j+N_j\cdot w_j\ge0$ and $q_{t_j,N_j}(w_j)\ge0$. Passing to the
limit gives $w\in\Hh_{t,N}$ and $q_{t,N}(w)\ge0$. Thus every limit of a sequence
from $\mathcal S_j^{\rm r}$, $\mathcal S_j^{\rm c}$, or
$\mathcal C_{t_j,N_j}$ belongs to $\mathcal C_{t,N}$.

For the converse, choose $w_0$ as in
\eqref{eq:model-section-bounds}, so that
$w_0\in\intr\Hh_{t,N}$ and $q_{t,N}(w_0)\ge c_0$. For
$w\in\mathcal C_{t,N}$ and $0<\theta<1$, set
$w_\theta=(1-\theta)w+\theta w_0$. Then
$w_\theta\in\intr\Hh_{t,N}$, and concavity of $q_{t,N}$ gives
$q_{t,N}(w_\theta)\ge\theta c_0>0$. For fixed $\theta$, the inclusions for compact sets proved above and in the
proof of \cref{thm:microscopic} give
$w_\theta\in\Omega^{\rm r}_{\eps_j,x_j}\cap
\Omega_{\eps_j,x_j}$ for all sufficiently large $j$. Since
$Q^{\rm r}_{\eps_j,x_j}(w_\theta)\to q_{t,N}(w_\theta)>0$ and
$Q^{\rm c}_{\eps_j,x_j}(w_\theta)\to q_{t,N}(w_\theta)>0$, one has
$w_\theta\in\mathcal S_j^{\rm r}\cap\mathcal S_j^{\rm c}$ for all
sufficiently large $j$. Also,
$t_j+N_j\cdot w_\theta>0$ and
$q_{t_j,N_j}(w_\theta)>0$ for all sufficiently large $j$, so
$w_\theta\in\mathcal C_{t_j,N_j}$. Since $w_\theta\to w$ as $\theta\downarrow0$, every
$w\in\mathcal C_{t,N}$ can be approximated arbitrarily closely by
points of $\mathcal S_j^{\rm r}$, $\mathcal S_j^{\rm c}$, and
$\mathcal C_{t_j,N_j}$ for all sufficiently large $j$. Together with
the preceding limit-point property and the common boundedness, this
gives
$d_H(\mathcal S_j^{\rm r},\mathcal C_{t,N})\to0$,
$d_H(\mathcal S_j^{\rm c},\mathcal C_{t,N})\to0$, and
$d_H(\mathcal C_{t_j,N_j},\mathcal C_{t,N})\to0$.
The triangle inequality for $d_H$ then yields
\eqref{eq:all-boundary-sections} along the chosen subsequence.

We next consider the rescaled densities. The changes of
variables  together with \eqref{eq:scale} and
\eqref{eq:transport-jacobian} give
\begin{equation}\label{eq:row-density-exact}
\begin{aligned}
 f^{\rm r}_{\eps,x}(w)
 &=
 \frac{\rho_1(T(x)+\ell B_xw)}{\rho_1(T(x))}
 (Q^{\rm r}_{\eps,x}(w))_+,
 &&w\in\Omega^{\rm r}_{\eps,x},\\
 f^{\rm c}_{\eps,x}(w)
 &=
 \frac{\rho_0(x+\ell\Lambda_xw)}{\rho_0(x)}
 (Q^{\rm c}_{\eps,x}(w))_+,
 &&w\in\Omega_{\eps,x}.
\end{aligned}
\end{equation}
By \eqref{eq:height-upper}, \eqref{eq:prelocalization}, and
\eqref{eq:model-slack-bounds}, the functions
$f^{\rm r}_{\eps_j,x_j}$, $f^{\rm c}_{\eps_j,x_j}$, and
$\kappa_{t_j,N_j}$ are uniformly bounded and vanish outside a common
ball. Extend $f^{\rm r}_{\eps_j,x_j}$ and
$f^{\rm c}_{\eps_j,x_j}$ by zero outside their respective domains. The convergences in \eqref{eq:two-point-profile} and
\eqref{eq:profile-deep-limit}, together with
\eqref{eq:row-slack-pullback} and \eqref{eq:row-density-exact}, then give
$f^{\rm r}_{\eps_j,x_j}(w)\to\kappa_{t,N}(w)$,
$f^{\rm c}_{\eps_j,x_j}(w)\to\kappa_{t,N}(w)$, and
$\kappa_{t_j,N_j}(w)\to\kappa_{t,N}(w)$ for almost every
$w\in\R^d$, including when $t=\infty$.  Dominated convergence and the triangle inequality
give \eqref{eq:all-boundary-densities} for every finite $p$ along the
chosen subsequence. Since the argument applies to arbitrary sequences
$\eps_j\downarrow0$ and $x_j\in X$, the convergence in
\eqref{eq:all-boundary-sections} and
\eqref{eq:all-boundary-densities} is uniform in $x$.

If $f$ and $g$ are probability densities supported in $\B(0,R)$, then
$\|f\dd w-g\dd w\|_{\TV}=\|f-g\|_{L^1}/2$. Applying this to the $L^1$ convergence obtained above gives the total variation and Wasserstein assertions.

It remains to consider the interior regime. On a bulk family,
$\delta(x)/\ell\to\infty$ uniformly. Hence, for every fixed
$R<\infty$, one has $\B(0,R)\subset\Omega_{\eps,x}$ for all
sufficiently small $\eps$, uniformly over the family, by the uniform
bounds for $\Lambda_x$. The corresponding inclusion for the target
domain follows from the bounds for $B_x$ and
\begin{equation}\label{eq:transport-boundary-distance}
 \dist(T(x),\partial Y)
 \ge
 \frac{\delta(x)}{\Lip(S)}.
\end{equation}
Indeed, if $y\in\partial Y$, then $S(y)\in\partial X$, and the
Lipschitz bound for $S$ gives
$\delta(x)\le |x-S(y)|=|S(T(x))-S(y)|
\le\Lip(S)|T(x)-y|$. Taking the infimum over $y\in\partial Y$ proves
\eqref{eq:transport-boundary-distance}.

By \eqref{eq:profile-deep-limit}, the profile in
\eqref{eq:microscopic-comparison} converges uniformly to $h_d/2$ on
every fixed bounded set. Hence
\eqref{eq:two-point-profile}, \eqref{eq:row-coordinate-change}, and
\eqref{eq:row-slack-pullback} give
\eqref{eq:bulk-slack-profiles}. Applying
\eqref{eq:row-density-exact} on a ball containing the supports gives
uniform convergence of both rescaled densities to $\kappa_d$.

The model densities also converge uniformly to $\kappa_d$. For
sufficiently large $t$, the hyperplane
$\{w:N\cdot w=-t\}$ does not intersect the fixed ball containing the
supports, and \eqref{eq:profile-deep-limit} applies uniformly to both
arguments of $\varphi_d$ in $q_{t,N}$. This proves
\eqref{eq:all-boundary-densities} with $p=\infty$ on bulk families.

Finally, replacing $w$ by $B_x^{-1}w$ in
\eqref{eq:bulk-slack-profiles} and multiplying by $\sigma(x)^2$ gives
\eqref{eq:physical-bulk}. Moreover, since
$\ell^{-1}(S_{\eps,x}-T(x))=B_x\mathcal S_{\eps,x}$ and $\mathcal S_{\eps,x}$ converges in Hausdorff distance to
$\B(0,a_d)$ on bulk families, $\sup_x\|B_x\|<\infty$, and
$B_x\B(0,a_d)=\mathscr E_x$, we derive
\eqref{eq:bulk-ellipsoid}. 
\end{proof}

We record two consequences of the section and conditional-law profiles in \cref{thm:section-laws}.
For $x\in X$ and $y\in Y$, define
\begin{equation}\label{eq:conditional-covariance-definitions}
 \Sigma_\eps(x)
 =
 \int_Y
 (y-m_\eps(x))(y-m_\eps(x))^T\,\pi_\eps^x(\dd y),
 \qquad
 \widehat\Sigma_\eps(y)
 =
 \int_X
 (z-\widehat m_\eps(y))(z-\widehat m_\eps(y))^T\,
 \widehat\pi_\eps^y(\dd z).
\end{equation}

\begin{corollary}[Covariance and support measure]
\label{cor:covariance-support}
On every bulk family $E_\eps$,
\begin{equation}\label{eq:covariance-limits}
 \sup_{x\in E_\eps}
 \left(
 \left\|
 \ell^{-2}\Sigma_\eps(x)
 -\frac{2h(x)}{d+4}H(x)
 \right\|
 +
 \left\|
 \ell^{-2}\widehat\Sigma_\eps(T(x))
 -\frac{2h(x)}{d+4}H(x)^{-1}
 \right\|
 \right)
 \longrightarrow0.
\end{equation}
Moreover,
\begin{equation}\label{eq:support-size-asymptotic}
 \ell^{-d}(\mu\otimes\nu)(\spt\pi_\eps)
 \longrightarrow
 \frac{d+2}{2}\int_X h(x)^{-1}\dd\mu(x).
\end{equation}
\end{corollary}

\begin{proof}
We first prove \eqref{eq:covariance-limits}. By the $W_2$ convergence
of the rescaled conditional laws in \cref{thm:section-laws}, their
first and second moments converge uniformly on every bulk family to
those of $\kappa_d(w)\dd w$. Since $\kappa_d$ is radial, its mean is
zero, and \eqref{eq:model-cap-second-moment} gives
\begin{equation}\label{eq:cap-vector-moments}
 \int_{\R^d}w\kappa_d(w)\dd w=0,\qquad \int_{\R^d}w^{\otimes2}\kappa_d(w)\dd w=\frac{2h_d}{d+4}\Id.
\end{equation}
Hence, uniformly for $x\in E_\eps$, the covariance matrix of the
probability density $f^{\rm r}_{\eps,x}$ converges to
$2h_d\Id/(d+4)$. The same limit holds for
the covariance matrix of $f^{\rm c}_{\eps,x}$. Changing variables 
 proves
\eqref{eq:covariance-limits}.

We next prove \eqref{eq:support-size-asymptotic}. Fix
$x\in\intr X$. By \eqref{eq:bulk-slack-profiles}, the indicator of
$\ell^{-1}B_x^{-1}(S_{\eps,x}-T(x))$ converges almost everywhere to
$\1_{\B(0,a_d)}$, since the limiting function
$h_d-|w|^2/2$ vanishes only on $\partial\B(0,a_d)$. By
\eqref{eq:prelocalization}, these sets are contained in a fixed ball.
The change of variables $y=T(x)+\ell B_xw$ and dominated convergence
therefore give
\begin{equation}\label{eq:section-volume-limit}
 \ell^{-d}\nu(S_{\eps,x})
 \longrightarrow
 \det(B_x)\rho_1(T(x))\omega_da_d^d
 =
 \sigma(x)^{-2}\omega_da_d^d
 =
 \frac{d+2}{2h(x)},
\end{equation}
where the two equalities follow from \eqref{eq:scale} and
\eqref{eq:cap-normalization}. By \eqref{eq:preliminary-support} and Fubini's theorem, we have
$(\mu\otimes\nu)(\spt\pi_\eps)
=\int_X\nu(S_{\eps,x})\dd\mu(x)$. Moreover,
\eqref{eq:pregrad} and the upper bound for $\rho_1$ give
$\ell^{-d}\nu(S_{\eps,x})\le C$ uniformly in $x$ and $\eps$.
Since $\mu(\partial X)=0$, dominated convergence applied to
\eqref{eq:section-volume-limit} proves
\eqref{eq:support-size-asymptotic}.
\end{proof}

The potential profiles and the first moments of the half-space conditional laws give the leading boundary errors of the two transport approximations defined in~\eqref{eq:forward-approximations}.

\begin{corollary}[Boundary bias]\label{cor:boundary-gradients}
Uniformly for $x\in X$,
\begin{align}\label{eq:boundary-input}
 z_\eps(x)-T(x)&=-\ell\frac{\sigma(x)}{\gamma(x)}\varphi_d'(t_\eps(x))n(x)+o(\ell),\\
 m_\eps(x)-T(x)&=\ell\frac{\sigma(x)}{\gamma(x)}\beta_d(t_\eps(x))n(x)+o(\ell),\nonumber\\
 \widehat z_\eps(T(x))-x&=-\ell\frac{\sigma(x)}{\gamma(x)}\varphi_d'(t_\eps(x))H(x)^{-1}n(x)+o(\ell),\nonumber\\
 \widehat m_\eps(T(x))-x&=\ell\frac{\sigma(x)}{\gamma(x)}\beta_d(t_\eps(x))H(x)^{-1}n(x)+o(\ell).\nonumber
\end{align}
The leading terms vanish outside $\mathcal U$. On every bulk family, the displacements are $o(\ell)$ uniformly.
In particular,
\begin{equation}\label{eq:uniform-bias-orders}
\begin{aligned}
 \|z_\eps-T\|_{L^\infty(X)}&\asymp\|m_\eps-T\|_{L^\infty(X)}\asymp\ell,\\
 \|\widehat z_\eps-S\|_{L^\infty(Y)}&\asymp\|\widehat m_\eps-S\|_{L^\infty(Y)}\asymp\ell.
\end{aligned}
\end{equation}
\end{corollary}
\begin{proof}
The $C^1$ convergence in \eqref{eq:centered-potentials} permits differentiation at $\xi=0$, with the base point $x$ held fixed. Since $\nabla_\xi\Phi_{\eps,x}(0)=\varphi_d'(t_\eps(x))N(x)$, the identities $H(x)\Lambda_x=B_x$ and $\sigma(x)^2\Lambda_x^{-1}=B_x$ give
\begin{equation}\label{eq:normalized-gradient-profiles}
\begin{aligned}
 \ell^{-1}B_x^{-1}(z_\eps(x)-T(x))&=-\varphi_d'(t_\eps(x))N(x)+o(1),\\
 \ell^{-1}\Lambda_x^{-1}(\widehat z_\eps(T(x))-x)&=-\varphi_d'(t_\eps(x))N(x)+o(1).
\end{aligned}
\end{equation}
For the conditional means, the $L^1$ convergence in \eqref{eq:all-boundary-densities} and the common support bound give convergence of the first moments in rescaled coordinates. The second identity in \eqref{eq:model-centroid} identifies the model mean as $\beta_d(t_\eps(x))N(x)$. Thus \eqref{eq:normalized-gradient-profiles} holds with $(z_\eps,\widehat z_\eps,-\varphi_d')$ replaced by $(m_\eps,\widehat m_\eps,\beta_d)$. Since
\begin{equation}\label{eq:normal-coordinate-identities}
 B_xN(x)=\frac{\sigma(x)}{\gamma(x)}n(x),\qquad
 \Lambda_xN(x)=\frac{\sigma(x)}{\gamma(x)}H(x)^{-1}n(x),
\end{equation}
multiplication of the forward and reverse normalized displacements by $\ell B_x$ and $\ell\Lambda_x$, respectively, gives \eqref{eq:boundary-input}. On bulk families, $t_\eps\to\infty$ uniformly, so \eqref{eq:profile-exponential} makes both scalar coefficients tend uniformly to zero and gives the $o(\ell)$ bounds.

The bounded profiles and coefficient fields in \eqref{eq:boundary-input} give the upper bounds in \eqref{eq:uniform-bias-orders}, with $T(X)=Y$ for the reverse errors. At $x\in\partial X$, the depth is $t_\eps(x)=0$ and \cref{lem:profile-regularity} gives $-\varphi_d'(0)>0$ and $\beta_d(0)>0$. The vectors multiplying these scalar profiles have norms bounded away from zero on $\partial X$. The uniform $o(\ell)$ remainders therefore give boundary errors of size at least $c\ell$, also on $\partial Y$ because $T(\partial X)=\partial Y$. All four error fields are continuous, so their suprema equal their essential suprema.
\end{proof}
\subsection{Hessian and Sobolev convergence}\label{sec:second-derivatives}
\begin{corollary}[Hessian and $W^{2,p}$ convergence]\label{prop:interior-hessian}
For every measurable bulk family $E_\eps$,
\begin{equation}\label{eq:bulk-hessian-convergence}
 \esssup_{x\in E_\eps}\left(\|D^2u_\eps(x)-H(x)\|+\|D^2v_\eps(T(x))-H(x)^{-1}\|\right)\longrightarrow0.
\end{equation}
Consequently, $u_\eps\to u_0$ in $W^{2,p}(X)$ and $v_\eps\to v_0$ in $W^{2,p}(Y)$ for every $1\le p<\infty$.
\end{corollary}
\begin{proof}
Fix a measurable bulk family $E_\eps$. All the $o(1)$ and $o(\ell)$
terms below are uniform for $x\in E_\eps$ and, when a point
$y\in S_{\eps,x}$ is present, also uniformly in $y$. Choose $r<a_d$. By \eqref{eq:bulk-slack-profiles},
$Q^{\rm r}_{\eps,x}(w)>0$ for $|w|\le r$ and all sufficiently small
$\eps$, uniformly for $x\in E_\eps$. By the definition
\eqref{eq:rescaled-slacks}, this implies
$T(x)+\ell B_x\B(0,r)\subset S_{\eps,x}$. The uniform bounds for
$B_x$ and $B_x^{-1}$ in \eqref{eq:rescaling-matrices} therefore give
$B(T(x),c\ell)\subset S_{\eps,x}$. The opposite inclusion
$S_{\eps,x}\subset B(T(x),C\ell)$ follows from
\eqref{eq:prelocalization}.

Differentiating the first identity in
\eqref{eq:centered-potentials} with respect to $\xi$ at $\xi=0$ gives
$\Lambda_x^T(z_\eps(x)-T(x))/(\ell\sigma(x)^2)
=-\nabla_\xi\Phi_{\eps,x}(0)+o(1)$. Since $E_\eps$ is a bulk family,
$t_\eps(x)\to\infty$ uniformly on $E_\eps$. By
\eqref{eq:microscopic-comparison} and
\eqref{eq:profile-deep-limit},
$\nabla_\xi\Phi_{\eps,x}(0)\to0$ uniformly. The uniform bounds for
$\sigma$ and $\Lambda_x^{-1}$ then give
$z_\eps(x)-T(x)=o(\ell)$.

Let $y\in S_{\eps,x}$ and write
$y=T(x)+\ell B_xw$. By \eqref{eq:prelocalization} and the uniform
bound for $B_x^{-1}$, the corresponding $w$ remains in a fixed ball.
Differentiating the first identity in \eqref{eq:rescaled-slacks} with
respect to $w$ gives
$\nabla_yq_\eps(x,y)=\ell\sigma(x)^2B_x^{-T}
\nabla_wQ^{\rm r}_{\eps,x}(w)$. By
\eqref{eq:bulk-slack-profiles},
$\nabla_wQ^{\rm r}_{\eps,x}(w)=-w+o(1)$. Since
$w=\ell^{-1}B_x^{-1}(y-T(x))$ and
$B_xB_x^T=\sigma(x)^2H(x)$ by
\eqref{eq:rescaling-matrices}, we obtain
\begin{equation}\label{eq:physical-hessian-profiles}
\begin{gathered}
 B(T(x),c\ell)\subset S_{\eps,x}\subset B(T(x),C\ell),
 \qquad z_\eps(x)-T(x)=o(\ell),\\
 \nabla_yq_\eps(x,y)
 =
 -H(x)^{-1}(y-T(x))+o(\ell)
 \qquad (y\in S_{\eps,x}).
\end{gathered}
\end{equation}

By \eqref{eq:transport-boundary-distance} and the definition of a bulk
family,
$\dist(T(x),\partial Y)/\ell\to\infty$ uniformly for
$x\in E_\eps$. Together with
$S_{\eps,x}\subset B(T(x),C\ell)$ from
\eqref{eq:physical-hessian-profiles}, this gives
$S_{\eps,x}\subset\intr Y$ for sufficiently small $\eps$.

If $y\in\partial S_{\eps,x}$, then
$|y-T(x)|\ge c\ell$ by the first inclusion in
\eqref{eq:physical-hessian-profiles}, while
$|y-T(x)|\le C\ell$ by the second inclusion. Uniform ellipticity of
$H(x)$ and the second line of
\eqref{eq:physical-hessian-profiles} therefore give
$|\nabla_yq_\eps(x,y)|\asymp\ell$ on
$\partial S_{\eps,x}$. Since
$S_{\eps,x}\subset\intr Y$, its boundary is the zero level of
$q_\eps(x,\cdot)$. Thus, for $d\ge2$, the implicit function theorem
shows that $\partial S_{\eps,x}$ is a $C^1$ hypersurface, with outward
unit normal
$-\nabla_yq_\eps(x,y)/|\nabla_yq_\eps(x,y)|$.

The differentiation formula for \eqref{eq:barycenters}, displayed in
\cite[Section~2.2]{GNGeometry}, gives
\begin{equation}\label{eq:physical-moving-hessian}
 D^2u_\eps(x)
 =
 \frac1{\nu(S_{\eps,x})}
 \int_{\partial S_{\eps,x}}
 \frac{(y-z_\eps(x))(y-z_\eps(x))^T\rho_1(y)}
      {|\nabla_yq_\eps(x,y)|}
 \dd\mathcal H^{d-1}(y)
\end{equation}
for almost every
$x\in E_\eps$. For $y\in\partial S_{\eps,x}$,
\eqref{eq:physical-hessian-profiles} gives
$y-z_\eps(x)=-H(x)\nabla_yq_\eps(x,y)+o(\ell)$. Since
$|y-z_\eps(x)|=O(\ell)$,
$|\nabla_yq_\eps(x,y)|\asymp\ell$, and $H(x)$ is symmetric, it follows
that
$\frac{(y-z_\eps(x))(y-z_\eps(x))^T}
{|\nabla_yq_\eps(x,y)|}
=-(y-z_\eps(x))
\frac{\nabla_yq_\eps(x,y)^T}{|\nabla_yq_\eps(x,y)|}H(x)+o(\ell)$
uniformly on $\partial S_{\eps,x}$.

Since $v_\eps$ is convex, $q_\eps(x,\cdot)$ is concave, and hence
$S_{\eps,x}$ is convex. For $y\in\partial S_{\eps,x}$, the supporting
hyperplane at $y$ and the inclusion
$B(T(x),c\ell)\subset S_{\eps,x}$ give
$(y-T(x))\cdot
[-\nabla_yq_\eps(x,y)/|\nabla_yq_\eps(x,y)|]\ge c\ell$.
The divergence theorem and
$S_{\eps,x}\subset B(T(x),C\ell)$ therefore give
\begin{equation}\label{eq:hessian-divergence-identities}
\begin{gathered}
 c\ell^d\le |S_{\eps,x}|\le C\ell^d,\qquad
 \mathcal H^{d-1}(\partial S_{\eps,x})\le C\ell^{d-1},\\
 -\int_{\partial S_{\eps,x}}
 (y-z_\eps(x))
 \frac{\nabla_yq_\eps(x,y)^T}{|\nabla_yq_\eps(x,y)|}
 \dd\mathcal H^{d-1}(y)
 =
 |S_{\eps,x}|\Id .
\end{gathered}
\end{equation}
Indeed, the surface-area estimate follows from
$c\ell\mathcal H^{d-1}(\partial S_{\eps,x})
\le d|S_{\eps,x}|$, and the identity follows by applying the
divergence theorem componentwise to
$y\mapsto y-z_\eps(x)$.

By \eqref{eq:physical-hessian-profiles},
$\sup_{y\in S_{\eps,x}}|y-T(x)|\le C\ell$. Since $\rho_1$ is uniformly
continuous on $Y$, it follows that
$\sup_{y\in S_{\eps,x}}|\rho_1(y)-\rho_1(T(x))|=o(1)$. The
surface-area estimate in \eqref{eq:hessian-divergence-identities}
shows that the integral of the uniform $o(\ell)$ term above is
$o(\ell^d)$. Using the identity in
\eqref{eq:hessian-divergence-identities}, the numerator in
\eqref{eq:physical-moving-hessian} is therefore
$\rho_1(T(x))|S_{\eps,x}|H(x)+o(\ell^d)$. Similarly,
$\nu(S_{\eps,x})
=\rho_1(T(x))|S_{\eps,x}|+o(\ell^d)$. Since
$|S_{\eps,x}|\ge c\ell^d$ and $\rho_1$ is bounded away from zero,
\eqref{eq:physical-moving-hessian} gives
$D^2u_\eps(x)=H(x)+o(1)$ uniformly for almost every
$x\in E_\eps$. For $d=1$, the boundary integrals are sums over the two
endpoints of $S_{\eps,x}$, and the same calculation applies.

For $D^2v_\eps$, apply the same argument after interchanging
$X,\mu,u_\eps,T$ with $Y,\nu,v_\eps,S$. By
\eqref{eq:transport-boundary-distance} and the bulk assumption,
$\inf_{x\in E_\eps}\dist(T(x),\partial Y)/\ell\to\infty$, so
$T(E_\eps)$ satisfies the 
bulk condition in $Y$.
Therefore
$D^2v_\eps(y)-D^2v_0(y)\to0$ uniformly for almost every
$y\in T(E_\eps)$. Since
$D^2v_0(T(x))=DS(T(x))=H(x)^{-1}$, this gives the second term in
\eqref{eq:bulk-hessian-convergence}. 

Finally, every compact subset of $\intr X$ is a bulk family, so
\eqref{eq:bulk-hessian-convergence} gives
$D^2u_\eps\to D^2u_0$ in $L^\infty$ on each compact subset of
$\intr X$. By \eqref{eq:curvature-input},
$D^2u_\eps-D^2u_0$ is uniformly bounded on $X$. Since
$|\partial X|=0$, exhaustion of $\intr X$ by compact subsets gives
$D^2u_\eps\to D^2u_0$ in $L^p(X)$ for every $1\le p<\infty$.
The same argument gives
$D^2v_\eps\to D^2v_0$ in $L^p(Y)$. Finally,
\eqref{eq:G-inputs} gives uniform convergence of the potentials, and
\eqref{eq:pregrad} gives uniform convergence of their gradients.
Therefore $u_\eps\to u_0$ in $W^{2,p}(X)$ and
$v_\eps\to v_0$ in $W^{2,p}(Y)$.
\end{proof}

\begin{proof}[Proof of \cref{thm:main-geometry}]
The graph and support assertions \eqref{eq:main-graph-inclusion} follow from \eqref{eq:graph-positive} and \eqref{eq:tubes}. A fixed compact set $K\subset\intr X$ is a bulk family. By \cref{thm:section-laws} and \eqref{eq:profile-deep-limit}, the normalized sections converge to $\B(0,a_d)$ and the normalized conditional laws converge in total variation to $\kappa_d(w)\dd w$, uniformly for $x\in K$. The map $w\mapsto B_xw$ sends these limits to $\mathscr E_x$ and $k_x(w)\dd w$, respectively, by \eqref{eq:rescaling-matrices} and \eqref{eq:cap-change-of-variables}. Hausdorff distance increases by at most $\|B_x\|$, which is uniformly bounded, and total variation is unchanged under this invertible map. This proves \eqref{eq:main-interior-profiles}. Interchanging the marginals gives the corresponding conclusions for $S$. The Sobolev limits \eqref{eq:main-Sobolev} are proved in \cref{prop:interior-hessian}.
\end{proof}

\section{Neumann corrector and potential expansions}\label{sec:macroscopic}

The limit \eqref{eq:uniform-graph-profile} determines the leading term of $r_\eps+s_\eps\circ T$. To prove \cref{thm:corrector}, it remains to identify the uniform limit of the normalized difference $g_\eps$ in \eqref{eq:errors}. Write
\begin{equation}\label{eq:corrector-coefficients}
 \mathcal A=\rho_0 H^{-1},\qquad \mathfrak b=\frac{h}{2(d+4)}\nabla\log\det H.
\end{equation}
With this notation, the Neumann problem \eqref{eq:corrector-Neumann} is
\[
 \int_X \mathcal A(\nabla g_0-\mathfrak b)\cdot\nabla\psi\,\dd x=0.
\]
By \cref{prop:inputs}, $\mathcal A$ is symmetric and uniformly positive
definite, $\mathcal A\in C^{1,\beta}(X)$, and
$\mathcal A\mathfrak b\in C^{0,\beta}(X)$.

We now prove the uniform convergence of the normalized potential
difference to $g_0$ and the convergence of its $\ell$-scale difference
quotients used in \cref{sec:centering}. The compactness and truncation
estimates invoked in the proof are established in
\cref{sec:nonlocal} and applied with $l=\ell$.

\subsection{Convergence to the corrector}\label{sec:corrector-proof}
\begin{proposition}[Convergence to the Neumann corrector]\label{prop:corrector-convergence}
There is a unique $g_0\in H^1(X)$ satisfying \eqref{eq:corrector-Neumann} and $\int_Xg_0\dd\mu=0$. It belongs to $C^{1,\beta}(X)$ for some $\beta\in(0,\alpha)$, and $g_\eps\to g_0$ uniformly on $X$. Moreover, with $\widetilde g_\eps=g_\eps-g_0$, for every $r>0$,
\begin{equation}\label{eq:short-energy}
 \ell^{-d-2}\iint_{\substack{x,z\in X\\|x-z|<r\ell}}|\widetilde g_\eps(z)-\widetilde g_\eps(x)|^2\dd x\dd z\longrightarrow0.
\end{equation}
\end{proposition}
\begin{proof}
\emph{Nonlocal equation.}
Recall the notation $Q_\eps(x,z)=q_\eps(x,T(z))$. Define
\begin{equation}\label{eq:corrector-kernel}
 \omega_\eps(x,z)
 =
 \int_0^1
 \1_{\{(1-t)Q_\eps(z,x)+tQ_\eps(x,z)>0\}}
 \dd t,
 \qquad
 J_\eps(x,z)=\ell^{-d}\omega_\eps(x,z).
\end{equation}
Changing $t$ to $1-t$ in \eqref{eq:corrector-kernel} gives
$\omega_\eps(x,z)=\omega_\eps(z,x)$, and
$0\le\omega_\eps\le1$. By \eqref{eq:W-secant} and \eqref{eq:Q},
\begin{equation}\label{eq:corrector-slack-difference}
 (Q_\eps(x,z))_+-(Q_\eps(z,x))_+
 =
 2\omega_\eps(x,z)
 [\eta_\eps(z)-\eta_\eps(x)-\mathfrak a(x,z)].
\end{equation}
By \eqref{eq:Q-marginals}, integrating
\eqref{eq:corrector-slack-difference} with respect to $\mu$
 gives
\begin{equation}\label{eq:corrector-exact}
 \int_X
 J_\eps(x,z)
 [\eta_\eps(z)-\eta_\eps(x)-\mathfrak a(x,z)]
 \dd\mu(z)=0
 \qquad(x\in X).
\end{equation}
By \eqref{eq:tubes} and the Lipschitz
continuity of $T$, there exist $R>a>0$, independent of $\eps$, such
that %
\begin{equation}\label{eq:omega-ellipticity}
 \1_{\{|x-z|<a\ell\}}
 \le
 \omega_\eps(x,z)
 \le
 \1_{\{|x-z|<R\ell\}}
 \qquad(x,z\in X).
\end{equation}
Increase $R$ so that
$R^2>2\max_{x\in X}h(x)\|H(x)^{-1}\|$. Then
$\{w:w^TH(x)w/2<h(x)\}\subset B_R$ for every $x\in X$.

For $x\in X$ and $w\in B_R$ with $x+\ell w\in X$, define
\begin{equation}\label{eq:corrector-edge-definitions}
\begin{gathered}
 D_\ell v(x,w)
 =
 \frac{v(x+\ell w)-v(x)}{\ell},
 \qquad
 F_\ell(x,w)
 =
 \frac{\mathfrak a(x,x+\ell w)}{\ell^3},\\
 a_\ell(x,w)
 =
 \rho_0(x)\rho_0(x+\ell w)
 \omega_\eps(x,x+\ell w).
\end{gathered}
\end{equation}
Set these expressions equal to zero when $x+\ell w\notin X$.
By \eqref{eq:D}, two integrations by parts along
$t\mapsto x+t\ell w$ give
\begin{equation}\label{eq:cubic-source-limit}
\begin{gathered}
 F_\ell(x,w)
 =
 \frac12\int_0^1
 t(1-t)D^3u_0(x+t\ell w)[w,w,w]\dd t
 =
 F_0(x,w)+O(\ell^\beta),\\
 F_0(x,w)=\frac1{12}D^3u_0(x)[w,w,w].
\end{gathered}
\end{equation}
The $C^{0,\beta}$ regularity of $D^3u_0$ makes the error in
\eqref{eq:cubic-source-limit} uniform for
$w\in B_R$ and $x+\ell w\in X$. In particular, $F_\ell$ is uniformly
bounded on $X\times B_R$.

Multiply \eqref{eq:corrector-exact} by $\psi(x)$ and integrate with
respect to $\mu$. Interchanging $x$ and $z$, using the
symmetry of $J_\eps$ and the antisymmetry of $\mathfrak a$, and
averaging gives the factor $\psi(z)-\psi(x)$. Since
$\eta_\eps=\ell^2g_\eps$, the change of variables $z=x+\ell w$ yields
\begin{equation}\label{eq:corrector-weak-edges}
 \int_{X\times B_R}
 a_\ell D_\ell g_\eps D_\ell\psi\dd x\dd w
 =
 \int_{X\times B_R}
 a_\ell F_\ell D_\ell\psi\dd x\dd w
 \qquad(\psi\in L^2(X)).
\end{equation}

\emph{Compactness and the limiting equation.}
Taking $\psi=g_\eps$ in \eqref{eq:corrector-weak-edges} and applying
Cauchy--Schwarz gives
\begin{equation}\label{eq:corrector-energy-bound}
 \int_{X\times B_R}
 a_\ell|D_\ell g_\eps|^2\dd x\dd w
 \le
 \int_{X\times B_R}
 a_\ell|F_\ell|^2\dd x\dd w
 \le C.
\end{equation}
By \eqref{eq:omega-ellipticity} and the positive lower bound for
$\rho_0$, $a_\ell(x,w)\ge c$ whenever
$|w|<a$ and $x+\ell w\in X$. Hence
\eqref{eq:corrector-energy-bound},
\eqref{eq:preliminary-error-bounds}, and
\cref{lem:edge-compactness} imply that, along a subsequence of any
sequence $\eps\downarrow0$,
\begin{equation}\label{eq:corrector-subsequential-limit}
 g_\eps\longrightarrow g
 \quad\text{in }L^2(\mu),
 \qquad
 g\in H^1(X),
 \qquad
 \int_Xg\dd\mu=0.
\end{equation}

Let $x$ vary in a fixed compact subset of $\intr X$ and
$w\in B_R$. For all sufficiently small $\eps$,
$x+\ell w\in X$. In \eqref{eq:two-point-profile}, first take
$\xi=0$ and $\zeta=\Lambda_x^{-1}w$, and then take
$\xi=\Lambda_x^{-1}w$ and $\zeta=0$. Since
$x+\ell\Lambda_x\Lambda_x^{-1}w=x+\ell w$,
\eqref{eq:rescaling-matrices} and
\eqref{eq:profile-deep-limit} give
\begin{equation}\label{eq:corrector-slack-limit}
 \ell^{-2}Q_\eps(x,x+\ell w)
 \longrightarrow
 h(x)-\frac12w^TH(x)w,
 \qquad
 \ell^{-2}Q_\eps(x+\ell w,x)
 \longrightarrow
 h(x)-\frac12w^TH(x)w,
\end{equation}
uniformly on compact subsets of $\intr X\times B_R$.

Let
$E_x=\{w:w^TH(x)w/2<h(x)\}$.
If $x\in\intr X$ and $w\in B_R\setminus\partial E_x$, then
\eqref{eq:corrector-slack-limit} implies that
$Q_\eps(x,x+\ell w)$ and $Q_\eps(x+\ell w,x)$ eventually have the
sign of $h(x)-w^TH(x)w/2$. Hence
$\omega_\eps(x,x+\ell w)\to\1_{E_x}(w)$.
Since $\partial E_x$ and $\partial X$ have Lebesgue measure zero,
continuity of $\rho_0$ gives
\begin{equation}\label{eq:weight-limit}
 a_\ell(x,w)\longrightarrow
 a_0(x,w):=\rho_0(x)^2\1_{E_x}(w)
 \qquad\text{for a.e. }(x,w)\in X\times B_R.
\end{equation}

The functions $a_\ell$ are uniformly bounded.  \cref{lem:edge-compactness}, together with
\eqref{eq:weight-limit}, therefore gives
\begin{equation}\label{eq:corrector-weak-edge-limit}
 \sqrt{a_\ell}D_\ell g_\eps
 \rightharpoonup
 \sqrt{a_0}\,\nabla g(x)\cdot w
 \qquad\text{in }L^2(X\times B_R).
\end{equation}
For $\psi\in C^1(X)$ and $x+\ell w\in X$, convexity of $X$ gives
$D_\ell\psi(x,w)
=\int_0^1\nabla\psi(x+t\ell w)\cdot w\dd t$.
For almost every $x\in X$ and every fixed $w\in B_R$,
$x+\ell w\in X$ for all sufficiently small $\ell$. Using
\eqref{eq:weight-limit}, the uniform bound for $a_\ell$,
\eqref{eq:cubic-source-limit}, and dominated convergence gives
\begin{equation}\label{eq:corrector-strong-test-limits}
 \sqrt{a_\ell}D_\ell\psi
 \longrightarrow
 \sqrt{a_0}\,\nabla\psi(x)\cdot w,
 \qquad
 \sqrt{a_\ell}F_\ell
 \longrightarrow
 \sqrt{a_0}F_0
 \quad\text{in }L^2(X\times B_R).
\end{equation}
Combining \eqref{eq:corrector-weak-edge-limit} and
\eqref{eq:corrector-strong-test-limits} in
\eqref{eq:corrector-weak-edges} then yields
\begin{equation}\label{eq:corrector-limit-integral}
 \int_X\rho_0(x)^2
 \int_{E_x}
 [\nabla g(x)\cdot w-F_0(x,w)]
 \nabla\psi(x)\cdot w
 \dd w\dd x=0
 \qquad(\psi\in C^1(X)).
\end{equation}
We next identify two terms in
\eqref{eq:corrector-limit-integral}. We claim that
\begin{equation}\label{eq:corrector-moments}
 \rho_0(x)^2\int_{E_x}ww^T\dd w=\mathcal A(x),
 \qquad
 \rho_0(x)^2\int_{E_x}wF_0(x,w)\dd w
 =
 \mathcal A(x)\mathfrak b(x).
\end{equation}
For $w\in E_x$, set
$c_x(w)=h(x)-w^TH(x)w/2$. Since
$c_x=0$ on $\partial E_x$ and
$\nabla_wc_x=-H(x)w$, integration by parts gives
\begin{equation}\label{eq:corrector-cap-integration}
 \int_{E_x}wf(w)\dd w
 =
 H(x)^{-1}
 \int_{E_x}c_x(w)\nabla_wf(w)\dd w
 \qquad
 (f\in C^1(\overline E_x)).
\end{equation}
The change of variables
$w=\sigma(x)H(x)^{-1/2}v$ maps $B_{a_d}$ onto $E_x$.
Using \eqref{eq:scale}, \eqref{eq:cap-normalization}, and
\eqref{eq:cap-vector-moments} gives
\begin{equation}\label{eq:corrector-cap-moments}
 \int_{E_x}c_x(w)\dd w=\rho_0(x)^{-1},
 \qquad
 \int_{E_x}c_x(w)ww^T\dd w
 =
 \frac{2h(x)}{\rho_0(x)(d+4)}H(x)^{-1}.
\end{equation}
Taking $f(w)=w_j$ in
\eqref{eq:corrector-cap-integration} and using the first identity in
\eqref{eq:corrector-cap-moments} proves the first identity in
\eqref{eq:corrector-moments}. For
$f(w)=F_0(x,w)$,
\[
 \partial_{w_j}F_0(x,w)
 =
 \frac14\sum_{k,m}
 \partial_{jkm}u_0(x)w_kw_m.
\]
Thus the second identity in \eqref{eq:corrector-cap-moments} and
Jacobi's formula
$\sum_{k,m}(H(x)^{-1})_{km}\partial_{jkm}u_0(x)
=\partial_j\log\det H(x)$ give
\[
 \int_{E_x}c_x(w)\nabla_wF_0(x,w)\dd w
 =
 \frac{h(x)}{2\rho_0(x)(d+4)}
 \nabla\log\det H(x).
\]
Substitution into \eqref{eq:corrector-cap-integration} proves the
second identity in \eqref{eq:corrector-moments}.

Using \eqref{eq:corrector-moments} in
\eqref{eq:corrector-limit-integral} gives
\eqref{eq:corrector-Neumann} with $g$ in place of $g_0$.
Since $C^1(X)$ is dense in $H^1(X)$, the identity holds for every
$\psi\in H^1(X)$.

If $g^{(1)}$ and $g^{(2)}$ satisfy
\eqref{eq:corrector-Neumann} and
$\int_Xg^{(1)}\dd\mu=\int_Xg^{(2)}\dd\mu=0$, subtracting their
equations and taking
$\psi=g^{(1)}-g^{(2)}$ gives
\[
 \int_X
 (\nabla g^{(1)}-\nabla g^{(2)})^T
 \mathcal A
 (\nabla g^{(1)}-\nabla g^{(2)})
 \dd x=0.
\]
By the uniform positive definiteness of $\mathcal A$ from
\cref{prop:inputs},
$\nabla g^{(1)}=\nabla g^{(2)}$ almost everywhere. Since $X$ is
connected, $g^{(1)}-g^{(2)}$ is constant, and the zero $\mu$-mean
condition makes this constant zero. Thus
\eqref{eq:corrector-Neumann} has a unique solution with zero
$\mu$-mean, namely $g_0$.

Every subsequential limit in
\eqref{eq:corrector-subsequential-limit} is therefore $g_0$, and
hence
$g_\eps\to g_0$ in $L^2(\mu)$ for the full family. The bound
\eqref{eq:corrector-energy-bound} and the uniqueness of the
subsequential limit in \eqref{eq:corrector-weak-edge-limit} also give
\eqref{eq:corrector-weak-edge-limit}  with
$g=g_0$. Finally, the uniform bound for $g_\eps$ in
\eqref{eq:preliminary-error-bounds} and a.e.\ convergence
along an $L^2$-convergent subsequence give
$g_0\in L^\infty(X)$.

\emph{Regularity and energy convergence.}
For $d\ge2$, apply
\cite[Corollary~1.5 and Remark~1.6]{Vita2022} with $k=0$,
coefficient matrix $\mathcal A$, vector forcing
$-\mathcal A\mathfrak b$, and all other data equal to zero.
By \cref{prop:inputs}, $\mathcal A$ is symmetric and uniformly
positive definite,
$\mathcal A\in C^{1,\beta}(X)$, and
$\mathcal A\mathfrak b\in C^{0,\beta}(X)$, while the boundary regularity
required in \cite{Vita2022} follows from \cref{ass:data}. The weak formulation
\cite[equation~(2.3)]{Vita2022} is precisely
\eqref{eq:corrector-Neumann}, and the homogeneous Neumann condition is
the vanishing normal component of
$\mathcal A(\nabla g_0-\mathfrak b)$. The boundary and interior
estimates therefore give
$g_0\in C^{1,\beta}(X)$.

If $d=1$, every function in $L^2(X)$ is the derivative of a function
in $H^1(X)$. Thus \eqref{eq:corrector-Neumann} implies
$\mathcal A(g_0'-\mathfrak b)=0$ almost everywhere. Since
$\mathcal A>0$, $g_0'=\mathfrak b$, and hence
$g_0\in C^{1,\beta}(X)$ also in dimension one.

Set $\widetilde g_\eps=g_\eps-g_0$. Taking
$\psi=\widetilde g_\eps$ in
\eqref{eq:corrector-weak-edges} and subtracting the term containing
$D_\ell g_0$ gives
\begin{equation}\label{eq:corrector-energy-error}
 \int_{X\times B_R}
 a_\ell|D_\ell\widetilde g_\eps|^2
 \dd x\dd w
 =
 \int_{X\times B_R}
 \sqrt{a_\ell}(F_\ell-D_\ell g_0)
 \sqrt{a_\ell}D_\ell\widetilde g_\eps
 \dd x\dd w
 \longrightarrow0.
\end{equation}
Indeed, since $g_0\in C^{1,\beta}(X)$,
\eqref{eq:corrector-strong-test-limits} applies with $\psi=g_0$, so
$\sqrt{a_\ell}(F_\ell-D_\ell g_0)$ converges strongly in
$L^2(X\times B_R)$. By
\eqref{eq:corrector-weak-edge-limit} with $g=g_0$ and the strong
convergence of $\sqrt{a_\ell}D_\ell g_0$,
$\sqrt{a_\ell}D_\ell\widetilde g_\eps\rightharpoonup0$ in
$L^2(X\times B_R)$. Hence the right-hand side of
\eqref{eq:corrector-energy-error} tends to zero.

Since $g_0\in C^{1,\beta}(X)$,
$D_\ell g_0(x,w)-\nabla g_0(x)\cdot w\to0$ uniformly for
$w\in B_R$ and $x+\ell w\in X$. It follows from
\eqref{eq:corrector-energy-error} that
\begin{equation}\label{eq:strong-edge-convergence}
 \int_{X\times B_R}
 a_\ell
 |D_\ell g_\eps-\nabla g_0(x)\cdot w|^2
 \dd x\dd w
 \longrightarrow0.
\end{equation}

By \eqref{eq:omega-ellipticity} and the positive lower bound for
$\rho_0$, \eqref{eq:corrector-energy-error} gives
\[
 \ell^{-d-2}
 \iint_{\substack{x,z\in X\\|x-z|<a\ell}}
 |\widetilde g_\eps(z)-\widetilde g_\eps(x)|^2
 \dd x\dd z
 \longrightarrow0.
\]
For any fixed $r>0$, choose $n\in\mathbb N$ such that $r\le na$.
Then \eqref{eq:app-length-comparison} gives
\eqref{eq:short-energy}.

\emph{Uniform convergence.}
Dividing \eqref{eq:corrector-exact} by $\ell^2$ and using $\eta_\eps=\ell^2g_\eps$ yields
\begin{equation}\label{eq:corrector-error-equation}
 \int_X
 J_\eps(x,z)
 [\widetilde g_\eps(z)-\widetilde g_\eps(x)
 -\mathcal F_\eps(x,z)]
 \dd\mu(z)=0,
 \qquad
 \mathcal F_\eps(x,z)
 =
 \frac{\mathfrak a(x,z)}{\ell^2}
 -
 [g_0(z)-g_0(x)].
\end{equation}
The function $\mathcal F_\eps$ is antisymmetric. If
$J_\eps(x,z)>0$, then
\eqref{eq:omega-ellipticity} gives $|x-z|<R\ell$. Hence
\eqref{eq:common-gap-bounds} and the Lipschitz continuity of $g_0$
show
\[
 |\mathcal F_\eps(x,z)|
 \le
 C\ell^{-2}|x-z|^3
 +
 \Lip(g_0)|x-z|
 \le C\ell
\]
whenever $J_\eps(x,z)>0$. This inequality, \eqref{eq:preliminary-error-bounds},
\eqref{eq:corrector-subsequential-limit} (with $g=g_0$), \eqref{eq:omega-ellipticity},
\eqref{eq:corrector-error-equation} and  the antisymmetry of
$\mathcal F_\eps$ verify the hypotheses of
\cref{lem:uniform-truncation}. Consequently,
$\|\widetilde g_\eps\|_{L^\infty(X)}\to0$.
\end{proof}

\begin{proof}[Proof of \cref{thm:corrector}]
Proposition~\ref{prop:corrector-convergence} gives existence, uniqueness, and regularity of $g_0$, together with uniform convergence $g_\eps\to g_0$. By \eqref{eq:common-gap-bounds}, $2\tau_\eps(x)=q_\eps(x,T(x))$. The limit \eqref{eq:uniform-graph-profile} and boundedness of $\sigma$ therefore give $\tau_\eps=\ell^2\mathfrak p_\eps+o(\ell^2)$ uniformly on $X$. Substituting into $r_\eps=\ell^2g_\eps-\tau_\eps$ and $s_\eps\circ T=-\ell^2g_\eps-\tau_\eps$ proves \eqref{eq:corrector-uniform}. By \eqref{eq:profile-deep-limit} and \eqref{eq:boundary-coordinates}, $\mathfrak p_\eps\to h/2$ locally uniformly on $\intr X$. The uniform profile bounds and dominated convergence give \eqref{eq:global-bulk-bias}.
\end{proof}

While the centered profiles \eqref{eq:centered-potentials} removed the values at the base point, we can now combine those profiles with the base-point values from \cref{thm:corrector} to determine the uncentered profiles.

\begin{corollary}[Uncentered potential profiles]\label{cor:uncentered-potential-profiles}
For every $M>0$, uniformly for $x\in X$ and $\xi\in\Omega_{\eps,x}\cap\B(0,M)$,
\begin{equation}\label{eq:uncentered-potential-profiles}
\begin{aligned}
 -\frac{r_\eps(x_{\eps,\xi})}{\ell^2\sigma(x)^2}
 &=\Phi_{\eps,x}(\xi)-\frac{g_0(x)}{\sigma(x)^2}+o_{C^1}(1),\\
 -\frac{s_\eps(T(x_{\eps,\xi}))}{\ell^2\sigma(x)^2}
 &=\Phi_{\eps,x}(\xi)+\frac{g_0(x)}{\sigma(x)^2}+o_{C^1}(1).
\end{aligned}
\end{equation}
\end{corollary}
\begin{proof}
Subtracting the corresponding unregularized Taylor expansions from \eqref{eq:centered-potentials} gives the profiles of $r_\eps(x_{\eps,\xi})-r_\eps(x)$ and $s_\eps(T(x_{\eps,\xi}))-s_\eps(T(x))$. The normalized Taylor remainders are $O_{C^1}(\ell)$. Adding the base-point values from \eqref{eq:corrector-uniform} and using $\mathfrak p_\eps(x)=\sigma(x)^2\Phi_{\eps,x}(0)$ gives \eqref{eq:uncentered-potential-profiles}. The normalized errors in the base-point values are $o(1)$ uniformly in $x$ and constant in $\xi$, so the resulting remainders are $o_{C^1}(1)$.
\end{proof}

\subsection{Nonlocal estimates}\label{sec:nonlocal}\label{app:nonlocal}

This subsection supplies analytic tools used to prove the corrector convergence in the preceding Section~\ref{sec:corrector-proof} as well as the small-displacement estimates in Section~\ref{sec:small-displacements} below. Throughout the subsection, $\Omega\subset\R^d$ is a convex body and $\vartheta(\dd x)=\varrho(x)\dd x$, with $0<\varrho_-\le\varrho\le\varrho_+<\infty$ almost everywhere. The measure $\vartheta$ has arbitrary finite mass. Here $l>0$ is a free parameter, and limits are taken as $l\downarrow0$. Proposition~\ref{prop:corrector-convergence} uses these estimates with $\Omega=X$, $\vartheta=\mu$, and $l=\ell$. Fix $0<a\le b<\infty$ and $0<\lambda\le\Lambda<\infty$. We consider measurable kernels $J_l:\Omega\times \Omega\to[0,\infty)$ satisfying
\begin{equation}\label{eq:app-kernel-bounds}
 J_l(x,y)=J_l(y,x),\qquad \lambda l^{-d}\1_{\{|x-y|<al\}}\le J_l(x,y)\le\Lambda l^{-d}\1_{\{|x-y|<bl\}}.
\end{equation}
Thus interactions have range at most $bl$, with a positive kernel lower bound on $|x-y|<al$. These conditions and all equations involving measurable functions hold almost everywhere. The density may depend on $l$, subject to the same bounds $\varrho_-,\varrho_+$. Constants depend only on $\Omega,d,a,b,\lambda,\Lambda,\varrho_\pm$ and the fixed smooth cutoffs in the smoothing estimates.
\subsubsection{Smoothing and averaged Sobolev estimates}

For $t>0$ and $v\in L^2(\Omega)$, set
\begin{equation}\label{eq:app-short-energy}
 \mathcal E_{l,t}^\Omega(v)
 =l^{-d-2}\iint_{\substack{x,y\in\Omega\\|x-y|<tl}}
 |v(y)-v(x)|^2\dd x\dd y,\qquad
 \mathcal E_l^\Omega(v)=\mathcal E_{l,a}^\Omega(v).
\end{equation}
For $J_l$ as in \eqref{eq:app-kernel-bounds}, set $m_l(x)=\int_\Omega J_l(x,y)\dd\vartheta(y)$ and
$(P_l v)(x)=m_l(x)^{-1}\int_\Omega J_l(x,y)v(y)\dd\vartheta(y)$,
$$\mathscr E_l^{\Omega,\vartheta}(v)
=l^{-2}\iint_{\Omega\times\Omega}J_l(x,y)|v(y)-v(x)|^2
\dd\vartheta(x)\dd\vartheta(y).$$
Convexity gives $|\Omega\cap B(x,r)|\asymp r^d$ uniformly in $x$
for small $r$, hence $c\le m_l\le C$ for small $l$. Set $Z_l=\int_\Omega m_l\dd\vartheta$,
$\vartheta_l^\star=Z_l^{-1}m_l\vartheta$, and
$K_l(x,y)=Z_lJ_l(x,y)/(m_l(x)m_l(y))$. Applying \cite[Lemma~6.8, equation~(6.25)]{GNGeometry}, with $\kappa=1+1/d$, to the symmetric Markov density $K_l$ relative to $\vartheta_l^\star$, yields the bound
\begin{equation}\label{eq:app-averaged-sobolev}
 \|P_l v\|_{L^{2\kappa}(\Omega,\vartheta)}
 \le C\bigl(\|v\|_{L^2(\Omega,\vartheta)}
 +\mathscr E_l^{\Omega,\vartheta}(v)^{1/2}\bigr)
 \qquad(v\ge0).
\end{equation}
We use the smoothing of \cite[equation~(6.22)]{GNGeometry} with normalized Lebesgue measure on $\Omega$. By \cite[Lemma~6.8]{GNGeometry}, this gives a smoothing operator $\mathcal S_l:L^2(\Omega)\to H^1(\Omega)$, independent of $J_l$ and $\vartheta$, satisfying $\|\mathcal S_l v\|_2\le C\|v\|_2$ and
\begin{equation}\label{eq:smoothing}
 \|v-\mathcal S_l v\|_2
 \le Cl\mathcal E_l^\Omega(v)^{1/2},\qquad
 \|\nabla\mathcal S_l v\|_2
 \le C\mathcal E_l^\Omega(v)^{1/2},
\end{equation}
where the $L^2$ norms use Lebesgue measure.
Finally, \cite[equation~(6.26)]{GNGeometry} yields the bound
\begin{equation}\label{eq:app-length-comparison}
 \mathcal E_{l,t}^\Omega(v)
 \le n^{d+2}\mathcal E_l^\Omega(v)
 \qquad(t>0,\ n\in\N,\ t\le na,\ l>0).
\end{equation}

\subsubsection{Nonlocal compactness}
\begin{lemma}[Nonlocal compactness]\label{lem:edge-compactness}
Let $\Omega$ be a convex body, $l_j\downarrow0$, and $u_j\in L^2(\Omega)$ with
\begin{equation}\label{eq:app-compactness-hypothesis}
 \sup_j\bigl(\|u_j\|_{L^2(\Omega)}^2+\mathcal E_{l_j,a}^\Omega(u_j)\bigr)<\infty
\end{equation}
for some $a>0$. There is a subsequence, not relabeled, and $u\in H^1(\Omega)$, such that $u_j\to u$ strongly in $L^2(\Omega)$. Along this subsequence, for every $R>0$, the functions
\begin{equation}\label{eq:app-difference-quotient}
 D_j u_j(x,w)=\begin{cases}
 \displaystyle\frac{u_j(x+l_jw)-u_j(x)}{l_j},&x+l_jw\in \Omega,\\[5pt]
 0,&x+l_jw\notin \Omega,
 \end{cases}\qquad (x,w)\in \Omega\times B_R
\end{equation}
are bounded in $L^2(\Omega\times B_R)$ and satisfy
\begin{equation}\label{eq:app-quotient-limit}
 D_j u_j\rightharpoonup\nabla u(x)\cdot w\qquad\text{in }L^2(\Omega\times B_R).
\end{equation}
For a fixed $R$, let $\gamma_j:\Omega\times B_R\to[0,\infty)$ be measurable, with $\sup_j\|\gamma_j\|_\infty<\infty$ and $\gamma_j\to\gamma$ almost everywhere. Then, along the same subsequence,
\begin{equation}\label{eq:app-weighted-quotient-limit}
 \sqrt{\gamma_j}\,D_j u_j\rightharpoonup\sqrt\gamma\,\nabla u(x)\cdot w\qquad\text{in }L^2(\Omega\times B_R).
\end{equation}
\end{lemma}
\begin{proof}
By \eqref{eq:smoothing}, $\mathcal S_{l_j}u_j$ is bounded in $H^1(\Omega)$ and $\|u_j-\mathcal S_{l_j}u_j\|_2=O(l_j)$. Weak compactness in $H^1(\Omega)$ and Rellich compactness on $\intr\Omega$ give a subsequential limit $u\in H^1(\Omega)$ of $\mathcal S_{l_j}u_j$. Such a limit holds strongly in  $L^2(\Omega)$ and weakly in $H^1(\Omega)$. Hence, $u_j$ shares the same limit in $L^2(\Omega)$. For fixed $R>0$, choose an integer $n\ge1$ with $R\le na$. Changing variables $y=x+l_jw$ and applying \eqref{eq:app-length-comparison} gives
\begin{equation}\label{eq:app-cylinder-bound}
 \|D_j u_j\|_{L^2(\Omega\times B_R)}^2=\mathcal E_{l_j,R}^\Omega(u_j)\le n^{d+2}\mathcal E_{l_j,a}^\Omega(u_j).
\end{equation}
For $\zeta\in C_c^\infty(\intr\Omega\times B_R)$,  its support in $x$ and its translates by $l_jw$ lie in $\intr\Omega$ for large $j$, so that  by changing variables,  the strong $L^2$ convergence of $u_j$ and integration by parts, 
\begin{equation}\label{eq:app-quotient-test}
\begin{aligned}
& \int_{\Omega\times B_R}D_j u_j\,\zeta\dd x\dd w =\int_{\Omega\times B_R}u_j(x)\frac{\zeta(x-l_jw,w)-\zeta(x,w)}{l_j}\dd x\dd w\\
 & \qquad \longrightarrow \int_{\Omega\times B_R}-u(x) \nabla_x\zeta(x,w)\cdot w\dd x\dd w =\int_{\Omega\times B_R}\nabla u(x)\cdot w\,\zeta(x,w)\dd x\dd w.
\end{aligned}
\end{equation}
Combining this with 
 \eqref{eq:app-cylinder-bound} yields \eqref{eq:app-quotient-limit} for each $R$ along the same subsequence. For any $\zeta\in L^2(\Omega\times B_R)$, boundedness and a.e.\ convergence of $\gamma_j$ give $\sqrt{\gamma_j}\,\zeta\to\sqrt\gamma\,\zeta$ strongly in $L^2$. Pairing with \eqref{eq:app-quotient-limit} proves \eqref{eq:app-weighted-quotient-limit}. %
\end{proof}

\subsubsection{Uniform convergence by truncation}
\begin{lemma}[Uniform convergence of bounded solutions]\label{lem:uniform-truncation}
Suppose \eqref{eq:app-kernel-bounds} holds. Let $u_l\in L^\infty(\Omega)$ satisfy
\begin{equation}\label{eq:app-forced-equation}
 \int_\Omega J_l(x,y)[u_l(y)-u_l(x)-F_l(x,y)]\dd\vartheta(y)=0\qquad\text{for $\vartheta$-a.e. }x\in\Omega,
\end{equation}
where $F_l$ is measurable and antisymmetric, and $|F_l|\le Ll$ on $\{J_l>0\}$ for a fixed $L\ge0$. If $\sup_l\|u_l\|_\infty<\infty$ and $\|u_l\|_{L^2(\Omega,\vartheta)}\to0$, then $\|u_l\|_\infty\to0$.
\end{lemma}
\begin{proof}
Fix $M\ge1$ such that $\sup_l\|u_l\|_\infty\le M$.
For $k>0$, set $w_k=(u_l-k)_+$ and $A_k=\{u_l>k\}$.
Testing \eqref{eq:app-forced-equation} with $w_k$, using
symmetry of $J_l$ and antisymmetry of $F_l$, gives
\begin{equation}\label{eq:truncation-weak}
\begin{aligned}
 &\iint_{\Omega^2}J_l(x,y)
 [u_l(y)-u_l(x)][w_k(y)-w_k(x)]\dd\vartheta(x)\dd\vartheta(y)\\
 &\qquad=\iint_{\Omega^2}J_l(x,y)F_l(x,y)
 [w_k(y)-w_k(x)]\dd\vartheta(x)\dd\vartheta(y).
\end{aligned}
\end{equation}
Since $ [u_l(y)-u_l(x)][w_k(y)-w_k(x)]
 \ge |w_k(y)-w_k(x)|^2$
and $w_k(y)\ne w_k(x)$ requires $x\in A_k$ or $y\in A_k$,
Young's inequality in \eqref{eq:truncation-weak},
$|F_l|\le Ll$, and \eqref{eq:app-kernel-bounds} imply $\mathscr E_l^{\Omega,\vartheta}(w_k)
 \le CL^2\vartheta(A_k)$ and 
$ \|w_k\|_{L^2(\vartheta)}
 \le M\vartheta(A_k)^{1/2},$ 
where the second bound follows from $0\le w_k\le M\1_{A_k}$.
By \eqref{eq:app-averaged-sobolev} 
\begin{equation}\label{eq:truncation-sobolev}
 \|P_lw_k\|_{L^{2\kappa}(\vartheta)}
 \le C(M+L)\vartheta(A_k)^{1/2},
 \qquad \kappa=1+1/d.
\end{equation}
Dividing \eqref{eq:app-forced-equation} by $m_l(x)$
and applying Jensen's inequality gives
\begin{equation}\label{eq:truncation-average}
 u_l\le P_lu_l+Ll,\qquad w_k\le P_lw_k+Ll.
\end{equation}
Hence $P_lw_k\ge h-k-Ll$ on $A_h$ for $h>k+Ll$.
Together with \eqref{eq:truncation-sobolev}, this yields
\begin{equation}\label{eq:truncation-levels}
 \vartheta(A_h)
 \le\left(\frac{C(M+L)}{h-k-Ll}\right)^{2\kappa}
 \vartheta(A_k)^\kappa.
\end{equation}
Fix $\eta>0$ and set $k_j=\eta/2+(\eta/4)(1-2^{-j})$.
Since $k_{j+1}-k_j=\eta2^{-j-3}$,
\eqref{eq:truncation-levels} gives
\begin{equation}\label{eq:truncation-recursion}
 \vartheta(A_{k_{j+1}})
 \le D_\eta B^j\vartheta(A_{k_j})^\kappa
 \qquad\text{if }\eta2^{-j-3}\ge2Ll,
\end{equation}
where $D_\eta\ge1$ and $B>1$ are independent of $l,j$.
We claim that, for sufficiently small $l$,
\begin{equation}\label{eq:truncation-decay}
 \vartheta(A_{k_j})
 \le D_\eta^{-1/(\kappa-1)}
 B^{-j/(\kappa-1)-1/(\kappa-1)^2}e^{-\kappa^j}
 \le e^{-\kappa^j},
\end{equation}
provided $\eta2^{-i-3}\ge2Ll$ for $0\le i<j$.
For $j=0$, Chebyshev's inequality and
$\|u_l\|_{L^2(\vartheta)}\to0$ give
\[
 \vartheta(A_{k_0})
 \le4\eta^{-2}\|u_l\|_{L^2(\vartheta)}^2
 \le D_\eta^{-1/(\kappa-1)}
 B^{-1/(\kappa-1)^2}e^{-1}
\]
for small $l$. Substituting \eqref{eq:truncation-decay}
into \eqref{eq:truncation-recursion} proves the induction step. Let $N$ be the least integer such that
$\kappa^N\ge(d+2)\log(1/l)$.
Then $N=O(\log\log(1/l))$ and $2^{-N}/l\to\infty$,
so the condition in \eqref{eq:truncation-recursion}
holds for $0\le j<N$ when $l$ is small.
By \eqref{eq:truncation-decay},
$\vartheta(A_{k_N})\le e^{-\kappa^N}\le l^{d+2}$.
Together with \eqref{eq:app-kernel-bounds} and $m_l\ge c$,
this gives
$P_lw_{k_N}\le CMl^{-d}\vartheta(A_{k_N})\le CMl^2$.
Thus \eqref{eq:truncation-average} implies
\[
 u_l\le k_N+w_{k_N}
 \le k_N+P_lw_{k_N}+Ll
 \le\frac{3\eta}{4}+CMl^2+Ll
 \le\eta
\]
for small $l$. Applying the same argument to $-u_l$ and $-F_l$
gives $\|u_l\|_\infty\le\eta$. Since $\eta>0$ is arbitrary,
$\|u_l\|_\infty\to0$.
\end{proof}

\subsection{Equality of the leading corrections}
We use the inner product induced by $\mathcal A$ to characterize equality of the leading corrections to $r_\eps$ and $s_\eps\circ T$, or equivalently $g_0\equiv0$. For square-integrable vector fields $V,W:X\to\R^d$, write $\langle V,W\rangle_{\mathcal A}=\int_XV^T\mathcal A W\dd x$ and $\|V\|_{\mathcal A}=\langle V,V\rangle_{\mathcal A}^{1/2}$.

\begin{proposition}[Equality of the leading corrections]\label{prop:equal}
Under the normalization \eqref{eq:gauge},
\begin{equation}\label{eq:equal-splitting-criterion}
 g_0\equiv0\quad\Longleftrightarrow\quad\det DT\text{ is constant on }X\quad\Longleftrightarrow\quad\|r_\eps-s_\eps\circ T\|_\infty=o(\ell^2).
\end{equation}
More precisely, with $L=\log\det H$,
\begin{equation}\label{eq:quant-equal}
 \frac{\min_Xh}{2(d+4)}\|\nabla L\|_{\mathcal A}\le\|\nabla g_0\|_{\mathcal A}\le\frac{\max_Xh}{2(d+4)}\|\nabla L\|_{\mathcal A}.
\end{equation}
\end{proposition}
\begin{proof}
Testing \eqref{eq:corrector-Neumann} with $L$ and with $g_0$, respectively, gives
\[
 \langle\nabla g_0,\nabla L\rangle_{\mathcal A}=\frac{\langle h\nabla L,\nabla L\rangle_{\mathcal A}}{2(d+4)},\qquad
 \|\nabla g_0\|_{\mathcal A}^2=\frac{\langle h\nabla L,\nabla g_0\rangle_{\mathcal A}}{2(d+4)},
\]
and now \eqref{eq:quant-equal} follows by Cauchy--Schwarz. Thus $\nabla g_0=0$ if and only if $\nabla L=0$. Connectedness of $X$ and $\int_Xg_0\dd\mu=0$ give the first equivalence in \eqref{eq:equal-splitting-criterion}. For the second, \eqref{eq:errors} and \cref{prop:corrector-convergence} yield $\ell^{-2}\|r_\eps-s_\eps\circ T\|_\infty=2\|g_\eps\|_\infty\to2\|g_0\|_\infty$.
\end{proof}

\begin{remark}[Affine Brenier maps]
\Cref{prop:equal} shows in particular that affine Brenier maps have $g_0=0$. In dimension one, the converse also holds, since $\det DT$ is the derivative of $T$. For $d\ge2$, however, the converse can fail. Indeed, on $\{x_2>0\}$, the function
$u_0(x)=\frac{x_1^2}{2x_2}+\frac{x_2^3}{6}
 +\frac12\sum_{j=3}^d x_j^2$
satisfies $D^2u_0>0$ and $\det D^2u_0=1$, whereas
$T=\nabla u_0$ is nonaffine. Let $X=\overline B_r(e_2)$ and
$Y=T(X)$. Since $DT(e_2)=I$, the maps
$\xi\mapsto r^{-1}[T(e_2+r\xi)-T(e_2)]$ converge to the identity
in $C^2(\overline B_1)$ as $r\downarrow0$. Thus, for sufficiently
small $r>0$, both domains have smooth uniformly convex boundaries.
Moreover, $\det DT=1$ gives $|Y|=|X|$, and $T$ transports the
uniform probability measure on $X$ to that on $Y$. These measures
satisfy \cref{ass:data}, and \cref{prop:equal} yields $g_0=0$.
\end{remark}
\section{Second-order bias}\label{sec:centering}\label{sec:cancellation}

In this section, we prove the local limits in \cref{thm:main-bias} and the remainder
bounds used in \cref{sec:norms}, under \cref{ass:data} and the normalization
\eqref{eq:gauge}. The global limits in \cref{thm:main-bias} will be proved in \cref{sec:norms}. On $X$, set
\begin{equation}\label{eq:bias-error-fields}
 V^z_\eps=z_\eps-T,\qquad
 V^m_\eps=m_\eps-T,\qquad
 \widehat V^z_\eps=\widehat z_\eps\circ T-\id,\qquad
 \widehat V^m_\eps=\widehat m_\eps\circ T-\id.
\end{equation}
Let $b_z,b_m$ be as in \eqref{eq:biases}, and define
\begin{equation}\label{eq:reverse-biases}
 \widehat b_z\circ T=-H^{-1}(\nabla g_0+\tfrac12\nabla h),\qquad
 \widehat b_m\circ T=-H^{-1}(\nabla g_0+\tfrac{d}{2(d+4)}\nabla h).
\end{equation}
For efficient exposition, we denote the respective coefficients by
$b_V\in\{b_z,b_m,\widehat b_z\circ T,\widehat b_m\circ T\}$;
each belongs to $C(X;\R^d)$.
Throughout, $\ell=\eps^{1/(d+2)}$ and $\eps$ is sufficiently small.
\begin{theorem}[Second-order bias expansion]\label{thm:centers}
For each error field $V_\eps$ in \eqref{eq:bias-error-fields} and every compact $K\subset\intr X$,
\begin{equation}\label{eq:centers-local}
 \ell^{-2}V_\eps\longrightarrow b_V\qquad\text{strongly in }L^2(K).
\end{equation}
There is $L_0>0$ such that, on $X_\ell=\{x\in X:\delta(x)>L_0\ell\}$, each error admits the decomposition
\begin{equation}\label{eq:field-decomp}
\begin{gathered}
 V_\eps=\ell^2b_V+\mathcal R_\eps+\mathcal B_\eps,\qquad \|\mathcal R_\eps\|_{L^2(X_\ell)}=o(\ell^2),\qquad \|\mathcal R_\eps\|_{L^\infty(X_\ell)}\le C\ell,\\
 |\mathcal B_\eps(x)|\le C\ell e^{-c\delta(x)/\ell}\qquad\text{for a.e. }x\in X_\ell.
\end{gathered}
\end{equation}
The remainders depend on the error field, while $L_0,c,C>0$ can be chosen for all four fields simultaneously.
\end{theorem}

To wit, \eqref{eq:field-decomp} holds up to distance $O(\ell)$
from the boundary, with an $L^2$ remainder $\mathcal R_\eps$
and an exponentially decaying term $\mathcal B_\eps$. This separation will be crucial below. Note that for fixed $L>0$ and small $\ell$,
$\{\delta>L\ell\}$ is nonempty, open, and convex. 
The proof of \cref {thm:centers} will be given in \cref{sec:bias-proof}, based on the preparations developed next.

\subsection{Estimates for small displacements}\label{sec:small-displacements}
Use \eqref{eq:app-short-energy} with $\Omega=X$, $l=\ell$,
and $a$ from \eqref{eq:omega-ellipticity}.
By \cref{prop:corrector-convergence},
$\mathcal E_\ell(g_\eps-g_0)\to0$.
The next lemma controls increments along small perturbations
of the identity.

\Needspace{12\baselineskip}
\begin{lemma}[Difference quotients along small displacements]\label{lem:sampling}
Let $E_\ell\subset\intr X$ be open and convex. Suppose that $\Phi_\ell\in C^1(E_\ell;X)$ satisfies
\[
 \sup_{E_\ell}|\Phi_\ell-\id|\le M\ell,\qquad \sup_{E_\ell}\|D\Phi_\ell-\Id\|\le\kappa
\]
for fixed $M>0$ and $0\le\kappa<1$. Then
\begin{equation}\label{eq:moving-increments}
 \left\|\frac{g_\eps\circ\Phi_\ell-g_\eps}{\ell}-\nabla g_0\cdot\frac{\Phi_\ell-\id}{\ell}\right\|_{L^2(E_\ell)}\le C\mathcal E_\ell(\widetilde g_\eps)^{1/2}+C\omega_0(M\ell)\longrightarrow0,
\end{equation}
where $\omega_0(t)=\sup\{|\nabla g_0(x)-\nabla g_0(z)|:x,z\in X,\ |x-z|\le t\}$. Moreover,
\begin{equation}\label{eq:displacement-composition-bound}
 \|f\circ\Phi_\ell\|_{L^2(E_\ell)}\le(1-\kappa)^{-d/2}\|f\|_{L^2(X)}\qquad(f\in L^2(X)).
\end{equation}
The constant $C$ in \eqref{eq:moving-increments} depends only on the fixed data, $M$, and $\kappa$.
\end{lemma}
\begin{proof}
For $t\in[0,1]$, let $\Phi_{\ell,t}=(1-t)\id+t\Phi_\ell$. The segment from $x$ to $\Phi_\ell(x)$ is in $X$ by convexity. Integrating $D\Phi_{\ell,t}-\Id$ along segments in the convex set $E_\ell$ gives
\[
 |\Phi_{\ell,t}(x)-\Phi_{\ell,t}(z)|\ge(1-\kappa)|x-z|\qquad(x,z\in E_\ell).
\]
Each $\Phi_{\ell,t}$ is therefore injective. The derivative bound also makes its singular values at least $1-\kappa$, so the absolute value of its Jacobian determinant is at least $(1-\kappa)^d$. Change of variables proves \eqref{eq:displacement-composition-bound} for every $\Phi_{\ell,t}$, with the same constant. In particular, composing with these maps preserves equality almost everywhere.

The smoothing estimate \eqref{eq:smoothing} separates $\widetilde g_\eps$ into an $L^2$ error and a smooth function with a small $L^2$ gradient. Write $\widetilde g_\eps=e_\ell+w_\ell$, where $w_\ell=\mathcal S_\ell\widetilde g_\eps$. Then $\|e_\ell\|_2\le C\ell\mathcal E_\ell(\widetilde g_\eps)^{1/2}$
and $\|\nabla w_\ell\|_2\le C\mathcal E_\ell(\widetilde g_\eps)^{1/2}$.
By \eqref{eq:displacement-composition-bound}, the increment $e_\ell\circ\Phi_\ell-e_\ell$ has $L^2(E_\ell)$ norm at most $C\ell\mathcal E_\ell(\widetilde g_\eps)^{1/2}$. For the smooth part, integration along the segment gives
\[
 w_\ell(\Phi_\ell(x))-w_\ell(x)=\int_0^1\nabla w_\ell(\Phi_{\ell,t}(x))\cdot[\Phi_\ell(x)-x]\dd t.
\]
Minkowski's inequality and \eqref{eq:displacement-composition-bound}
applied to $\Phi_{\ell,t}$ give
$\|w_\ell\circ\Phi_\ell-w_\ell\|_{L^2(E_\ell)}
\le C\ell\mathcal E_\ell(\widetilde g_\eps)^{1/2}$.
Moreover,
$|g_0\circ\Phi_\ell-g_0-\nabla g_0\cdot(\Phi_\ell-\id)|
\le M\ell\omega_0(M\ell)$.
Combining these bounds with the estimate for $e_\ell$ and
dividing by $\ell$ proves \eqref{eq:moving-increments}.
\end{proof}

\subsection{Slack on the Brenier graph}
Recall from \eqref{eq:common-gap-bounds} that $\tau_\eps(x)=q_\eps(x,T(x))/2$. We sharpen its leading approximation $\ell^2h/2$ by separating a small $L^2$ error from the correction caused by the boundary.

\begin{lemma}[Slack expansion on the Brenier graph]\label{thm:height}
There are measurable functions $e_\eps,b_\eps$ on $X$ such that
\begin{equation}\label{eq:height-decomp}
 \tau_\eps=\ell^2h/2+e_\eps+b_\eps,\qquad \|e_\eps\|_{L^2(X)}=o(\ell^3),\qquad |b_\eps(x)|\le C\ell^2e^{-c\delta(x)/\ell}\quad\text{a.e. on }X.
\end{equation}
In addition, $\|e_\eps\|_\infty\le C\ell^2$. In particular, for every compact $K\subset\intr X$,
\begin{equation}\label{eq:height-local}
 \|q_\eps(\cdot,T(\cdot))-\ell^2h\|_{L^2(K)}=o(\ell^3).
\end{equation}
\end{lemma}
\begin{proof}
\emph{Symmetrization.} By \eqref{eq:Q},
\begin{equation}\label{eq:symmetric-slack-parts}
\begin{aligned}
 U_\eps(x,z)
 &=\tfrac12[Q_\eps(x,z)+Q_\eps(z,x)]
 =\tau_\eps(x)+\tau_\eps(z)-D^{\rm s}(x,z),\\
 V_\eps^{\rm a}(x,z)
 &=\tfrac12[Q_\eps(x,z)-Q_\eps(z,x)]
 =\ell^2[g_\eps(z)-g_\eps(x)]-\mathfrak a(x,z).
\end{aligned}
\end{equation}
The identity $ \tfrac12[(U+V)_++(U-V)_+]
 =U_++\tfrac12(|V|-|U|)_+$ 
and \eqref{eq:Q-marginals} give
\begin{equation}\label{eq:symmetric-equation}
 \int_X(U_\eps(x,z))_+\dd\mu(z)
 =\eps-\int_XC_\eps(x,z)\dd\mu(z),
\end{equation}
where $C_\eps=(|V_\eps^{\rm a}|-|U_\eps|)_+/2$.
By \eqref{eq:prelocalization}, both integrands vanish
for $|x-z|\ge R\ell$, after increasing~$R$. Extend $U_\eps(x,x+\ell w)$, $V_\eps^{\rm a}(x,x+\ell w)$,
and $C_\eps(x,x+\ell w)$ by zero when $x+\ell w\notin X$.

To estimate $C_\eps$, we first derive unweighted strong convergence of the difference quotients from \eqref{eq:short-energy}.
With $D_\ell$ as in \eqref{eq:corrector-edge-definitions},
the change of variables $z=x+\ell w$ and \eqref{eq:short-energy}
give, for each $M>0$,
\[
 \|D_\ell\widetilde g_\eps\|_{L^2(X\times B_M)}^2
 =\ell^{-d-2}\iint_{\substack{x,z\in X\\|x-z|<M\ell}}
 |\widetilde g_\eps(z)-\widetilde g_\eps(x)|^2\dd x\dd z
 \longrightarrow0.
\]
For $x+\ell w\in X$ and $|w|<M$,
$|D_\ell g_0(x,w)-\nabla g_0(x)\cdot w|
\le M\omega_0(M\ell)$.
Otherwise, $D_\ell g_0=0$ and $\delta(x)\le M\ell$.
Since $|\{\delta\le M\ell\}|\to0$ and $\nabla g_0$ is bounded,
$D_\ell g_0\to\nabla g_0(x)\cdot w$ in $L^2(X\times B_M)$.
Consequently,
\begin{equation}\label{eq:unweighted-strong-differences}
 D_\ell g_\eps\longrightarrow\nabla g_0(x)\cdot w
 \qquad\text{strongly in }L^2(X\times B_M).
\end{equation}

By \eqref{eq:symmetric-slack-parts},
\eqref{eq:unweighted-strong-differences}, and
\eqref{eq:cubic-source-limit},
\begin{equation}\label{eq:antisymmetric-slack-limit}
 \ell^{-3}V_\eps^{\rm a}(x,x+\ell w)
 \longrightarrow
 \nabla g_0(x)\cdot w-\tfrac1{12}D^3u_0(x)[w,w,w]
 \qquad\text{strongly in }L^2(X\times B_R).
\end{equation}
Moreover, \cref{prop:corrector-convergence} and
\eqref{eq:common-gap-bounds} give, for $|x-z|<R\ell$,
\[
 \ell^{-2}|V_\eps^{\rm a}(x,z)|
 \le2\|g_\eps-g_0\|_\infty+C\ell\longrightarrow0.
\]
By \eqref{eq:corrector-slack-limit},
$\ell^{-2}U_\eps(x,x+\ell w)\to h(x)-w^TH(x)w/2$
almost everywhere. For each $x$, the ellipsoid boundary
$\{w:w^TH(x)w=2h(x)\}$ has Lebesgue measure zero.
Hence, for almost every $(x,w)\in X\times B_R$, we have $\ell^{-2}|U_\eps(x,x+\ell w)|
\to|h(x)-w^TH(x)w/2|>0$, whereas
$\ell^{-2}|V_\eps^{\rm a}(x,x+\ell w)|\to0$.
Thus $\1_{\{|U_\eps(x,x+\ell w)|<|V_\eps^{\rm a}(x,x+\ell w)|\}}
\to0$ almost everywhere, and dominated convergence shows that
the measure of this set tends to zero. The $L^2$-convergence in \eqref{eq:antisymmetric-slack-limit} and
$ 0\le C_\eps
 \le\tfrac12|V_\eps^{\rm a}|
 \1_{\{|U_\eps|<|V_\eps^{\rm a}|\}}$ 
therefore imply
$\|\ell^{-3}C_\eps(x,x+\ell w)\|_{L^2(X\times B_R)}\to0$.
Finally, changing variables,
\begin{equation}\label{eq:switching-small}
 \left\|\ell^{-d-3}\int_XC_\eps(\cdot,z)\dd\mu(z)\right\|_{L^2(X)}
 \longrightarrow0.
\end{equation}

\emph{Comparison with the quadratic profile.}
Set $f_\ell=\ell^2h/2$,
$U_\ell^0(x,z)=f_\ell(x)+f_\ell(z)-D^{\rm s}(x,z)$, and
$R_\ell(x)=\int_X(U_\ell^0(x,z))_+\dd\mu(z)-\eps$.
By \eqref{eq:common-gap-bounds},
$(U_\ell^0)_+=0$ for $|x-z|\ge R\ell$, after increasing~$R$. 
For $x+\ell w\in X$ and $|w|<R$,  a Taylor expansion gives
\begin{equation}\label{eq:symmetric-comparison-expansion}
\begin{aligned}
 \ell^{-2}U_\ell^0(x,x+\ell w)
 &=A_x(w)+\ell L_x(w)+O(\ell^{1+\beta}),\\
 A_x(w)&=h(x)-\tfrac12w^TH(x)w,\\
 L_x(w)&=\tfrac12\nabla h(x)\cdot w
 -\tfrac14D^3u_0(x)[w,w,w].
\end{aligned}
\end{equation}
Here we used
$D^{\rm s}(x,z)=\tfrac12(T(z)-T(x))\cdot(z-x)$,
$T\in C^{2,\beta}$, and $h\in C^{1,\beta}$. The uniform bounds on $H$ and $h$, with
$v=H(x)^{1/2}w$, give
\begin{equation}\label{eq:quadratic-level-slab}
 |\{w\in B_R:|A_x(w)|\le t\}|
 \le Ct\qquad(0<t<t_0).
\end{equation}
Using
$ |(A+B)_+-A_+-\1_{\{A>0\}}B|
 \le |B|\1_{\{|A|\le|B|\}},$
equations \eqref{eq:symmetric-comparison-expansion}
and \eqref{eq:quadratic-level-slab} yield
\[
 \int_{B_R}
 \left|(\ell^{-2}U_\ell^0(x,x+\ell w))_+
 -(A_x(w))_+-\ell L_x(w)\1_{\{A_x(w)>0\}}\right|\dd w
 \le C(\ell^2+\ell^{1+\beta})
\]
when $\delta(x)>R\ell$.
Combining this with
$\rho_0(x+\ell w)=\rho_0(x)+\ell\nabla\rho_0(x)\cdot w
+O(\ell^{1+\beta})$, the order-$\ell$ terms integrate to zero
as they are odd. Since
$\rho_0(x)\int_{B_R}(A_x)_+\dd w=1$
by \eqref{eq:scale} and \eqref{eq:cap-normalization}, and
$\eps=\ell^{d+2}$, we obtain
\begin{equation}\label{eq:height-comparison-residual}
 |R_\ell(x)|\le C\ell^{d+3+\beta}
 \quad\text{if }\delta(x)>R_0\ell,\qquad
 |R_\ell(x)|\le C\ell^{d+2}\quad(x\in X),
\end{equation}
for fixed $R_0\ge R$.
Since $(U_\ell^0)_+\le C\ell^2$ and vanishes for
$|x-z|\ge R\ell$, boundedness of $\rho_0$ gives
$|R_\ell(x)|\le C\ell^2\mu(B(x,R\ell))+\eps
\le C\ell^{d+2}$ for all $x\in X$.

\emph{Inversion and the boundary remainder.}
Define
\[
 W_\ell(x,z)=\int_0^1
 \1_{\{(1-t)U_\ell^0(x,z)+tU_\eps(x,z)>0\}}\dd t.
\]
This kernel is symmetric. By \eqref{eq:graph-positive},
$\tau_\eps\ge c\ell^2$, and $f_\ell\ge c\ell^2$
by \eqref{eq:scale}. Together with \eqref{eq:common-gap-bounds} and \eqref{eq:prelocalization}, these bounds show that  $U_\eps$ and $U_\ell^0$ are positive for $|x-z|<a'\ell$ and nonpositive for $|x-z|>R'\ell$, for suitable constants $0<a'<R'$. Hence
\[
 \1_{\{|x-z|<a'\ell\}}
 \le W_\ell(x,z)\le
 \1_{\{|x-z|<R'\ell\}}.
\]
Set $m_\ell(x)=\int_XW_\ell(x,z)\dd\mu(z)$ and
$P_\ell f(x)=m_\ell(x)^{-1}\int_XW_\ell(x,z)f(z)\dd\mu(z)$.
By \cref{lem:plus}, $m_\ell\asymp\ell^d$ and
$I+P_\ell$ has a uniformly bounded inverse on $L^2$.
In view of 
$(U_\eps)_+-(U_\ell^0)_+
=W_\ell[(\tau_\eps-f_\ell)(x)+(\tau_\eps-f_\ell)(z)]$,
\eqref{eq:symmetric-equation} and the definition of $R_\ell$ give
\[
 m_\ell(x)(\tau_\eps-f_\ell)(x)
 +\int_XW_\ell(x,z)(\tau_\eps-f_\ell)(z)\dd\mu(z)
 =-R_\ell(x)-\int_XC_\eps(x,z)\dd\mu(z).
\]
Dividing by $m_\ell$ and applying $(I+P_\ell)^{-1}$ yields
\begin{equation}\label{eq:height-inverse-equation}
 \tau_\eps-f_\ell
 =-(I+P_\ell)^{-1}
 \left[
 \frac{R_\ell}{m_\ell}
 +\frac1{m_\ell}\int_XC_\eps(\cdot,z)\dd\mu(z)
 \right].
\end{equation}
Set
\[
 b_\eps=-(I+P_\ell)^{-1}
 \left[\frac{R_\ell}{m_\ell}
 \1_{\{\delta\le R_0\ell\}}\right],
 \qquad
 e_\eps=\tau_\eps-f_\ell-b_\eps.
\]
By \eqref{eq:height-comparison-residual},
the forcing defining $b_\eps$ is bounded by $C\ell^2$.
Since
$\operatorname{dist}(x,\{\delta\le R_0\ell\})
\ge(\delta(x)-R_0\ell)_+$,
\eqref{eq:inverse-tail} gives
$ |b_\eps(x)|\le C\ell^2e^{-c\delta(x)/\ell}.$ 
Equations \eqref{eq:height-inverse-equation},
\eqref{eq:height-comparison-residual}, and
\eqref{eq:switching-small}, together with \cref{lem:plus}, give
\[
 \|e_\eps\|_2
 \le C\left(
 \ell^{3+\beta}
 +\left\|\frac1{m_\ell}
 \int_XC_\eps(\cdot,z)\dd\mu(z)\right\|_2
 \right)
 =o(\ell^3).
\]
Moreover, \eqref{eq:preliminary-error-bounds} and the bound
for $b_\eps$ imply $\|e_\eps\|_\infty\le C\ell^2$.
This proves \eqref{eq:height-decomp}.
On each compact subset of $\intr X$, $b_\eps=o(\ell^3)$
uniformly. Since $q_\eps(x,T(x))=2\tau_\eps(x)$ by \eqref{eq:common-gap-bounds}, the decomposition \eqref{eq:height-decomp} then gives \eqref{eq:height-local}.
\end{proof}

\subsection{Slack and radial expansions}
To identify the biases of $z_\eps$ and $m_\eps$ at order $\ell^2$,
we expand $q_\eps(x,T(x)+\ell B_xw)$ to order $\ell^3$.
Define
\begin{equation}\label{eq:tensor}
\begin{gathered}
 \mathcal T_x[v_1,v_2,v_3]
 =D^3u_0(x)[H(x)^{-1/2}v_1,H(x)^{-1/2}v_2,H(x)^{-1/2}v_3],\\
 \zeta(x)=H(x)^{-1/2}\nabla\log\det H(x).
\end{gathered}
\end{equation}
Jacobi's formula gives
$\zeta_i(x)=\sum_{j=1}^d\mathcal T_x[e_i,e_j,e_j]$.
Set
\begin{equation}\label{eq:LP}
 \lambda(x)=\sigma(x)^{-1}H(x)^{-1/2}
 (\nabla g_0(x)+\tfrac12\nabla h(x)),\qquad
 P(x,w)=\lambda(x)\cdot w
 +\tfrac{\sigma(x)}6\mathcal T_x[w,w,w].
\end{equation}
The next lemma gives the expansion
$q_\eps(x,T(x)+\ell B_xw)
=\ell^2\sigma(x)^2[h_d-|w|^2/2+\ell P(x,w)]$
up to the stated remainders.
The term $P$ determines the first correction to the positive set
of $q_\eps$ and to its positive part, which are used to compute $z_\eps$ and $m_\eps$, respectively.
\begin{lemma}[Slack expansion under small displacements]\label{thm:slack}
Fix $M>0$. Let $E_\ell\subset\intr X$ be open and convex, and let $w_\ell:E_\ell\to B_M$ be such that
$\Phi_\ell(x)=S(T(x)+\ell B_xw_\ell(x))$
is well-defined and satisfies the hypotheses of \cref{lem:sampling}, with fixed displacement and derivative bounds. Then, for almost every $x\in E_\ell$,
\begin{equation}\label{eq:moving-slack-decomp}
 \frac{q_\eps(x,T(x)+\ell B_xw_\ell(x))}{\ell^2\sigma(x)^2}
 =h_d-\tfrac12|w_\ell(x)|^2+\ell P(x,w_\ell(x))+E_\eps(x)+B_\eps(x),
\end{equation}
where $\|E_\eps\|_{L^2(E_\ell)}=o(\ell)$, $\|E_\eps\|_{L^\infty(E_\ell)}\le C$, and $|B_\eps(x)|\le Ce^{-c\delta(x)/\ell}$. These three estimates are uniform over maps satisfying the same displacement and derivative bounds. In particular, for every compact $K\subset\intr X$,
\begin{equation}\label{eq:first-slack}
 \ell^{-1}\bigl[Q^{\rm r}_{\eps,x}(w)-h_d+|w|^2/2\bigr]\longrightarrow P(x,w)\qquad\text{strongly in }L^2(K\times B_M).
\end{equation}
\end{lemma}
\begin{proof}
Set $y=T(x)+\ell B_xw_\ell(x)$ and $z=\Phi_\ell(x)=S(y)$.
Since $DS(T(x))B_x=\Lambda_x$, Taylor's formula gives
\begin{equation}\label{eq:source-displacement-expansion}
 z-x=\ell\Lambda_xw_\ell(x)+O(\ell^2),
\end{equation}
with a remainder bounded by $C\ell^2|w_\ell(x)|^2$.
By \eqref{eq:gap-slack-identity} and \eqref{eq:errors},
\begin{equation}\label{eq:moving-slack-identity}
 q_\eps(x,y)
 =-G_0(x,y)+\tau_\eps(x)+\tau_\eps(z)
 +\ell^2[g_\eps(z)-g_\eps(x)].
\end{equation}
Equations \eqref{eq:height-decomp} and
\eqref{eq:displacement-composition-bound} give
$\|e_\eps+e_\eps\circ\Phi_\ell\|_{L^2(E_\ell)}=o(\ell^3)$.
Since $|\delta(z)-\delta(x)|\le C\ell$,
\eqref{eq:height-decomp} gives
$ |B_\eps(x)|\le Ce^{-c\delta(x)/\ell}$, where $B_\eps(x):=\frac{b_\eps(x)+b_\eps(z)}{\ell^2\sigma(x)^2}$.
Moreover, \eqref{eq:source-displacement-expansion} and
$h\in C^{1,\beta}$ imply
\[
 \tfrac{\ell^2}{2}[h(x)+h(z)]
 =\ell^2h(x)+\tfrac{\ell^3}{2}
 \nabla h(x)\cdot\Lambda_xw_\ell(x)+O(\ell^{3+\beta}).
\]
By \eqref{eq:moving-increments} and
\eqref{eq:source-displacement-expansion},
\[
 g_\eps(z)-g_\eps(x)
 =\ell\nabla g_0(x)\cdot\Lambda_xw_\ell(x)+o(\ell)
 \qquad\text{in }L^2(E_\ell).
\]
 Differentiating $D^2v_0(T(x))=H(x)^{-1}$ gives
\begin{equation}\label{eq:dual-cubic-identity}
 D^3v_0(T(x))[B_xw,B_xw,B_xw]
 =-\sigma(x)^3\mathcal T_x[w,w,w].
\end{equation}
Taylor's formula and \eqref{eq:dual-cubic-identity} yield
\[
 G_0(x,T(x)+\ell B_xw)
 =\tfrac12\ell^2\sigma(x)^2|w|^2
 -\tfrac16\ell^3\sigma(x)^3\mathcal T_x[w,w,w]
 +O(\ell^{3+\beta}).
\]
Substituting these expansions into
\eqref{eq:moving-slack-identity} and dividing by
$\ell^2\sigma(x)^2$ proves \eqref{eq:moving-slack-decomp},
with $P$ from \eqref{eq:LP} and
$\|E_\eps\|_{L^2(E_\ell)}=o(\ell)$.
All estimates are uniform over the indicated maps.
Compositions of Borel representatives are well-defined
up to null sets by \eqref{eq:displacement-composition-bound}.
Equations \eqref{eq:unregularized-gap},
\eqref{eq:gap-slack-identity}, and \eqref{eq:G-inputs}
bound the normalized slack uniformly. Hence
\eqref{eq:moving-slack-decomp} also gives
$\|E_\eps\|_\infty\le C$.

For $|w|\le M$, set
$\Phi_{\ell,w}(x)=S(T(x)+\ell B_xw)$.
By \eqref{eq:transport-boundary-distance}, these maps are
well-defined on $\{\delta>L_M\ell\}$ for fixed sufficiently
large $L_M$. By \eqref{eq:source-displacement-expansion},
$\sup_{x\in\{\delta>L_M\ell\}}|\Phi_{\ell,w}(x)-x|\le C_M\ell$
for all $|w|\le M$. Furthermore, 
$D\Phi_{\ell,w}=I+O(\ell)$ due to $S\circ T=\id$,
$S,T\in C^2$, and $B\in C^1$.
Thus \cref{lem:sampling} applies uniformly for $|w|\le M$.
Apply \eqref{eq:moving-slack-decomp} on
$\{\delta>L_M\ell\}$ and integrate in $w$.
For small $\ell$, this domain contains $K$, and
$\ell^{-1}B_\eps\to0$ uniformly on $K\times B_M$.
This proves \eqref{eq:first-slack}.
\end{proof}

For $x\in X_\ell=\{\delta>L_0\ell\}$,
\eqref{eq:transport-boundary-distance} gives
$\dist(T(x),\partial Y)\ge cL_0\ell$.
Choose $L_0$ so that $cL_0>R$, with $R$ from
\eqref{eq:prelocalization}. Then
$S_{\eps,x}\subset\overline B(T(x),R\ell)\subset\intr Y$.
By \eqref{eq:curvature-input}, $q_\eps(x,\cdot)$ is strictly
concave, and \eqref{eq:graph-positive} gives
$q_\eps(x,T(x))>0$.
Thus, for each $\theta\in\Sn$, there is a unique
$R_\eps(x,\theta)>0$ such that
$q_\eps(x,T(x)+\ell B_xR_\eps(x,\theta)\theta)=0$, and
\begin{equation}\label{eq:radial}
 \partial S_{\eps,x}
 =\{T(x)+\ell B_xR_\eps(x,\theta)\theta:\theta\in\Sn\},
 \qquad c\le R_\eps(x,\theta)\le C.
\end{equation}
The radius bounds follow from \eqref{eq:tubes} and the
uniform bounds on $B_x$ and $B_x^{-1}$.
To determine $R_\eps(x,\theta)-a_d$, we compare the slack at the fixed reference point
$T(x)+\ell B_xa_d\theta$ with its zero at
$T(x)+\ell B_xR_\eps(x,\theta)\theta$.
The following lemma quantifies the resulting radial displacement.

\begin{lemma}[Radial expansion of the sections]\label{thm:radial}
Define
\begin{equation}\label{eq:F}
 F(x,\theta)=\frac{P(x,a_d\theta)}{a_d}
 =\lambda(x)\cdot\theta
 +\frac{\sigma(x)h_d}{3}\mathcal T_x[\theta,\theta,\theta].
\end{equation}
There exist $L_0,C,c,\eps_0>0$ such that, for every
$0<\eps<\eps_0$, there are jointly measurable functions
$r_\eps^{\rm rad},b_\eps^{\rm rad}:X_\ell\times\Sn\to\R$
satisfying, for every $\theta\in\Sn$ and almost every $x\in X_\ell$,
\begin{equation}\label{eq:radial-decomp}
\begin{gathered}
 R_\eps(x,\theta)
 =a_d+\ell F(x,\theta)
 +r_\eps^{\rm rad}(x,\theta)+b_\eps^{\rm rad}(x,\theta),\\[.2em]
 |r_\eps^{\rm rad}(x,\theta)|\le C,\qquad
 |b_\eps^{\rm rad}(x,\theta)|
 \le Ce^{-c\delta(x)/\ell},\qquad
 \sup_{\theta\in\Sn}
 \|r_\eps^{\rm rad}(\cdot,\theta)\|_{L^2(X_\ell)}=o(\ell).
\end{gathered}
\end{equation}
\end{lemma}

\begin{proof}
In view of \eqref{eq:radial}, we can choose $M>\max\{a_d,\sup_{\eps,x,\theta}R_\eps(x,\theta)\}$. By \eqref{eq:transport-boundary-distance},
$\dist(T(x),\partial Y)\ge c_0\delta(x)$.
Taking $c_0L_0>M\sup_X\|B_x\|$ ensures
$T(x)+\ell B_x\overline B_M\subset\intr Y$ for $x\in X_\ell$.
Equations \eqref{eq:two-point-profile},
\eqref{eq:row-coordinate-change}, and
\eqref{eq:profile-exponential} yield
\begin{equation}\label{eq:deep-slack-comparison}
 \|Q^{\rm r}_{\eps,x}-(h_d-|\cdot|^2/2)\|_{C^1(B_M)}
 \le\eta_\eps^{\rm prof}(x)
 :=Ce^{-c\delta(x)/\ell}+\vartheta_\eps,
 \qquad \vartheta_\eps\to0,
\end{equation}
uniformly in $x\in X_\ell$, with derivatives taken in $w$.
Indeed, when $\delta(x)<\delta_*$, every argument $s$ of $\varphi_d$ arising for $w\in B_M$ satisfies
$|s-t_\eps(x)|\le C_M$ and $s\ge c\delta(x)/\ell-C_M$ by \eqref{eq:boundary-coordinates}.
Thus \eqref{eq:profile-exponential} bounds
$|\varphi_d(s)-h_d/2|+|\varphi_d'(s)|$
by $Ce^{-c\delta(x)/\ell}$.
For $\delta(x)\ge\delta_*$, the comparison profile is $h_d/2$.
The uniform errors in \eqref{eq:two-point-profile} and
\eqref{eq:row-coordinate-change} are included in $\vartheta_\eps$.

For $\theta\in\Sn$, the maps
$\Phi_{\ell,\theta}(x)=S(T(x)+\ell B_xa_d\theta)$ satisfy
$|\Phi_{\ell,\theta}(x)-x|\le C\ell$ and
$\|D\Phi_{\ell,\theta}(x)-I\|\le C\ell$ on $X_\ell$,
with constants independent of $\theta$.
Thus \cref{thm:slack} and $a_d^2=2h_d$ give
\begin{equation}\label{eq:fixed-sphere-slack}
 Q^{\rm r}_{\eps,x}(a_d\theta)
 =\ell P(x,a_d\theta)+E_\eps(x,\theta)+B_\eps(x,\theta),
\end{equation}
with jointly measurable remainders satisfying
$\sup_\theta\|E_\eps(\cdot,\theta)\|_{L^2(X_\ell)}=o(\ell)$,
$|E_\eps|\le C$, and
$|B_\eps(x,\theta)|\le Ce^{-c\delta(x)/\ell}$.
Using $Q^{\rm r}_{\eps,x}(R_\eps(x,\theta)\theta)=0$,
\eqref{eq:deep-slack-comparison} gives
$|R_\eps(x,\theta)^2-a_d^2|\le2\eta_\eps^{\rm prof}(x)$,
hence $|R_\eps(x,\theta)-a_d|\le C\eta_\eps^{\rm prof}(x)$.
Set
\[
 A_\eps(x,\theta)
 =-\int_0^1\nabla_wQ^{\rm r}_{\eps,x}
 \bigl((a_d+t(R_\eps(x,\theta)-a_d))\theta\bigr)
 \cdot\theta\dd t.
\]
By \eqref{eq:deep-slack-comparison},
$|A_\eps(x,\theta)-a_d|\le C\eta_\eps^{\rm prof}(x)$.
Since $\delta(x)>L_0\ell$, choosing $L_0$ large and then
$\eps$ small gives $a_d/2\le A_\eps\le3a_d/2$.
Continuity of $R_\eps$ in $(x,\theta)$ implies joint
measurability of $A_\eps$.
The fundamental theorem of calculus and
\eqref{eq:fixed-sphere-slack} yield
\[
 R_\eps(x,\theta)-a_d
 =\frac{\ell P(x,a_d\theta)+E_\eps(x,\theta)+B_\eps(x,\theta)}
 {A_\eps(x,\theta)}.
\]
Define
\[
 r_\eps^{\rm rad}(x,\theta)
 =\frac{E_\eps(x,\theta)}{A_\eps(x,\theta)}
 +\ell P(x,a_d\theta)
 \left(\frac1{A_\eps(x,\theta)}-\frac1{a_d}\right),
 \qquad
 b_\eps^{\rm rad}(x,\theta)
 =\frac{B_\eps(x,\theta)}{A_\eps(x,\theta)}.
\]
These functions are jointly measurable and uniformly bounded.
By \eqref{eq:F}, they satisfy \eqref{eq:radial-decomp}, with
$|b_\eps^{\rm rad}(x,\theta)|\le Ce^{-c\delta(x)/\ell}$.
Moreover, \eqref{eq:deep-slack-comparison} and dominated
convergence give
$\|\eta_\eps^{\rm prof}\|_{L^2(X_\ell)}\to0$. Therefore,
\[
 \sup_{\theta\in\Sn}
 \|r_\eps^{\rm rad}(\cdot,\theta)\|_{L^2(X_\ell)}
 \le C\sup_{\theta\in\Sn}
 \|E_\eps(\cdot,\theta)\|_{L^2(X_\ell)}
 +C\ell\|\eta_\eps^{\rm prof}\|_{L^2(X_\ell)}
 =o(\ell). \qedhere
\]
\end{proof}

\subsection{Proof of the second-order expansion}\label{sec:bias-proof}
\begin{proof}[Proof of \cref{thm:centers}]
The limiting normalized sections and densities are symmetric,
so their first moments vanish. We compute their 
corrections at order $\ell$. Multiplication by $\ell B_x$ then gives the
transport biases at order $\ell^2$.

Choose $M>a_d$ larger than the upper radius bound in
\eqref{eq:radial}. By \eqref{eq:transport-boundary-distance},
increasing $L_0$ ensures
$T(x)+\ell B_x\overline B_M\subset\intr Y$ for $x\in X_\ell$.
Since $\rho_1\in C^{1,\beta}$ is bounded away from zero,
Taylor's formula gives
\begin{equation}\label{eq:density-tilt}
 \frac{\rho_1(T(x)+\ell B_xw)}{\rho_1(T(x))}
 =1+\ell d_1(x)\cdot w+O(\ell^{1+\beta}),\qquad
 d_1(x)=B_x\nabla_y\log\rho_1(T(x)),
\end{equation}
uniformly on $X_\ell\times B_M$.

\emph{Potential gradients.}
Set
\[
 M_\eps(x)=\int_{\mathcal S_{\eps,x}}
 \frac{\rho_1(T(x)+\ell B_xw)}{\rho_1(T(x))}\dd w,\qquad
 N_\eps(x)=\int_{\mathcal S_{\eps,x}}
 w\frac{\rho_1(T(x)+\ell B_xw)}{\rho_1(T(x))}\dd w.
\]
Changing variables $y=T(x)+\ell B_xw$ in
\eqref{eq:barycenters} cancels the common factors
$\ell^d\det B_x$ and $\rho_1(T(x))$, giving
$\ell^{-1}B_x^{-1}V_\eps^z(x)=\frac{N_\eps(x)}{M_\eps(x)}.$
By \eqref{eq:radial} and the density bounds,
$c\le M_\eps(x)\le C$ on $X_\ell$. Radial integration gives
\begin{equation}\label{eq:radial-mass-moment}
\begin{aligned}
 |\mathcal S_{\eps,x}|
 &=\frac1d\int_{\Sn}R_\eps(x,\theta)^d
 \dd\mathcal H^{d-1}(\theta),\quad
 \int_{\mathcal S_{\eps,x}}w\dd w
 &=\frac1{d+1}\int_{\Sn}\theta R_\eps(x,\theta)^{d+1}
 \dd\mathcal H^{d-1}(\theta).
\end{aligned}
\end{equation}
For $p\in\{d,d+1\}$, \eqref{eq:radial-decomp} and boundedness
of $F,r_\eps^{\rm rad},b_\eps^{\rm rad}$ imply
\[
 \left|R_\eps^p-a_d^p
 -pa_d^{p-1}(\ell F+r_\eps^{\rm rad})\right|
 \le C|\ell F+r_\eps^{\rm rad}|^2
 +C|b_\eps^{\rm rad}|.
\]
Indeed, first remove $b_\eps^{\rm rad}$ using boundedness
of the derivative of $s\mapsto s^p$ on a fixed bounded interval,
and then apply Taylor's formula at $a_d$.
Furthermore,
\[
 \sup_{\theta\in\Sn}
 \|(\ell F+r_\eps^{\rm rad})^2\|_{L^2(X_\ell)}
 \le C\ell^2+
 C\sup_{\theta\in\Sn}
 \|r_\eps^{\rm rad}(\cdot,\theta)\|_{L^2(X_\ell)}
 =o(\ell).
\]
Minkowski's inequality gives
\[
 \left\|\int_{\Sn}|\ell F+r_\eps^{\rm rad}|^2
 \dd\mathcal H^{d-1}(\theta)\right\|_{L^2(X_\ell)}
 \le\int_{\Sn}
 \|(\ell F+r_\eps^{\rm rad})^2\|_{L^2(X_\ell)}
 \dd\mathcal H^{d-1}(\theta)
 =o(\ell).
\]
The terms linear in $r_\eps^{\rm rad}$ also have
$L^2$ norm $o(\ell)$, while the terms containing
$b_\eps^{\rm rad}$ are bounded by $Ce^{-c\delta(x)/\ell}$.
Since $\int_{\Sn}\theta\dd\mathcal H^{d-1}(\theta)=0$,
\eqref{eq:radial-mass-moment} therefore gives
\[
 \int_{\mathcal S_{\eps,x}}w\dd w
 =\ell a_d^d\int_{\Sn}\theta F(x,\theta)
 \dd\mathcal H^{d-1}(\theta)
 +\big[\text{an }o(\ell)\text{ remainder in }L^2(X_\ell)\big]
 +O(e^{-c\delta(x)/\ell}).
\]
Note that $|B_{a_d}|=\omega_da_d^d$.
By \eqref{eq:radial-decomp} and dominated convergence,
\[
 \sup_{\theta\in\Sn}
 \|R_\eps(\cdot,\theta)-a_d\|_{L^2(X_\ell)}
 \le C\ell+o(\ell)
 +C\|e^{-c\delta/\ell}\|_{L^2(X_\ell)}
 \longrightarrow0.
\]
Since $R_\eps$ is bounded, \eqref{eq:radial-mass-moment}
implies
$\||\mathcal S_{\eps,\cdot}|-\omega_da_d^d\|_{L^2(X_\ell)}\to0$.
Equation \eqref{eq:density-tilt} also gives
$|M_\eps(x)-|\mathcal S_{\eps,x}||\le C\ell$.
Consequently,
\[
 \|M_\eps-\omega_da_d^d\|_{L^2(X_\ell)}\longrightarrow0,\qquad
 \|M_\eps^{-1}-(\omega_da_d^d)^{-1}\|_{L^2(X_\ell)}
 \le C\|M_\eps-\omega_da_d^d\|_{L^2(X_\ell)}
 \longrightarrow0.
\]
For the density contribution to $N_\eps$, radial integration
similarly gives
\[
 \int_{\mathcal S_{\eps,x}}w\otimes w\dd w
 =\frac1{d+2}\int_{\Sn}\theta\otimes\theta\,
 R_\eps(x,\theta)^{d+2}\dd\mathcal H^{d-1}(\theta)
 \longrightarrow
 \int_{B_{a_d}}w\otimes w\dd w=I
 \quad\text{in }L^2(X_\ell).
\]
Thus \eqref{eq:density-tilt} gives
\[
 N_\eps(x)
 =\int_{\mathcal S_{\eps,x}}w\dd w
 +\ell\left(\int_{\mathcal S_{\eps,x}}w\otimes w\dd w\right)d_1(x)
 +O(\ell^{1+\beta}).
\]
Replacing the second-moment matrix by $I$ costs $o(\ell)$
in $L^2(X_\ell)$ because $d_1$ is bounded. Hence
\begin{equation}\label{eq:restricted-first-moment-expansion}
 N_\eps
 =\ell\left[
 a_d^d\int_{\Sn}\theta F(x,\theta)\dd\mathcal H^{d-1}(\theta)
 +d_1(x)\right]
 +\widetilde r_\eps+\widetilde b_\eps,
\end{equation}
where $\|\widetilde r_\eps\|_{L^2(X_\ell)}=o(\ell)$ and
$|\widetilde b_\eps(x)|\le Ce^{-c\delta(x)/\ell}$. 
The coefficient of $\ell$ in
\eqref{eq:restricted-first-moment-expansion} is bounded, so
replacing its denominator by $\omega_da_d^d$ costs at most
$C\ell\|M_\eps^{-1}-(\omega_da_d^d)^{-1}\|_{L^2(X_\ell)}=o(\ell)$.
By \eqref{eq:F} and the divergence theorem,
\begin{equation}\label{eq:bias-shape-integration}
\begin{aligned}
 a_d^d\int_{\Sn}\theta F(x,\theta)\dd\mathcal H^{d-1}(\theta)
 &=\int_{\partial B_{a_d}}
 P(x,w)\frac{w}{a_d}\dd\mathcal H^{d-1}(w)\\
 &=\int_{B_{a_d}}\nabla_wP(x,w)\dd w
 =\omega_da_d^d\lambda(x)+\frac{\sigma(x)}2\zeta(x).
\end{aligned}
\end{equation}
For the last equality, \eqref{eq:LP} gives
$\partial_{w_i}P=\lambda_i+\frac{\sigma}{2}\mathcal T_x[e_i,w,w]$,
and \eqref{eq:tensor} gives
$\sum_j\mathcal T_x[e_i,e_j,e_j]=\zeta_i$. Since $\omega_da_d^d=(d+2)/(2h_d)$, we obtain
\begin{equation}\label{eq:normalized-barycenter-bias}
 \ell^{-1}B_x^{-1}V_\eps^z
 =\ell\left[
 \lambda+\frac{\sigma h_d}{d+2}\zeta
 +\frac{2h_d}{d+2}d_1\right]+r_\eps^z+b_\eps^z,
\end{equation}
with $\|r_\eps^z\|_{L^2(X_\ell)}=o(\ell)$ and
$|b_\eps^z(x)|\le Ce^{-c\delta(x)/\ell}$.
All coefficient fields are evaluated at $x$.

\emph{Conditional means.}
Equation \eqref{eq:row-density-exact} gives
$\ell^{-1}B_x^{-1}V_\eps^m(x)
 =\int_{B_M}wf^{\rm r}_{\eps,x}(w)\dd w.$
\cref{thm:slack} applies to $\Phi_{\ell,w}(x)=S(T(x)+\ell B_xw)$ for  $|w|\le M$ with constants independent of $w$.
Equations \eqref{eq:row-density-exact} and \eqref{eq:moving-slack-decomp} yield
\[
 f^{\rm r}_{\eps,x}(w)
 =
 \frac{\rho_1(T(x)+\ell B_xw)}{\rho_1(T(x))}
 \bigl(h_d-|w|^2/2+\ell P(x,w)
       +E_\eps(x,w)+B_\eps(x,w)\bigr)_+,
\]
where
$\|E_\eps\|_{L^2(X_\ell\times B_M)}=o(\ell)$ and
$|B_\eps(x,w)|\le Ce^{-c\delta(x)/\ell}$. Remove $E_\eps$ first and then $B_\eps$ from the positive part.
Since $|s_+-t_+|\le|s-t|$ and the density ratio is bounded,
the corresponding first moment errors satisfy, respectively,
\[
 \left\|\int_{B_M}C|w||E_\eps(x,w)|\dd w\right\|_{L^2(X_\ell)}
 \le C_M\|E_\eps\|_{L^2(X_\ell\times B_M)}
 =o(\ell)
\]
and
$\int_{B_M}C|w||B_\eps(x,w)|\dd w
\le C_Me^{-c\delta(x)/\ell}$.
It remains to expand the first moment of
\[
 \frac{\rho_1(T(x)+\ell B_xw)}{\rho_1(T(x))}
 (h_d-|w|^2/2+\ell P(x,w))_+.
\]
Boundedness of $P$ on $X\times B_M$ gives
\[
 \left|
 (h_d-|w|^2/2+\ell P(x,w))_+
 -\kappa_d(w)-\ell P(x,w)\1_{B_{a_d}}(w)
 \right|
 \le C\ell\1_{\{|h_d-|w|^2/2|\le C\ell\}}.
\]
The set on the right has volume $O(\ell)$.
Thus its contribution to the first moment is $O(\ell^2)$.
Multiplying by \eqref{eq:density-tilt}, the remaining errors
are $O(\ell^2+\ell^{1+\beta})$ after integration in $w$.
Since $\int_{B_{a_d}}w\kappa_d(w)\dd w=0$, we conclude that
\[
 \ell^{-1}B_x^{-1}V_\eps^m
 =\ell\int_{B_{a_d}}wP(x,w)\dd w
 +\ell\int_{B_{a_d}}w(d_1(x)\cdot w)\kappa_d(w)\dd w
 +r_\eps^m+b_\eps^m,
\]
where $\|r_\eps^m\|_{L^2(X_\ell)}=o(\ell)$ and
$|b_\eps^m(x)|\le Ce^{-c\delta(x)/\ell}$. Since $\nabla\kappa_d=-w$ on $B_{a_d}$ and
$\kappa_d=0$ on $\partial B_{a_d}$, integration by parts gives
\begin{equation}\label{eq:bias-cap-integration}
 \int_{B_{a_d}}wP(x,w)\dd w
 =\int_{B_{a_d}}\kappa_d(w)\nabla_wP(x,w)\dd w
 =\lambda+\frac{\sigma h_d}{d+4}\zeta.
\end{equation}
Indeed, \eqref{eq:cap-normalization} and
\eqref{eq:cap-vector-moments} give
$\int\kappa_d\dd w=1$ and
$\int w_jw_k\kappa_d(w)\dd w
=\frac{2h_d}{d+4}\delta_{jk}$, and then inserting
$\partial_{w_i}P=\lambda_i+\frac{\sigma}{2}\mathcal T_x[e_i,w,w]$
proves the last equality.
The same second moment identity gives
\[
 \int_{B_{a_d}}w(d_1\cdot w)\kappa_d(w)\dd w
 =\frac{2h_d}{d+4}d_1.
\]
Therefore,
\begin{equation}\label{eq:normalized-mean-bias}
 \ell^{-1}B_x^{-1}V_\eps^m
 =\ell\left[
 \lambda+\frac{\sigma h_d}{d+4}\zeta
 +\frac{2h_d}{d+4}d_1\right]+r_\eps^m+b_\eps^m.
\end{equation}

\emph{Bias coefficients and remainders.}
Differentiating \eqref{eq:transport-jacobian} and
\eqref{eq:scale} gives
\begin{equation}\label{eq:identities}
 H\nabla_y\log\rho_1(T)
 =\nabla\log\rho_0-\nabla\log\det H,\qquad
 \nabla h=\frac{h}{d+2}
 (\nabla\log\det H-2\nabla\log\rho_0).
\end{equation}
By \eqref{eq:LP}, \eqref{eq:tensor},
\eqref{eq:density-tilt}, and $h=\sigma^2h_d$,
\[
 B_x\lambda=\nabla g_0+\tfrac12\nabla h,\qquad
 B_x(\sigma h_d\zeta)=h\nabla\log\det H,\qquad
 B_x(2h_dd_1)=2hH\nabla_y\log\rho_1(T).
\]
Using \eqref{eq:identities}, the last two terms satisfy
\[
 B_x(\sigma h_d\zeta+2h_dd_1)
 =h(2\nabla\log\rho_0-\nabla\log\det H)
 =-(d+2)\nabla h.
\]
Multiplying \eqref{eq:normalized-barycenter-bias} and
\eqref{eq:normalized-mean-bias} by $\ell B_x$ therefore yields
\[
\begin{aligned}
 V_\eps^z
 &=\ell^2(\nabla g_0-\tfrac12\nabla h)
   +\ell B_xr_\eps^z+\ell B_xb_\eps^z,\\
 V_\eps^m
 &=\ell^2(\nabla g_0-\tfrac{d}{2(d+4)}\nabla h)
   +\ell B_xr_\eps^m+\ell B_xb_\eps^m.
\end{aligned}
\]
For each error, take
$\mathcal R_\eps=\ell B_xr_\eps^z$ or $\ell B_xr_\eps^m$
and $\mathcal B_\eps=\ell B_xb_\eps^z$ or $\ell B_xb_\eps^m$.
Boundedness of $B_x$ gives $ \|\mathcal R_\eps\|_{L^2(X_\ell)}=o(\ell^2)$ and 
$ |\mathcal B_\eps(x)|\le C\ell e^{-c\delta(x)/\ell}.$ 
By \eqref{eq:prelocalization} and \eqref{eq:barycenters},
both $z_\eps(x)$ and $m_\eps(x)$ are averages of points
within distance $C\ell$ of $T(x)$. Hence
$\|V_\eps^z\|_\infty+\|V_\eps^m\|_\infty\le C\ell$.
Since the bias fields are bounded,
$\mathcal R_\eps=V_\eps-\ell^2b_V-\mathcal B_\eps$
also satisfies $\|\mathcal R_\eps\|_\infty\le C\ell$.
This proves the forward cases of \eqref{eq:field-decomp}.

\emph{Reverse errors.}
Interchanging $\mu$ and $\nu$ replaces $T$ by $S$,
$g_\eps$ by $-g_\eps\circ S$, and $h$ by $h\circ S$.
By \cref{prop:corrector-convergence},
$-g_\eps\circ S\to-g_0\circ S$ uniformly.
The forward calculation therefore applies with these quantities.
Since $DS(T(x))=H(x)^{-1}$,
\[
 \nabla_y(-g_0\circ S)(T(x))=-H(x)^{-1}\nabla g_0(x),
 \qquad
 \nabla_y(h\circ S)(T(x))=H(x)^{-1}\nabla h(x),
\]
which gives \eqref{eq:reverse-biases}. By \eqref{eq:transport-boundary-distance},
$\dist(T(x),\partial Y)\ge c\delta(x)$.
After increasing $L_0$, the reverse estimates thus apply
at $T(x)$ for every $x\in X_\ell$, and their exponential
remainders are bounded by $C\ell e^{-c\delta(x)/\ell}$.
Changing variables $y=T(x)$, with $\det DS$ bounded,
preserves the $o(\ell^2)$ bounds in $L^2$.
The $L^\infty$ bounds are preserved by composition.
This proves \eqref{eq:field-decomp} for all four errors.

Finally, for each compact $K\subset\intr X$, $K\subset X_\ell$
for small $\ell$ and
\[
 \|\ell^{-2}\mathcal B_\eps\|_{L^\infty(K)}
 \le C\ell^{-1}e^{-c\dist(K,\partial X)/\ell}\longrightarrow0,
 \qquad
 \|\ell^{-2}\mathcal R_\eps\|_{L^2(K)}\longrightarrow0,
\]
so \eqref{eq:centers-local} also follows.
\end{proof}

\begin{corollary}[Potential expansions in $H^1_{\mathrm{loc}}$]\label{cor:H1}
Strongly in $H^1_{\rm loc}(\intr X)$,
\begin{equation}\label{eq:local-H1-potential-limits}
 \ell^{-2}r_\eps\longrightarrow g_0-h/2,\qquad \ell^{-2}(s_\eps\circ T)\longrightarrow-g_0-h/2.
\end{equation}
Consequently,
\[
 g_\eps\longrightarrow g_0,\qquad \ell^{-2}\tau_\eps\longrightarrow h/2
 \qquad\text{strongly in }H^1_{\rm loc}(\intr X).
\]
Moreover, $\ell^{-2}(m_\eps-z_\eps)\to2\nabla h/(d+4)$ strongly in $L^2_{\rm loc}(\intr X)$.
\end{corollary}
\begin{proof}
The gradient identities are $\nabla r_\eps=V^z_\eps$ and $\nabla(s_\eps\circ T)=H\widehat V^z_\eps$. After division by $\ell^2$, \eqref{eq:centers-local} and \eqref{eq:reverse-biases} give their local $L^2$ limits as $\nabla g_0-\nabla h/2$ and $-\nabla g_0-\nabla h/2$, respectively. Together with convergence of the functions in \eqref{eq:global-bulk-bias}, this proves \eqref{eq:local-H1-potential-limits}. Taking half the difference and minus half the sum in \eqref{eq:local-H1-potential-limits}, respectively, gives the asserted limits for $g_\eps$ and $\ell^{-2}\tau_\eps$ by \eqref{eq:errors}. Subtracting the two forward bias limits gives the final assertion, since $b_m-b_z=2\nabla h/(d+4)$.
\end{proof}

\section{Global bias asymptotics}\label{sec:norms}

This section proves the global $L^p$ asymptotics in \cref{thm:main-bias} and the weak limits of the bias measures in \cref{thm:Lp} below. 
For each error field $V_\eps$ in \eqref{eq:bias-error-fields}, we derive
weak limits of the normalized measures $|V_\eps|^p\mu$ and
$V_\eps\mu$ from \eqref{eq:boundary-input} and
\eqref{eq:field-decomp}. These limits describe where the error
concentrates and, for the vector measures, its direction.
Testing the scalar limits against $1$ then gives the norm asymptotics.

Let $b_V$ be the corresponding interior coefficient in
\cref{thm:centers}. Set $p_V=-\varphi_d'$ for potential-gradient
errors and $p_V=\beta_d$ for conditional-mean errors.
For $\delta(x)<\delta_*$, set $a_V=n$ for forward errors
and $a_V=H^{-1}n$ for reverse errors. Then the boundary term
in \eqref{eq:boundary-input} is
$\ell(\sigma/\gamma)p_V(t_\eps)a_V$.
Each $a_V$ is continuous and satisfies $c\le|a_V|\le C$.
Define
\begin{equation}\label{eq:bias-profile-integrals}
 c_{V,p}=\int_0^\infty|p_V(t)|^p\dd t
 \quad(1\le p<\infty),\qquad
 j_V=\int_0^\infty p_V(t)\dd t.
\end{equation}
By \eqref{eq:profile-exponential}, these integrals converge
absolutely. Continuity and $p_V(0)>0$ from
\cref{lem:profile-regularity} imply $c_{V,p}>0$.
For potential-gradient errors,
$c_{V,p}=A_{d,p}$ and
$j_V=\varphi_d(0)-h_d/2$ by \eqref{eq:profile-deep-limit}.
For conditional-mean errors,
$c_{V,p}=B_{d,p}$ and $j_V=h_d/(d+4)$
by \eqref{eq:beta-sum-rule}. In the following result, \eqref{eq:absolute-bias-measures}
holds weak* in $\mathcal M(X)=C(X)^*$, the space of finite
signed Radon measures on $X$, and \eqref{eq:signed-measure}
holds weak* in $\mathcal M(X;\R^d)=C(X;\R^d)^*$. 
\begin{theorem}[Weak-* limits of the bias measures]\label{thm:Lp}\label{cor:measures}
Under \cref{ass:data}, each error $V_\eps$ in
\eqref{eq:bias-error-fields} satisfies, for every $1\le p<\infty$,
\begin{equation}\label{eq:absolute-bias-measures}
 \ell^{-(p+1)}|V_\eps|^p\mu
 \overset{*}{\rightharpoonup}
 \1_{\{p=1\}}|b_V|\mu
 +c_{V,p}\rho_0\sigma^{p+1}\gamma^{1-p}|a_V|^p\,
 \mathcal H^{d-1}|_{\partial X}.
\end{equation}
Moreover,
\begin{equation}\label{eq:signed-measure}
 \ell^{-2}V_\eps\mu
 \overset{*}{\rightharpoonup}
 b_V\mu+j_V\rho_0\sigma^2a_V\,
 \mathcal H^{d-1}|_{\partial X}.
\end{equation}

\end{theorem}
\begin{proof}
Fix $V_\eps$, $f\in C(X)$, and $\Psi\in C(X;\R^d)$.
For $L>L_0$, set
$C_{\ell,L}=\{\delta\le L\ell\}$ and
$D_{\ell,L}=\{\delta>L\ell\}$.
We apply \eqref{eq:boundary-input} on $C_{\ell,L}$ and
\eqref{eq:field-decomp} on $D_{\ell,L}$, first letting
$\eps\downarrow0$ with $L$ fixed, then $L\to\infty$.

\emph{Boundary contribution.}
For sufficiently small $s_0>0$, the map
\begin{equation}\label{eq:bias-normal-coordinates}
 \Xi(y,s)=y+sn(y),\qquad
 y\in\partial X,\quad 0\le s<s_0,
\end{equation}
parametrizes $\{\delta<s_0\}$, with
$\delta(\Xi(y,s))=s$ and
$\dd x=J(y,s)\dd s\dd\mathcal H^{d-1}(y)$,
where $J(y,s)=1+O(s)$ uniformly in $y$.
For $d=1$, $\mathcal H^0$ counts the endpoints and $J=1$.
Define
\begin{equation}\label{eq:bias-boundary-profile}
 \mathcal P_V(y,r)
 =\frac{\sigma(y)}{\gamma(y)}
 p_V\!\left(\frac{r}{\sigma(y)\gamma(y)}\right)a_V(y).
\end{equation}
By \eqref{eq:boundary-coordinates},
\[
 t_\eps(\Xi(y,\ell r))
 =\frac{r}{\sigma(\Xi(y,\ell r))\gamma(\Xi(y,\ell r))}
 \longrightarrow\frac{r}{\sigma(y)\gamma(y)}
\]
uniformly on $\partial X\times[0,L]$.
Continuity of $\sigma,\gamma,a_V,p_V$ and
\eqref{eq:boundary-input} therefore give
\begin{equation}\label{eq:bias-rescaled-collar-limit}
 \sup_{\substack{y\in\partial X\\0\le r\le L}}
 \left|\ell^{-1}V_\eps(\Xi(y,\ell r))
 -\mathcal P_V(y,r)\right|\longrightarrow0.
\end{equation}
Changing variables $x=\Xi(y,\ell r)$ yields
\[
 \ell^{-(p+1)}\int_{C_{\ell,L}}f|V_\eps|^p\dd\mu
 =
 \int_{\partial X}\int_0^L
 f(\Xi(y,\ell r))\rho_0(\Xi(y,\ell r))
 |\ell^{-1}V_\eps(\Xi(y,\ell r))|^p
 J(y,\ell r)\dd r\dd\mathcal H^{d-1}(y).
\]
Thus \eqref{eq:bias-rescaled-collar-limit} implies, for $1\le p<\infty$,
\begin{align}
 \ell^{-(p+1)}\int_{C_{\ell,L}}f|V_\eps|^p\dd\mu
 &\longrightarrow
 \int_{\partial X}\rho_0(y)f(y)
 \int_0^L|\mathcal P_V(y,r)|^p\dd r
 \dd\mathcal H^{d-1}(y),
 \label{eq:truncated-boundary}\\
 \ell^{-2}\int_{C_{\ell,L}}\Psi\cdot V_\eps\dd\mu
 &\longrightarrow
 \int_{\partial X}\rho_0(y)\Psi(y)\cdot
 \int_0^L\mathcal P_V(y,r)\dd r
 \dd\mathcal H^{d-1}(y).
 \label{eq:bias-signed-collar}
\end{align}
The substitution $r=\sigma(y)\gamma(y)t$ and
\eqref{eq:bias-profile-integrals} give
\begin{equation}\label{eq:bias-boundary-integral-limit}
\begin{aligned}
 \int_0^\infty|\mathcal P_V(y,r)|^p\dd r
 &=c_{V,p}\sigma(y)^{p+1}\gamma(y)^{1-p}|a_V(y)|^p,\quad
 \int_0^\infty\mathcal P_V(y,r)\dd r
=j_V\sigma(y)^2a_V(y).
\end{aligned}
\end{equation}
By \eqref{eq:profile-exponential} and the bounds on
$\sigma,\gamma,a_V$, $|\mathcal P_V(y,r)|\le Ce^{-cr}$
uniformly in $y$. Hence the right-hand sides of
\eqref{eq:truncated-boundary} and \eqref{eq:bias-signed-collar}
converge as $L\to\infty$ to the corresponding expressions
in \eqref{eq:bias-boundary-integral-limit}.

\emph{Contribution from $\{\delta>L\ell\}$.}
Since $D_{\ell,L}\subset X_\ell$, \eqref{eq:field-decomp} gives
$V_\eps=\ell^2b_V+\mathcal R_\eps+\mathcal B_\eps$ on $D_{\ell,L}$.
The density bounds, H\"older's inequality, and
\eqref{eq:field-decomp} imply
\begin{equation}\label{eq:bias-remainder-moments}
 \int_{D_{\ell,L}}|\mathcal R_\eps|^p\dd\mu
 \le
 \begin{cases}
 \|\mathcal R_\eps\|_{L^2(X_\ell,\mu)}^p
 =o(\ell^{2p}),&1\le p\le2,\\
 \|\mathcal R_\eps\|_{L^\infty(X_\ell)}^{p-2}
 \|\mathcal R_\eps\|_{L^2(X_\ell,\mu)}^2
 =o(\ell^{p+2}),&p>2.
 \end{cases}
\end{equation}
These bounds are independent of $L>L_0$. 
For the boundary remainder $\mathcal B_\eps$, \eqref{eq:field-decomp} and
\eqref{eq:bias-normal-coordinates} give, when $L\ell<s_0$,
\[
 \int_{D_{\ell,L}}|\mathcal B_\eps|^p\dd\mu
 \le C_p\ell^p\int_{L\ell}^{s_0}e^{-cp s/\ell}\dd s
 +C_p\ell^pe^{-cp s_0/\ell}.
\]
For fixed $L$ and sufficiently small $\ell$,
$L\ell<s_0/2$ and
$e^{-cp s_0/\ell}\le C_p\ell e^{-cpL/2}$.
Consequently,
\begin{equation}\label{eq:tail-integral}
 \int_{D_{\ell,L}}|\mathcal B_\eps|^p\dd\mu
 \le C_p\ell^{p+1}e^{-c_pL},
 \qquad 1\le p<\infty,
\end{equation}
with constants independent of $L$. For $p>1$, boundedness of $b_V$ gives
$\ell^{-(p+1)}\int_{D_{\ell,L}}|\ell^2b_V|^p\dd\mu
\le C_p\ell^{p-1}\to0$.
Equation \eqref{eq:bias-remainder-moments} likewise gives
$\ell^{-(p+1)}\int_{D_{\ell,L}}|\mathcal R_\eps|^p\dd\mu\to0$.
Using $|u+v+w|^p\le3^{p-1}(|u|^p+|v|^p+|w|^p)$
and \eqref{eq:tail-integral}, we obtain
\begin{equation}\label{eq:bias-complement-bound}
 \limsup_{\eps\downarrow0}
 \ell^{-(p+1)}
 \left|\int_{D_{\ell,L}}f|V_\eps|^p\dd\mu\right|
 \le C_p\|f\|_\infty e^{-c_pL}.
\end{equation}

For $p=1$, \eqref{eq:field-decomp},
\eqref{eq:bias-remainder-moments}, and
\eqref{eq:tail-integral} give
\begin{equation}\label{eq:bias-endpoint-complement}
 \limsup_{\eps\downarrow0}
 \int_{D_{\ell,L}}|\ell^{-2}V_\eps-b_V|\dd\mu
 \le Ce^{-cL}.
\end{equation}
The reverse triangle inequality therefore yields
\[
 \limsup_{\eps\downarrow0}
 \left|
 \ell^{-2}\int_{D_{\ell,L}}f|V_\eps|\dd\mu
 -\int_{D_{\ell,L}}f|b_V|\dd\mu
 \right|
 \le C\|f\|_\infty e^{-cL}.
\]
Similarly,
\[
 \limsup_{\eps\downarrow0}
 \left|
 \ell^{-2}\int_{D_{\ell,L}}\Psi\cdot V_\eps\dd\mu
 -\int_{D_{\ell,L}}\Psi\cdot b_V\dd\mu
 \right|
 \le C\|\Psi\|_\infty e^{-cL}.
\]
Since $\mu(C_{\ell,L})\le CL\ell$ and $b_V$ is bounded,
the integrals involving $b_V$ converge, for fixed $L$, to
$\int_Xf|b_V|\dd\mu$ and $\int_X\Psi\cdot b_V\dd\mu$,
respectively.

Combining \eqref{eq:truncated-boundary} with
\eqref{eq:bias-complement-bound} for $p>1$ and the preceding estimate for $p=1$,
and then letting $L\to\infty$, proves
\eqref{eq:absolute-bias-measures}.
Equations \eqref{eq:bias-signed-collar},
\eqref{eq:bias-endpoint-complement}, and
\eqref{eq:bias-boundary-integral-limit} similarly prove
\eqref{eq:signed-measure}.
Taking $f=1$ in the scalar limit for $p=1$ also gives $\sup_{\eps<\eps_0}\ell^{-2}\int_X|V_\eps|\dd\mu<\infty$ for some $\eps_0>0$, which is the total variation bound needed for the vector measures
$\ell^{-2}V_\eps\mu$.
\end{proof}

\begin{proof}[Proof of \cref{thm:main-bias}]
The local limits \eqref{eq:main-local-bias} follow from \eqref{eq:centers-local}. For the global conclusions, we first derive norm asymptotics for all four error fields in \eqref{eq:bias-error-fields}. Testing \eqref{eq:absolute-bias-measures} with $f=1$ gives
\begin{align}
 \ell^{-(p+1)}\|V_\eps\|_{L^p(\mu)}^p&\longrightarrow c_{V,p}\int_{\partial X}\rho_0\sigma^{p+1}\gamma^{1-p}|a_V|^p\dd\mathcal H^{d-1},&&1<p<\infty,\label{eq:Lp}\\
 \ell^{-2}\|V_\eps\|_{L^1(\mu)}&\longrightarrow\int_X|b_V|\dd\mu+c_{V,1}\int_{\partial X}\rho_0\sigma^2|a_V|\dd\mathcal H^{d-1}.\label{eq:L1}
\end{align}
The factors $\rho_0,\sigma,\gamma,|a_V|$ are bounded above and away from zero on $\partial X$, which has finite positive surface measure. Since $c_{V,p}>0$, the limits in \eqref{eq:Lp} and \eqref{eq:L1} are finite and strictly positive, even when $b_V=0$. Taking $p$th roots gives the rates for finite $p$, while \eqref{eq:uniform-bias-orders} supplies $p=\infty$. Thus
\begin{equation}\label{eq:global-bias-orders}
 \|V_\eps\|_{L^p(\mu)}\asymp\ell^{1+1/p}=\eps^{(1+1/p)/(d+2)},\qquad 1\le p\le\infty.
\end{equation}

For the forward errors $V_\eps=z_\eps-T$ and $V_\eps=m_\eps-T$, we have $a_V=n$ and $|a_V|=1$ on $\partial X$; their constants $c_{V,p}$ are $A_{d,p}$ and $B_{d,p}$, respectively. Substituting these choices into \eqref{eq:Lp} and \eqref{eq:L1} gives \eqref{eq:main-Lp} and \eqref{eq:main-L1}, and \eqref{eq:global-bias-orders} gives \eqref{eq:main-bias-rates}. Choosing the same two errors in \eqref{eq:absolute-bias-measures} gives their measure limits.

For the reverse errors, composition with $T$ preserves the norms because $T_\#\mu=\nu$. Thus \eqref{eq:global-bias-orders} also gives the rates of $\widehat z_\eps-S$ and $\widehat m_\eps-S$ in $L^p(\nu)$ for $1\le p\le\infty$.
\end{proof}

\appendix
\section{Failure of Brenier graph inclusion}\label{app:graph-exclusion}

The graph inclusion can fail when $\eps$ is not small and $d\geq2$. (It holds globally when $d=1$.)

\begin{proposition}[Failure of graph inclusion]\label{prop:graph-exclusion}
For every $d\ge2$, there are probability measures $\mu=\rho_0\1_\Omega\dd x$ and $\nu=\rho_1\1_\Omega\dd y$ on the ball $\Omega=\B(0,12)\subset\R^d$, with $C^\infty$ densities bounded above and away from zero, such that their Brenier map $T$ satisfies
\begin{equation}\label{eq:graph-exclusion}
  q_{\eps}(x,T(x))<-\frac14\qquad(x\in E)
\end{equation}
at $\eps=13/8$, for some open set $E\subset\intr\Omega$ with $\mu(E)>1/64$. In particular, $\operatorname{graph}T\not\subset\spt\pi_\eps$.
\end{proposition}
\begin{proof}
\emph{Atomic marginals.}
Let $e_1,e_2$ be the first two coordinate vectors.
Set $x_1=0$, $x_2=e_1$, $x_3=e_2$,
$y_1=-8e_1-7e_2$, $y_2=-4e_1+e_2$, and $y_3=0$.
For $a=\tfrac1{32}(13,14,5)^T$ and
$b=\tfrac1{32}(12,4,16)^T$, define
$\bar\mu=\sum_{i=1}^3a_i\delta_{x_i}$ and
$\bar\nu=\sum_{j=1}^3b_j\delta_{y_j}$.
Feasible couplings are matrices $P\ge0$ satisfying
$P\mathbf1=a$ and $P^T\mathbf1=b$.
Writing $M_{ij}=x_i\cdot y_j$,  optimal transport with squared Euclidean cost
maximizes $\sum_{i,j}M_{ij}P_{ij}$.
Subtracting $1$ from the second column changes this objective
by the constant $-b_2$. The resulting matrix $K$ and a feasible coupling are
\begin{equation}\label{eq:exclusion-unregularized-plan}
 K=\begin{pmatrix}0&-1&0\\-8&-5&0\\-7&0&0\end{pmatrix},
 \qquad
 P^0=\frac1{32}\begin{pmatrix}12&0&1\\0&0&14\\0&4&1\end{pmatrix}.
\end{equation}
Since $K\le0$ and $P^0$ is supported on $\{K=0\}$,
the maximum is zero. Every maximizer is supported on
$\{K=0\}$, where the marginal constraints uniquely determine
$P^0$. Thus the unregularized optimizer
$\bar\gamma=\sum_{i,j}P^0_{ij}\delta_{(x_i,y_j)}$
satisfies $\bar\gamma(\{(0,0)\})=1/32$. At $\eps=13/8$, take $U=(0,-4,-3)$ and $V=(-4,-1,1)$.
Direct calculation gives
\begin{equation}\label{eq:exclusion-atomic-slack}
 Q=(M_{ij}-U_i-V_j)_{i,j}
 =\begin{pmatrix}4&1&-1\\0&1&3\\0&5&2\end{pmatrix},
 \qquad
 Q_+b=\eps\mathbf1,\quad Q_+^Ta=\eps\mathbf1.
\end{equation}
Hence $P^*_{ij}=a_ib_j(Q_{ij})_+/\eps$ is feasible.
Up to terms fixed by the marginals, the regularized objective is
\[
 \mathcal F(P)
 =-\sum_{i,j}M_{ij}P_{ij}
 +\frac{\eps}{2}\sum_{i,j}\frac{P_{ij}^2}{a_ib_j}.
\]
For every feasible $P$, expansion and cancellation of the
potential terms give
\begin{equation}\label{eq:exclusion-optimality-certificate}
 \mathcal F(P)-\mathcal F(P^*)
 =\frac{\eps}{2}\sum_{i,j}
 \frac{(P_{ij}-P^*_{ij})^2}{a_ib_j}
 +\sum_{i,j}(-Q_{ij})_+P_{ij}\ge0.
\end{equation}
Equality requires $P=P^*$, so $P^*$ is the unique optimizer.
In particular, $P^*_{13}=0$ and $Q_{13}=-1$. Any other solution $(\widetilde U,\widetilde V)$ of the
discrete marginal equations induces the same optimizer by
\eqref{eq:exclusion-optimality-certificate}. For every $(i,j)$ with $P^*_{ij}>0$,
$M_{ij}-\widetilde U_i-\widetilde V_j
=\eps P^*_{ij}/(a_ib_j)=Q_{ij}$.
Normalize $\widetilde U_1=0$.
Since $P^*_{12},P^*_{22},P^*_{23}>0$, these identities give
$-\widetilde V_2=1$,
$-4-\widetilde U_2-\widetilde V_2=1$, and
$-\widetilde U_2-\widetilde V_3=3$.
Thus $\widetilde V_2=-1$, $\widetilde U_2=-4$, and
$\widetilde V_3=1$, so
$M_{13}-\widetilde U_1-\widetilde V_3=-1$.

\emph{Smooth marginals.}
Choose $\phi\in C_c^\infty(B_1)$ with $\phi\ge0$ and
$\int\phi=1$, and set $\phi_n(x)=n^d\phi(nx)$.
For $n\ge2$, define on $\Omega=\B(0,12)$
\begin{equation}\label{eq:exclusion-smooth-densities}
\begin{aligned}
 \rho_{0,n}(x)
 &=(1-n^{-1})\sum_{i=1}^3a_i\phi_n(x-x_i)
 +\frac1{n|\Omega|},\ 
 \rho_{1,n}(y)
 =(1-n^{-1})\sum_{j=1}^3b_j\phi_n(y-y_j)
 +\frac1{n|\Omega|}.
\end{aligned}
\end{equation}
The densities are smooth, integrate to one, and are bounded
below by $1/(n|\Omega|)$. Their measures satisfy
$\mu_n\rightharpoonup\bar\mu$ and
$\nu_n\rightharpoonup\bar\nu$. Let $u_n,v_n$ be the regularized potentials at $\eps$,
normalized by $u_n(0)=0$, and set
$q_n(x,y)=x\cdot y-u_n(x)-v_n(y)$.
By \eqref{eq:barycenters},
$\Lip(u_n),\Lip(v_n)\le12$.
Equation \eqref{eq:marginal} at $x=0$ gives
$\int(-v_n)_+\dd\nu_n=\eps$, hence
$\min v_n\le-\eps\le\max v_n$.
Consequently,
\begin{equation}\label{eq:exclusion-potential-bounds}
 \|u_n\|_\infty\le144,\qquad
 \|v_n\|_\infty\le288+\eps.
\end{equation}
Arzel\`a--Ascoli gives relative compactness in $C(\Omega)^2$.
For any subsequential limit $(u,v)$, passing to the
limit in \eqref{eq:marginal} at $x_i$ and $y_j$ gives
\[
 \sum_j b_j(M_{ij}-u(x_i)-v(y_j))_+=\eps,
 \qquad
 \sum_i a_i(M_{ij}-u(x_i)-v(y_j))_+=\eps.
\]
The atomic calculation therefore gives
$-u(0)-v(0)=-1$. Every subsequential limit has this value,
so $q_n(0,0)\to-1$. The Lipschitz bounds imply
$|q_n(x,y)-q_n(0,0)|\le|x||y|+12(|x|+|y|)$.
Taking $r=1/100$, for sufficiently large $n$ we obtain
\begin{equation}\label{eq:exclusion-product-neighborhood}
 q_n(x,y)<-\tfrac34+r^2+24r<-\tfrac14
 \qquad(x,y\in\B(0,r)).
\end{equation}

Let $T_n$ be the Brenier map and
$\gamma_n=(\id,T_n)_\#\mu_n$.
Compactness of $\Omega$ gives convergence of the marginals
in $W_2$, hence convergence of the optimal quadratic costs.
Every weak limit of $\gamma_n$ is therefore an optimal
coupling of $\bar\mu,\bar\nu$.
Uniqueness in \eqref{eq:exclusion-unregularized-plan} implies
$\gamma_n\rightharpoonup\bar\gamma$.
For $C=B(0,r)\times B(0,r)$, Portmanteau gives
\begin{equation}\label{eq:exclusion-unregularized-mass}
 \liminf_{n\to\infty}\gamma_n(C)
 \ge\bar\gamma(C)=\frac1{32}.
\end{equation}
Set $E_n=B(0,r)\cap T_n^{-1}(B(0,r))$.
By \cref{prop:inputs}, $T_n$ is continuous, so $E_n$ is open.
Moreover, $\mu_n(E_n)=\gamma_n(C)>1/64$ for sufficiently
large $n$ by \eqref{eq:exclusion-unregularized-mass}, and
\eqref{eq:exclusion-product-neighborhood} gives
$q_n(x,T_n(x))<-1/4$ on $E_n$.
Fixing such an $n$ proves \eqref{eq:graph-exclusion}.
\end{proof}

\paragraph{Declarations.} AGS was partially supported by the RSME--Fundaci\'on BBVA Jos\'e Luis Rubio de Francia Start-up Grant, awarded in the framework of the Jos\'e Luis Rubio de Francia Prize. MN was partially supported by NSF Grants DMS-2407074 and DMS-2606731. The authors declare no competing interests. No data were created or analyzed. AI was used in the preparation of this manuscript. Any remaining errors are our own.

\printbibliography[heading=bibintoc]

@book{GrafakosClassicalFourierAnalysis,
  author    = {Grafakos, Loukas},
  title     = {Classical Fourier Analysis},
  series    = {Graduate Texts in Mathematics},
  volume    = {249},
  edition   = {3},
  publisher = {Springer},
  address   = {New York},
  year      = {2014},
  doi       = {10.1007/978-1-4939-1194-3}
}

@article{GigliTamanini.21,
    author = "N. Gigli and L. Tamanini",
    title  = "Second order differentiation formula on {${\mathrm{RCD}}^{*}(K,N)$} spaces",
  JOURNAL = {Journal of the European Mathematical Society (JEMS)},
    VOLUME = {23},
      YEAR = {2021},
    NUMBER = {5},
     PAGES = {1727--1795},
      ISSN = {1435-9855},
   MRCLASS = {53C23},
  MRNUMBER = {4244516},
       DOI = {10.4171/JEMS/1042},
       URL = {https://doi.org/10.4171/JEMS/1042},
}

@article{NWPotentials,
  author = {Nutz, Marcel and Wiesel, Johannes},
  title = {Entropic optimal transport: convergence of potentials},
  journal = {Probability Theory and Related Fields},
  volume = {184},
  pages = {401--424},
  year = {2022},
  doi = {10.1007/s00440-021-01096-8}
}

@article{CCGTGradients,
  author = {Chiarini, Alberto and Conforti, Giovanni and Greco, Giacomo and Tamanini, Luca},
  title = {Gradient estimates for the {Schr\"odinger} potentials: convergence to the {Brenier} map and quantitative stability},
  journal = {Communications in Partial Differential Equations},
  volume = {48},
  number = {6},
  pages = {895--943},
  year = {2023},
  doi = {10.1080/03605302.2023.2215527}
}

@article{LRUniform,
  author = {L\'opez-Rivera, Pablo},
  title = {A uniform rate of convergence for the entropic potentials in the quadratic {Euclidean} setting},
  journal = {ESAIM: Control, Optimisation and Calculus of Variations},
  volume = {32},
  eid = {27},
  year = {2026},
  doi = {10.1051/cocv/2026012}
}

@online{GS,
  author = {Gonz{\'a}lez-Sanz, Alberto and Sheng, Shunan},
  title = {Linearization of {Monge--Amp\`ere} equations and statistical applications},
  year = {2024},
  eprint = {2408.06534v4},
  eprinttype = {arXiv},
  eprintclass = {math.AP}
}

@online{IF,
  author = {Gonz{\'a}lez-Sanz, Alberto and Sheng, Shunan and Wu, Bohan and Avella Medina, Marco},
  title = {The influence function of transport-based quantiles},
  year = {2026},
  eprint = {2607.19080},
  eprinttype = {arXiv},
  eprintclass = {math.ST}
}

@article{EssidSolomon.18,
  author = {Essid, Montacer and Solomon, Justin},
  title = {Quadratically regularized optimal transport on graphs},
  journal = {SIAM Journal on Scientific Computing},
  volume = {40},
  number = {4},
  pages = {A1961--A1986},
  year = {2018},
  doi = {10.1137/17M1132665}
}

@article{Muzellec.2017.AAAI,
  author = {Muzellec, Boris and Nock, Richard and Patrini, Giorgio and Nielsen, Frank},
  title = {{Tsallis} regularized optimal transport and ecological inference},
  journal = {Proceedings of the AAAI Conference on Artificial Intelligence},
  volume = {31},
  number = {1},
  pages = {2387--2393},
  year = {2017},
  doi = {10.1609/aaai.v31i1.10854}
}

@online{GGK,
  author = {Gonz\'alez-Sanz, Alberto and Gvalani, Rishabh S. and Koch, Lukas},
  title = {Sharp local sparsity of regularized optimal transport},
  year = {2026},
  eprint = {2604.00843v2},
  eprinttype = {arXiv}
}

@article{Nutz,
 author = {Nutz, Marcel},
 title = {Quadratically regularized optimal transport: existence and multiplicity of potentials},
 journal = {SIAM Journal on Mathematical Analysis},
 volume = {57},
 number = {3},
 pages = {2622--2649},
 year = {2025},
 doi = {10.1137/24M1638136},
 eprint = {2404.06847v3},
 eprinttype = {arXiv}
}

@online{GK,
  author = {Gvalani, Rishabh S. and Koch, Lukas},
  title = {Sparsity and uniform regularity for regularised optimal transport},
  year = {2026},
  eprint = {2601.05130v2},
  eprinttype = {arXiv}
}

@article{GSN,
  author = {Gonz\'alez-Sanz, Alberto and Nutz, Marcel},
  title = {Sparsity of quadratically regularized optimal transport: scalar case},
  journal = {SIAM Journal on Mathematical Analysis},
  volume = {58},
  number = {4},
  pages = {4162--4183},
  year = {2026},
  doi = {10.1137/25M1754480},
  eprint = {2410.03353v2},
  eprinttype = {arXiv}
}

@online{GMS,
 author={Gonz\'alez-Sanz, Alberto and Mordant, Gilles and Sheng, Shunan},
 title={Empirical optimal transport potentials: fast rates and a functional central limit theorem},
 year={2026}, eprint={2608.00649v1}, eprinttype={arXiv}
}

@article{WX,
  author = {Wiesel, Johannes and Xu, Xingyu},
  title = {Sparsity of quadratically regularized optimal transport: bounds on concentration and bias},
  journal = {SIAM Journal on Mathematical Analysis},
  volume = {57},
  number = {6},
  pages = {6498--6521},
  year = {2025},
  doi = {10.1137/25M1723633},
  eprint = {2410.03425v1},
  eprinttype = {arXiv}
}

@inproceedings{BSR,
 author={Blondel, Mathieu and Seguy, Vivien and Rolet, Antoine},
 title={Smooth and sparse optimal transport},
 booktitle={Proceedings of the Twenty-First International Conference on Artificial Intelligence and Statistics},
 series={Proceedings of Machine Learning Research},volume={84},pages={880--889},
 year={2018},publisher={PMLR},url={https://proceedings.mlr.press/v84/blondel18a.html}
}

@article{LMM,
  author = {Lorenz, Dirk A. and Manns, Paul and Meyer, Christian},
  title = {Quadratically regularized optimal transport},
  journal = {Applied Mathematics \& Optimization},
  volume = {83},
  number = {3},
  pages = {1919--1949},
  year = {2021},
  doi = {10.1007/s00245-019-09614-w}
}

@article{EN,
 author={Eckstein, Stephan and Nutz, Marcel},
 title={Convergence rates for regularized optimal transport via quantization},
 journal={Mathematics of Operations Research},volume={49},number={2},pages={1223--1240},
 year={2024},doi={10.1287/moor.2022.0245}
}

@online{GGM,
 author={Garriz-Molina, Alejandro and Gonz\'alez-Sanz, Alberto and Mordant, Gilles},
 title={Infinitesimal behavior of quadratically regularized optimal transport and its relation with the porous medium equation},
 year={2024},eprint={2407.21528},eprinttype={arXiv}
}

@online{CGN,
 author={Cardoso-Perell\'o, Carlos and Gonz\'alez-Sanz, Alberto and Nutz, Marcel},
 title={Sharp asymptotics for regularized optimal transport},
 year={2026},eprint={2607.18191},eprinttype={arXiv}
}

@inproceedings{NCN,
  author = {Nguyen-Chi, Long and Nguyen, Nam and Nguyen, Binh T.},
  title = {Quadratically regularized optimal transport: localization bounds and affine case analysis},
  booktitle = {Proceedings of the Forty-Third International Conference on Machine Learning},
  series = {Proceedings of Machine Learning Research},
  volume = {306},
  publisher = {PMLR},
  year = {2026},
  eprint = {2605.24644v2},
  eprinttype = {arXiv},
  pages = {19438--19458}
}

@online{GNGeometry,
  author = {Gonz\'alez-Sanz, Alberto and Nutz, Marcel},
  title = {Geometry and convergence of quadratically regularized optimal transport {I}},
  year = {2026},
  eprint = {2609.20400v2},
  eprinttype = {arXiv}
}

@article{MSC2015,
  author = {Murugan, Mathav and Saloff-Coste, Laurent},
  title = {{Harnack} inequalities and {Gaussian} estimates for random walks on metric measure spaces},
  journal = {Electronic Journal of Probability},
  volume = {28},
  year = {2023},
  eid = {64},
  pages = {1--81},
  doi = {10.1214/23-EJP954}
}

@article{Vita2022,
  author = {Vita, Stefano},
  title = {Boundary regularity estimates in {H\"older} spaces with variable exponent},
  journaltitle = {Calculus of Variations and Partial Differential Equations},
  volume = {61},
  year = {2022},
  eid = {166},
  doi = {10.1007/s00526-022-02274-9}
}

@article{GNRMonotonicity,
  author = {Gonz{\'a}lez-Sanz, Alberto and Nutz, Marcel and Riveros Valdevenito, Andr{\'e}s},
  title = {Monotonicity in quadratically regularized linear programs},
  journaltitle = {SIAM Journal on Optimization},
  volume = {35},
  number = {2},
  pages = {1419--1437},
  year = {2025},
  doi = {10.1137/24M1685171}
}

@article{BGNLargeDeviations,
  author = {Bernton, Espen and Ghosal, Promit and Nutz, Marcel},
  title = {Entropic optimal transport: geometry and large deviations},
  journaltitle = {Duke Mathematical Journal},
  volume = {171},
  number = {16},
  pages = {3363--3400},
  year = {2022},
  doi = {10.1215/00127094-2022-0035}
}

@online{PNMaps,
  author = {Pooladian, Aram-Alexandre and Niles-Weed, Jonathan},
  title = {Entropic estimation of optimal transport maps},
  year = {2024},
  eprint = {2109.12004v3},
  eprinttype = {arXiv}
}

@online{MordantExpansions,
  author = {Mordant, Gilles},
  title = {The statistics of entropic optimal transport with decreasing regularisation},
  year = {2026},
  eprint = {2412.12007v2},
  eprinttype = {arXiv}
}

@article{AHMPSelf,
  author = {Agarwal, Medha and Harchaoui, Zaid and Mulcahy, Garrett and Pal, Soumik},
  title = {{Langevin} diffusion approximation to same marginal {Schr\"odinger} bridge},
  journaltitle = {Journal of Functional Analysis},
  volume = {291},
  number = {1},
  eid = {111494},
  year = {2026},
  doi = {10.1016/j.jfa.2026.111494}
}

@online{YKSPotentials,
  author = {Yao, Rentian and Kim, Young-Heon and Schiebinger, Geoffrey},
  title = {Quantitative stability of many-marginal {Schr\"odinger} bridge},
  year = {2026},
  eprint = {2604.15191v1},
  eprinttype = {arXiv}
}

@online{ZMMSManifoldQOT,
  author = {Zhang, Stephen and Mordant, Gilles and Matsumoto, Tetsuya and Schiebinger, Geoffrey},
  title = {Manifold learning with sparse regularised optimal transport},
  year = {2023},
  eprint = {2307.09816v2},
  eprinttype = {arXiv}
}

@article{MSCouplingConvergence,
  author = {Malamut, Hugo and Sylvestre, Maxime},
  title = {Convergence rates of the regularized optimal transport: disentangling suboptimality and entropy},
  journal = {SIAM Journal on Mathematical Analysis},
  volume = {57},
  number = {3},
  pages = {2533--2558},
  year = {2025},
  doi = {10.1137/23M1591554}
}

@online{MPEntropicDiffusion,
  author = {Mulcahy, Garrett and Pal, Soumik},
  title = {Diffusion approximations to {Schr\"odinger} bridges and the convergence of entropic potentials},
  year = {2026},
  eprint = {2609.22595v1},
  eprinttype = {arXiv}
}

@article{SGKSemidiscreteStability,
  author = {Sadhu, Ritwik and Goldfeld, Ziv and Kato, Kengo},
  title = {Stability and statistical inference for semidiscrete optimal transport maps},
  journal = {The Annals of Applied Probability},
  volume = {34},
  number = {6},
  pages = {5694--5736},
  year = {2024},
  doi = {10.1214/24-AAP2104}
}

@inproceedings{PDNSemidiscreteMaps,
  author = {Pooladian, Aram-Alexandre and Divol, Vincent and Niles-Weed, Jonathan},
  title = {Minimax estimation of discontinuous optimal transport maps: the semi-discrete case},
  booktitle = {Proceedings of the 40th International Conference on Machine Learning},
  series = {Proceedings of Machine Learning Research},
  volume = {202},
  pages = {28128--28150},
  publisher = {PMLR},
  year = {2023},
  url = {https://proceedings.mlr.press/v202/pooladian23b.html}
}

@article{ANSSemidiscreteAsymptotics,
  author = {Altschuler, Jason M. and Niles-Weed, Jonathan and Stromme, Austin J.},
  title = {Asymptotics for semidiscrete entropic optimal transport},
  journal = {SIAM Journal on Mathematical Analysis},
  volume = {54},
  number = {2},
  pages = {1718--1741},
  year = {2022},
  doi = {10.1137/21M1440165}
}

@article{SGKEntropicApproximation,
  author = {Sadhu, Ritwik and Goldfeld, Ziv and Kato, Kengo},
  title = {Approximation rates of entropic maps in semidiscrete optimal transport},
  journal = {Electronic Communications in Probability},
  volume = {30},
  pages = {1--13},
  year = {2025},
  doi = {10.1214/25-ECP682}
}

@article{GHOLinearization,
  author = {Goldman, Michael and Huesmann, Martin and Otto, Felix},
  title = {Quantitative linearization results for the {Monge--Amp\`ere} equation},
  journal = {Communications on Pure and Applied Mathematics},
  volume = {74},
  number = {12},
  pages = {2483--2560},
  year = {2021},
  doi = {10.1002/cpa.21994}
}

@article{DFStrongSobolev,
  author = {De Philippis, Guido and Figalli, Alessio},
  title = {Second order stability for the {Monge--Amp\`ere} equation and strong {Sobolev} convergence of optimal transport maps},
  journal = {Analysis \& PDE},
  volume = {6},
  number = {4},
  pages = {993--1000},
  year = {2013},
  doi = {10.2140/apde.2013.6.993}
}

@article{GNExactRegularization,
  author = {Gonz{\'a}lez-Sanz, Alberto and Nutz, Marcel},
  title = {Quantitative convergence of quadratically regularized linear programs},
  journal = {Applied Mathematics \& Optimization},
  volume = {91},
  number = {3},
  eid = {68},
  year = {2025},
  doi = {10.1007/s00245-025-10267-1}
}
\end{document}